\documentclass[12pt, reqno]{amsart}
\usepackage[foot]{amsaddr}
\usepackage{amsmath,amsthm,amssymb}
\usepackage{mathtools}
\usepackage{enumerate,enumitem}
\usepackage{graphicx}
\usepackage{color}
\usepackage{bm}
\usepackage{subcaption}

\usepackage{float}
\usepackage[colorlinks=true]{hyperref}
\hypersetup{linkcolor=red,urlcolor=blue,citecolor=blue}

\newcommand{\DIYemail}[1]{E-mail: \texttt{#1}}

\usepackage{etoolbox}
    
\makeatletter
    
\patchcmd{\@maketitle}
{\ifx\@empty\thankses\else \@footnotetext{%
        \def\par{\let\par\@par}\@setthanks}\fi}
{}
{}
{\PackageError{frontmatter}{Cannot remove the original thanks block}{Check the amsaddr version and load this patch after amsaddr.}}

\patchcmd{\@maketitle}
{\let\@makefnmark\relax  \let\@thefnmark\relax}
{\let\@makefnmark\relax  \let\@thefnmark\relax
    \ifx\@empty\thankses\else
    \@footnotetext{\def\par{\let\par\@par}\@setthanks}%
    \fi}
{}
{\PackageError{frontmatter}{Cannot insert the first thanks block}{Check the amsaddr version and load this patch after amsaddr.}}

\patchcmd{\@maketitle}
{\@footnotetext{\@setkeywords}}
{}
{}
{\PackageError{frontmatter}{Cannot remove the keywords footnote}{Check the amsaddr version.}}

\patchcmd{\@maketitle}
{\@footnotetext{\@setsubjclass}}
{}
{}
{\PackageError{frontmatter}{Cannot remove the MSC footnote}{Check the amsaddr version.}}
    
\AtEndEnvironment{abstract}{%
    \ifx\@empty\@keywords\else
    \par\medskip
    \noindent\@setkeywords\par
    \fi
    \ifx\@empty\@subjclass\else
    \par\smallskip
    \noindent\@setsubjclass\par
    \fi
}

\renewcommand{\@setemails}{%
    E-mail addresses:\space{\ttfamily\emails}%
}

\makeatother

\newtheorem{theorem}{Theorem}[section]
\newtheorem{lemma}[theorem]{Lemma}
\newtheorem{proposition}[theorem]{Proposition}
\newtheorem{corollary}[theorem]{Corollary}

\theoremstyle{definition}
\newtheorem{definition}[theorem]{Definition}
\newtheorem{remark}[theorem]{Remark}
\newtheorem{example}[theorem]{Example}

\numberwithin{equation}{section}

\newcounter{ProblemsCounter}
\newcounter{GlobalProblemsCounter}

\title{Geometric and functional mixing by 2D stationary incompressible flows}

\author[Weiwei Hu, Ziqian Li and Yubiao Zhang]{%
    Weiwei Hu$^{1}$,
    Ziqian Li$^{2}$ and
    Yubiao Zhang$^{3,2,*}$%
}

\thanks{$^{*}$Corresponding author: Yubiao Zhang.
    \DIYemail{yubiao.zhang.math@outlook.com}%
}

\address{%
    $^{1}$Department of Mathematics,
    University of Georgia,
    Athens, GA 30602, USA
}
\address{%
    $^{2}$Chair for Dynamics, Control, Machine Learning and Numerics,
    Department of Mathematics,
    Friedrich-Alexander-Universit\"{a}t Erlangen-N\"{u}rnberg,
    91058 Erlangen, Germany
}
\address{%
    $^{3}$School of Mathematics,
    Jilin University,
    Changchun, Jilin 130012, China
}

\email{Weiwei.Hu@uga.edu}
\email{ziqian.li@fau.de}

    

\date{}
        
\begin{document}

\begin{abstract}
    We study quantitative mixing and deformation of sets and curves for a class
    of two-dimensional autonomous Hamiltonian flows with finitely
    many critical points satisfying local conditions that allow finite-order degeneracy.
    Variation of the period across neighboring trajectories generates
    transverse shear, providing a common mechanism for scalar mixing,
    set deformation, and curve stretching.
    First, for $H^1$ initial data supported away from equilibria and
    infinite-period trajectories, in regions where the period gradient
    is bounded away from zero, we establish sharp $(1+t)^{-1}$ decay
    in $H^{-1}$ towards the time average of the initial data along each
    periodic trajectory.
    Second, under the same geometric conditions, we prove matching
    upper and lower bounds of order $(1+t)^{-1}$ for an orbit-relative
    geometric mixing scale of transported Lipschitz subdomains whose
    closures are not invariant under the flow.
    This scale measures how closely the transported subdomain
    covers the union of trajectories meeting its initial position.
    Third, for Lipschitz curves separated from infinite-period
    trajectories, we derive an explicit first-order large-time
    expansion of their length with a remainder bounded uniformly
    in time. In particular, their length grows at most linearly.
    
    Counterexamples illustrate how the stated conclusions can fail
    when selected nondegeneracy or separation assumptions are removed.
    The analysis combines coordinates adapted to the periodic
    trajectories with quantitative estimates and asymptotic
    expansions for the flow Jacobian.
    Numerical simulations for cellular and radial flows illustrate
    the functional and geometric mixing rates and the evolution
    of curve length.
\end{abstract}

\maketitle

\section{Introduction}

\subsection{Problem setting and motivation}

Incompressible advection can transform an initially heterogeneous
scalar distribution into increasingly fine spatial structures.
We study this process for two-dimensional flows generated by
time-independent velocity fields. The basic model is the
transport equation
\begin{align}\label{20240925-yubiao-MainEquation}
    \partial_t \theta + V \cdot \nabla_x \theta = 0
    \quad \text{in } (0,+\infty)\times\Omega;
    \qquad
    \theta|_{t=0}=\theta_0
    \quad \text{in }\Omega.
\end{align}
Here, $\Omega\subset\mathbb R^2$ is a bounded Lipschitz domain,
$\theta$ is the transported scalar, $\theta_0$ is its initial
distribution, and $V$ is a prescribed time-independent
divergence-free velocity field tangent to the boundary.
The precise  assumptions on this field are stated
in the next subsection.

Since the flow preserves Lebesgue measure, pure transport
preserves the $L^2$ norm of the scalar. Thus, in the absence
of diffusion, mixing is understood through weak convergence
or through the geometry of increasingly fine spatial structures,
rather than through decay of the $L^2$ norm.
Negative Sobolev norms provide a quantitative way to measure
this weak homogenization, while transported sets and curves
describe its geometric aspects. Related works concerning these aspects are introduced in Subsection \ref{sec:RelatedWorks}. 

The autonomous Hamiltonian setting introduces an additional
constraint: trajectories remain on invariant level sets of
the Hamiltonian. In a region of periodic trajectories, the
time average of the initial scalar over each orbit is therefore
preserved. Denote this orbit average by $\bar\theta_0$, as
defined in \eqref{third:orbit-average}.
It is the natural candidate for the limiting profile, which
need not coincide with the global spatial average.
However, periodicity of the trajectories alone does not
guarantee convergence towards this profile.
For example, rigid rotation transports a nonconstant angular
pattern periodically without producing progressively finer
scales.

The mechanism considered here is the variation of the period
across neighboring trajectories. When their rotation
frequencies differ, neighboring trajectories accumulate
different phases over time. This differential rotation
produces transverse shear, which stretches material structures
across trajectories and generates fine spatial scales.
We ask how this mechanism determines three complementary
observables of mixing and material deformation.

First, we study convergence of a solution to equation \eqref{20240925-yubiao-MainEquation} towards its orbit average:
\[
\theta(t,\cdot)-\bar\theta_0
\rightharpoonup 0
\quad\text{in }L^2(\Omega)
\quad\text{as }t\to+\infty.
\]
We seek conditions ensuring this convergence and sharp
quantitative estimates for
$\|\theta(t,\cdot)-\bar\theta_0\|_{H^{-1}(\Omega)}$.
Subtracting the orbit average removes the invariant component,
so that the estimate measures only the part of the scalar
affected by differential rotation.

Second, we study the geometry of a transported subdomain.
If a material initially occupies an open set $A$, its
dynamically accessible region is the union of all trajectories
meeting $A$, denoted by $Orbit(A)$.
We ask how closely the transported set covers this region
at a given time. The geometric scale introduced below
quantifies this spatial coverage relative to $Orbit(A)$.
Its density parameter is not fixed in advance, so this
quantity differs from geometric mixing scales defined
at a prescribed volume fraction.

Third, we study the length of a transported Lipschitz curve,
which may represent a material interface.
We seek a first-order large-time formula for its length
and conditions distinguishing bounded length from growth
of linear order. Curve stretching records 
deformation of sets, but does not by itself imply the mixing property
of the scalar or uniform interpenetration of two phases.

Our aim is to relate these three observables through the
shear generated by period variation, while keeping their
distinct meanings and hypotheses explicit.
We also examine how the conclusions can fail when the
relevant nondegeneracy or separation assumptions are removed,
including situations involving vanishing period variation,
equilibria, and infinite-period trajectories.
The following subsection states the precise assumptions
and the quantitative results for each observable.

\vskip 10pt
\paragraph{\textbf{Notation.} } The following notations will be frequently used throughout this paper.
Let $\mathbb{R}^+ := (0,+\infty)$ and $\mathbb{N}^+ := \{1,2,3,\ldots\}$. For a set $E \subset \mathbb{R}^2$,  denote by $\overline{E}$ or $\mathrm{cl}\,E$ its closure, by $\mathrm{Int}\,E$ its interior, and by $|E|$ its Lebesgue measure. 
Write $d(\cdot,\cdot)$ for the distance between two points, between a point and a set, or between two sets; by convention, the distance is $+\infty$ if one of the sets involved is empty.
Denote by $\langle \cdot, \cdot \rangle$ and $|\cdot|$ the standard inner product and the corresponding Euclidean norm in $\mathbb{R}^2$, respectively. 
Denote by $B_r(x)$ the open ball in $\mathbb{R}^2$ centered at $x$ with radius $r$. 
Write $\vec v^{\perp}$ for the vector obtained by rotating $\vec v\in\mathbb{R}^2$ anticlockwise by $\pi/2$. Furthermore, write
\begin{align*}
    \nabla f := (\partial_1 f, \partial_2 f)^\top   
    \text{ and }~
    \vec{u} \otimes \vec{v} := \vec{u} \vec{v}^{\top},
    ~\forall\,\vec{u},\vec{v} \in \mathbb R^2. 
\end{align*}
The notation $C(\ldots)$ denotes a constant depending only on the parameters specified in the parentheses. 
Finally, for a Lipschitz curve $\gamma$,  denote by $|\gamma|$ its length.

\subsection{Main results}

To state the three main results, we need to  specify the class of stationary Hamiltonian flows. Let $\Omega \subset \mathbb{R}^2$ be a bounded Lipschitz domain.
Throughout the paper, we consider a divergence-free time-independent velocity
field $V \in C^1( \overline{\Omega}; \mathbb R^2)$ (i.e., $V\in C^1(U; \mathbb R^2)$ for some open set $U \supset \overline{\Omega}$) tangent to the boundary:
\begin{align*}
    \operatorname{div}V=0 ~\text{in }\Omega
    ~\text{ and }~
    V\cdot\vec n=0 ~\text{on }\partial\Omega,
\end{align*}
where $\vec n$ is the outward unit normal.  The transport therefore
preserves Lebesgue measure and the $L^2$ norm of the scalar.
We impose the following two assumptions:
\begin{itemize}
    \item[\textbf{(A1)}] There exists a function $H\in C^2(\mathbb R^2)$ such that
    \begin{align}\label{ham:representation}
        V=-\nabla^\perp H:=(\partial_2H,-\partial_1H)^\top
        ~\text{ in }\overline\Omega
        ~\text{ and }~
        H|_{\partial \Omega} = 0,
    \end{align}
    
    \item[\textbf{(A2)}] 
    The function $H$ has finitely many critical points $\{x_1^*,\ldots,x_K^*\} $  on $\overline{\Omega}$ such that for each critical point $x_k^*$, there exists a real number $m_k\ge1$ and a symmetric nonsingular matrix $S_k\in\mathbb R^{2\times2}$ such that
    \begin{align}\label{ham:local-model}
        H(x_k^*+y)-H(x_k^*)
        &=\frac12|y|^{m_k-1}y^\top S_k y+R_k(y)
        ~\text{ for $y$ near } 0,
    \end{align}
    where $R_k\in C^2(\mathbb R^2)$ satisfies
    \begin{align}\label{ham:remainder}
        \|D^jR_k(y)\|=o\!\left(|y|^{m_k+1-j}\right)
        ~\text{ as } y\to0,
        ~~ j=0,1,2.
    \end{align}
    Here, the norms for $j=1,2$ are the Euclidean and operator norms, respectively.
\end{itemize}
Assumption (A1) is standard for stationary incompressible flows. 
Assumption (A2) prescribes a finite-order local model at each critical point, with the remainder controlled through its second derivatives. The exponents $m_k$ need not be integers: the homogeneous leading term
$y\mapsto\frac12|y|^{m_k-1}y^\top S_k y$ is of class $C^2$. The case $m_k=1$ includes Hamiltonians with finitely many nondegenerate critical points; $m_k>1$ permits degenerate centers and saddles. In particular, this class contains the cellular flow and the radial Hamiltonians $H(x)=|x|^{m+1}-1$ with $m\ge1$. The local velocity and period estimates required below are proved in Section~\ref{section-OrbitAndPeriod}.

Now we specialize \eqref{20240925-yubiao-MainEquation} to a velocity field $V$
satisfying assumptions (A1)--(A2).  No boundary condition is required for
$\theta$, since $V$ is tangent to $\partial\Omega$. 
The mixing process is
described by the transport equation \eqref{20240925-yubiao-MainEquation}.
Write $\theta(\cdot; \theta_0)$ for the solution to equation \eqref{20240925-yubiao-MainEquation}. The characteristic equation associated with \eqref{20240925-yubiao-MainEquation} is the following ODE: 
\begin{align}\label{20240925-yubiao-Flow}
    x'(t)=V(x(t)), ~ t\in\mathbb{R};  ~~ x(0)=x_0,
\end{align}
whose solution with initial point $x_0$ is denoted by
$x(\cdot;x_0)$. 
The associated $C^1$ flow map $\{\varPhi(t)\}_{t\in\mathbb R}$ is defined by
\begin{align}\label{20240926-yubiao-DefinitionOfFlow}
    \varPhi(t)(x_0):=x(t;x_0), ~ x_0\in\overline{\Omega}.
\end{align}
The following definition related to \eqref{20240925-yubiao-Flow} will be frequently used throughout this paper. 

\begin{definition}\label{20250110-yb-definiton-InvariantSet} 
   The following function is called the \emph{orbit period function}:
    \begin{align}\label{20241021-yb-PeriodOfOrbits}
        T(x_0) := \inf \Big\{
        \hat t >0 ~:~ x(\hat t; x_0) = x_0
        \Big\} \in [0,+\infty],
        ~ x_0 \in \overline{\Omega}.
    \end{align}
    By convention, $T(x_0)=+\infty$ if the set in the definition of the infimum is empty.
\end{definition}
It deserves to mention the following properties of the orbit period function $T(\cdot)$ (see Lemma \ref{20250322-yb-propsotion-PropertiesOfOrbitPeriod}): (i)  it is of class $C^1$ in the open set $\{0<T<+\infty\}$, (ii) it takes the value $0$ only at the critical points of $H$, and (iii) it takes the value $+\infty$ on at most finitely many curves.

\medskip

We now state the three main results about the functional and geometric mixing, in order to understand the mixing properties from various aspects. 

Our first main result concerns the functional mixing scale of  solutions to equation \eqref{20240925-yubiao-MainEquation}.  
For a function $f \in L^2(\Omega)$, we define its orbit average by 
\begin{align}\label{third:orbit-average}
    \bar f(x) := \frac1{T(x)}\int_0^{T(x)}  f\big( \varPhi(s)(x) \big)\,ds
    ~\text{ for a.e. }~   x \in \Omega.
\end{align}
See Lemma~\ref{20260902-RegularityOfProjection} for more details. 
The first theorem gives the sharp decay rate of the difference between the
solution and the orbit average of its initial data in $H^{-1}(\Omega)$. The initial data 
is supported in a region separated from stationary points and nonperiodic
trajectories, where the gradient of the orbit period is bounded away from
zero. Numerical illustrations are presented in
Subsection~\ref{subsection-MixNorm}.

\begin{theorem}
    \label{third:optimal-negative-norm}
    Suppose that $V$ satisfies assumptions (A1)--(A2) and that
    $V\in C^2(\overline\Omega)$.  
    Let $A\subset\Omega$ be a subdomain such that
    \begin{align}
        d\big(A,\{T=0, +\infty\}\big)>0
        ~\text{ and }~
        \inf_{ x \in A}  | \nabla T(x) | > 0.
        \label{20260902-domain}
    \end{align}
    Then there exist constants $c,C>0$ such that for each $\theta_0 \in H^1(\Omega)$ with supp\,$\theta_0 \subset A$, 
    \begin{align}
        \frac{c}{1+t} \| \theta_0  - \bar{\theta}_0 \|_{ H^{-1}(\Omega)}  
        \le\|\theta(t; \theta_0) - \bar{\theta}_0 \|_{H^{-1}(\Omega)}
        \le\frac{C}{1+t}  \| \theta_0  - \bar{\theta}_0 \|_{ H^{1}(\Omega)}  ,
        ~~ t\ge0,
        \label{third:sharp-rate}
    \end{align}
    where $\theta$ solves equation 
    \eqref{20240925-yubiao-MainEquation} with the initial data $\theta_0$ and $\bar\theta_0$ is given by \eqref{third:orbit-average}.      
\end{theorem}

\begin{remark}
\begin{itemize}  
    \item[(i)]  The main reason for \eqref{third:sharp-rate} is the following:  variation
    of the orbit period supplies the transverse shear between neighboring
    periodic trajectories (related to the second inequality in assumption \eqref{20260902-domain}). This can be seen from the proof of Theorem \ref{third:optimal-negative-norm}. 
    
    \item[(ii)] Each inequality in assumption \eqref{20260902-domain} is sharp due to the counterexamples in Examples \ref{example:VanishingPeriodGradient-MixingNorm-20260910}--\ref{Counterexample-InfinitePeriod:MixingNorm} for some specific velocity fields.

\end{itemize}
\end{remark}

The same transverse shear also controls how rapidly a transported subdomain
covers the periodic region accessible from its initial position.  This is studied in our second main theorem. To state it, we need the following definition. 

\begin{definition}\label{20250103-yb-DefinitionOfMixingScale}
    \begin{itemize}
        \item[(i)] For $E\subset\overline\Omega$, we define its orbit by
        \begin{align}
            Orbit(E):=\bigcup_{t\in\mathbb R}\varPhi(t)(E).
            \label{20250103-yb-ErgodicSet}
        \end{align}
        For a singleton $\{x\}$, we rewrite $Orbit(\{x\})$ as $Orbit(x)$ for simplicity.  A set
        $E\subset\Omega$ is called $V$-invariant if $Orbit(E)=E$. 
        
        \item[(ii)] 
        Let $E,F\subset\mathbb R^2$ be nonempty open sets such that
        $E\subset F$.
        We say that the set $E$ is \emph{mixed to scale $\varepsilon>0$ over $F$} if there exists $\kappa > 0$ such that
        \begin{align}\label{20250310-yb-DensityForMixing}
            |B_{\varepsilon}(x)\cap E|
            \ge \kappa\,|B_{\varepsilon}(x)|
            ~ \text{for each } x\in F.
        \end{align}
        
        \item[(iii)] 
        Let $A\subset\Omega$ be a nonempty open set. Let
        $t\ge0$ and  $\varepsilon>0$.
        We say that the vector field $V$ \emph{mixes the set $A$ to scale $\varepsilon$ at time $t$} if the set $\varPhi(t)(A)$ is mixed to scale $\varepsilon$ over its orbit $Orbit(A)$ defined in \eqref{20250103-yb-ErgodicSet}. 
        The infimum of such $\varepsilon$ is called the \emph{mixing scale} of $A$ under the action of $V$ at time $t$, and is denoted by $MixingScale(t,A)$:
        \begin{align}\label{20250103-yb-ExactDefinitionOfMixingScale}
            MixingScale(t,A)
            :=
            \inf \Big\{
            \hat\varepsilon>0:
            V \text{ mixes } A \text{ to scale } \hat\varepsilon \text{ at time } t
            \Big\}.
        \end{align}
        
    \end{itemize}
\end{definition}

The use of $Orbit(A)$ reflects the dynamically accessible region of a
stationary flow.  Definition~\ref{20250103-yb-DefinitionOfMixingScale} is inspired by, but differs from, fixed-volume-fraction geometric mixing scales, such as those in
\cite{Bressan-2006,yao2017mixing}, because the constant $\kappa$ is not fixed in advance.  

The second main result gives the sharp $(1+t)^{-1}$ rate for the mixing
scale of a Lipschitz subdomain whose closure is not invariant under the flow.
The subdomain satisfies the same separation and period-variation conditions
as in Theorem \ref{third:optimal-negative-norm}. Numerical illustrations are presented in 
Subsection~\ref{subsection-SecondTheorem}.

\begin{theorem}\label{20250103-yb-theorem-MixingScale}
  Assume that $V$ satisfies (A1)--(A2). Let $A\subset\Omega$ be a  Lipschitz subdomain whose closure  is not $V$-invariant. Assume that
  \begin{align}
      d\big(A,\{T=0, +\infty\}\big)>0
      ~\text{ and }~
      \inf_{x \in A}  | \nabla T(x)| > 0,
      \label{rev:regular}
  \end{align}
  where  the orbit-period function $T$ is given by \eqref{20241021-yb-PeriodOfOrbits}. 
  Then, there exist two constants $C_1,C_2>0$ such that
  \begin{align}
      C_1(1+t)^{-1}
      \leq  MixingScale(t,A) \leq
      C_2(1+t)^{-1}, ~t\geq0,
      \label{20250112-yb-OrderOfMixingScale}
  \end{align}
  where  $MixingScale(t,A)$ is defined in  \eqref{20250103-yb-ExactDefinitionOfMixingScale}. 
\end{theorem}

\begin{remark}
\begin{itemize}
    \item[(i)] 
    The motivation for studying $MixingScale(t,A)$ arises from the following typical scenario: a certain material is initially distributed over a subdomain, and one seeks to understand to what extent this material can approach distant spatial locations under the action of a flow as time evolves. This naturally leads to the study of the smallest spatial scales to which a subdomain can be mixed by a stationary divergence-free vector field $V$.
    
    \item[(ii)] 
    Theorem \ref{20250103-yb-theorem-MixingScale} shows that  the flow generated by $V$ mixes the set $A$ to arbitrarily small spatial scales for sufficiently large time. Moreover, the optimal mixing scale decays at the rate $O(t^{-1})$.

     \item[(iii)] 
    The deformation of the orbit-relative complement $Orbit(A)\setminus A$ plays a central role in establishing the estimate \eqref{20250112-yb-OrderOfMixingScale}. As time evolves, this set becomes increasingly elongated and thin under the action of the flow (see Proposition \ref{rev:prop52} and Remark \ref{20250404-remark-MechanismForDeformationOfSet} as well as Figure \ref{fig:r2-onecircle-evolution} in Subsection~\ref{subsection-SecondTheorem}). The main mechanism behind this deformation is as follows: variation of the orbit period supplies the transverse shear between neighboring
    periodic trajectories, causing the anisotropic growth of the singular values of the Jacobian matrices of the flow $\{\varPhi(t)\}_{t\ge 0}$, as illustrated in Theorem \ref{20250326-yb-theorem-SingularValuesOfJacobian} and Remark \ref{remark-SingularValues}. This stretching--compression effect produces increasingly thin filamentary structures and determines the decay rate of the mixing scale.

     \item[(iv)] 
    The requirement that the closure of $A$ is not $V$-invariant is necessary to exclude trivial cases. If $\overline{A}$ were invariant under the action of  $V$, then it would remain unchanged in time. In this case the mixing scale of $A$ could always be taken to be zero (see Lemma \ref{lemma-20260122-FineMixingForInvariantSet}), making the problem trivial.
    
    \item[(v)] 
    Assumption \eqref{rev:regular} is sharp due to the counterexamples in  Subsection \ref{subsec:assumption-SecondMainTheorem}. 
    In particular, the second inequality in assumption \eqref{rev:regular}  prevents a dynamical obstruction to the formation of small mixing scales. Points with nearly identical orbit periods evolve in a synchronized manner under the flow generated by $V$, which tends to preserve their relative spatial configuration. Such coherent motion inhibits the formation of sufficiently fine structures. Requiring the gradient of the orbit period function to be bounded away from zero ensures sufficient shearing between nearby trajectories, which in turn promotes mixing.

    \item[(vi)] 
    Two consequences of Theorem \ref{20250103-yb-theorem-MixingScale} are presented in Section \ref{20250313-yb-subsection-MixingScale-ProofsOfCorollaries}. The first provides a sharp analogue of Bressan's conjecture \cite[p.\,101]{Bressan-2006} for stationary divergence-free vector fields (see Corollary \ref{20250110-yb-corollary-AnalogOfBressanConjecture}). The second concerns the rate at which two disjoint evolving open sets that share the same orbit approach each other at large time (see Corollary \ref{20250110-yb-corollary-DistanceBetweenTwoSets}).
   \end{itemize}
\end{remark}

At the differential level, the same shear appears in the flow Jacobian and
governs the stretching of transported curves.  For a Lipschitz curve
$\gamma_0:[0,1]\to\overline\Omega$, set
\begin{align*}
    \gamma_t(\alpha):=\varPhi(t)(\gamma_0(\alpha)),
    ~\alpha\in[0,1]
    ~\text{ for each } t\ge0.
\end{align*}
Our third main result gives a first-order expansion of the
transported length $|\gamma_t|$.  In particular, $|\gamma_t|=O(1+t)$; the expansion
characterizes exactly whether $|\gamma_t|$ remains bounded or has linear
order.  Numerical
illustrations are presented in Subsection~\ref{subsection-FirstTheorem}.

\begin{theorem}\label{20241021-yb-theorem-OptimalGrowthForCurves}
    Suppose that the vector field $V$ satisfies assumptions (A1)--(A2).
    Let $\gamma_0:[0,1]\to\overline{\Omega}$ be a Lipschitz curve satisfying
    \begin{align}\label{20250402-yb-FintiePeriodAssumptionOnCurve}
        D := d\big(\gamma_0,\{T=+\infty\}\big) > 0,
    \end{align}
    where $T$ denotes the orbit period function defined in \eqref{20241021-yb-PeriodOfOrbits}.   
    Define
    \[
    F(\alpha,t)
    := \begin{cases}
        |\nabla_{\gamma_0'(\alpha)}\ln T(\gamma_0(\alpha))| 
        \cdot 
        \,|V(\gamma_t(\alpha))|,   ~& 0 < T(\gamma_0(\alpha) ) < +\infty,
        \vspace{0.5em} \\
        0,  &\text{otherwise}.
    \end{cases}
    \]
        Set $D_*:=\min\{1,D\}>0$.  Then there exists a constant     $C=C(\Omega,V,D_*)>0$
    such that
    \begin{align}\label{20241023-yb-AsymptoticForLengthOfCurves}
        \sup_{t\ge0}
        \left|
        |\gamma_t|
        -
        t \int_0^1 F(\alpha,t)\,d\alpha
        \right|
        \le C|\gamma_0|.
    \end{align}
    Moreover, 
    \begin{align}\label{20241023-yb-AuxsiliaryEstimateForLengthOfCurves}
      \sup_{t \geq 0}  \int_0^1 F(\alpha,t)\,d\alpha
        \leq C |\gamma_0 |. 
    \end{align}   
\end{theorem}

\begin{remark}\label{20250402-yb-remark-FirstMainResult}
    \begin{itemize}
        \item[(i)] The motivation for studying $|\gamma_t|$ arises from the following typical scenario. 
        Consider a domain occupied by two distinct materials, which are often represented by binary values (e.g., $\pm1$). 
        The discontinuity set between these labels defines the interface separating the materials. 
        As the materials are transported and mixed by a flow, it is natural to expect that the length of this interface increases over time. 
        
        \item[(ii)] The asymptotic expansion \eqref{20241023-yb-AsymptoticForLengthOfCurves} shows that the linear growth of $|\gamma_t|$ is driven by the derivatives of the orbit-period function $T(\cdot)$ along the initial curve $\gamma_0$. 
        In particular, $|\gamma_t|$ remains bounded as $t\to+\infty$ if and only if the integral in \eqref{20241023-yb-AsymptoticForLengthOfCurves} vanishes. Indeed, the uniform same-orbit speed ratio in Proposition~\ref{20250322-yb-proposition-UsefulPropertiesOfOrbitPeriod} makes this integral uniformly comparable to its value at $t=0$.
        This occurs precisely in the situation: for almost every $\alpha\in[0,1]$, either $T(\gamma_0(\alpha))=0$ or $\nabla_{\gamma_0'(\alpha)}T(\gamma_0(\alpha))=0$. 
        
        
        In particular, linear growth requires nonzero variation of the period along the initial curve on a set of positive parameter measure.
        
        \item[(iii)] The main idea behind \eqref{20241023-yb-AsymptoticForLengthOfCurves} is to analyze the asymptotic expansion of the Jacobian matrices of the flow $\{ \varPhi(t) \}_{t\ge0}$ (see Theorem~\ref{20241023-yb-proposition-AsymptoticForJacobianOfVeolocityField}). 
        This analysis is closely related to the growth behavior of the singular values of these matrices (see Theorem~\ref{20250326-yb-theorem-SingularValuesOfJacobian} and Remark~\ref{remark-SingularValues}). 
        The behavior of these singular values may have further applications in the study of mixing phenomena.
        
        \item[(iv)] Assumption \eqref{20250402-yb-FintiePeriodAssumptionOnCurve} on the curve $\gamma_0$ is sharp due to the counterexample in Example \ref{Counterexample-InfinitePeriod:Length}. 
        Different growth rates can arise when condition \eqref{20250402-yb-FintiePeriodAssumptionOnCurve} fails, which would no longer be consistent with the rate obtained in Theorem~\ref{20241021-yb-theorem-OptimalGrowthForCurves}. 
        Understanding this situation remains an interesting topic for future investigation.
        
        \item[(v)] Theorem~\ref{20241021-yb-theorem-OptimalGrowthForCurves} also reveals an important dynamical feature of the ODE \eqref{20240925-yubiao-Flow}: small perturbations of the initial data can lead to linear deviations of trajectories at large times.
    \end{itemize}
\end{remark}

\subsection{Novelties and methods}

We study three quantitative measures of mixing 
for two-dimensional
stationary incompressible flows generated by Hamiltonians with finitely many
critical points, each of finite order. 
The main novelties of this work are summarized as follows:
\begin{itemize}
    \item[(a1)] We establish the sharp decay rate $(1+t)^{-1}$ in
    $H^{-1}$ towards the orbit average for solutions of the transport equation \eqref{20240925-yubiao-MainEquation} (with a stationary incompressible velocity field), in a regular region away from stable equilibria and infinite-period orbits. See Theorem \ref{third:optimal-negative-norm}.
    
      \item[(a2)] The mixing scale of a noninvariant transported Lipschitz subdomain (in the same regular region) is shown to be     comparable to $(1+t)^{-1}$ (see Theorem \ref{20250103-yb-theorem-MixingScale}). In addition, sharp two-sided estimates for the inradius and
      orbit-relative inradius of open sets satisfying a uniform-orbit cone
      condition are also established (see Proposition \ref{rev:prop52}).  
      
    \item[(a3)] We derive an explicit first-order formula for the length of transported Lipschitz curves (away from infinite-period orbits) and show that this length grows at most linearly in time. See Theorem \ref{20241021-yb-theorem-OptimalGrowthForCurves}.

    \item[(a4)] We identify the common mechanism behind the three measures in (a1)--(a3): variation
    of the orbit period supplies the transverse shear between neighboring
    periodic trajectories. Furthermore, this causes the anisotropic growth of singular values of the flow  Jacobian at the rates of $1+t$ and $(1+t)^{-1}$, respectively (see Theorem  \ref{20250326-yb-theorem-SingularValuesOfJacobian}), which produces the stretching--compression effect. This geometric phenomenon is responsible for the deformation of open sets and curves. 
    
    Although the shared shear mechanism links the aforementioned three measures, their precise estimates require separate geometric and analytic arguments.
    
    \item[(a5)]     We derive a uniform first-order expansion of
    $J_{\varPhi(t)}$ at non-equilibrium points away from the
    infinite-period orbits, with a remainder uniform as such points approach stable
    equilibria (see Theorem \ref{20241023-yb-proposition-AsymptoticForJacobianOfVeolocityField}).  
\end{itemize}
Furthermore, we present counterexamples concerning the sharpness of the assumptions (i.e., separation from stable equilibria, infinite-period orbits and regions of the zero gradient of orbit periods) in
the three main theorems to show that the conclusions in
the corresponding theorems fail when the indicated hypothesis is
removed (while the other hypotheses listed there are retained). See Section \ref{section-ExamplesOnAssumptions}.

We next describe the main methods used to prove Theorems \ref{third:optimal-negative-norm}, \ref{20250103-yb-theorem-MixingScale}, and \ref{20241021-yb-theorem-OptimalGrowthForCurves}:
\begin{itemize}
    \item[(b1)] The proof of Theorem \ref{third:optimal-negative-norm} relies mainly on a particular coordinate system (see Lemma \ref{rev:uniform-action-angle}), where the flow is represented in terms of rotations with nonzero angular velocities provided by the variation of orbit periods. Applying the classical Fourier series method in the angular variable introduces a non-stationary phase, allowing direct estimates of the solutions using tools from harmonic analysis. 
    
    \item[(b2)] In the proofs of Theorems \ref{20250103-yb-theorem-MixingScale} and \ref{20241021-yb-theorem-OptimalGrowthForCurves}, we first derive a first-order expansion of
    $J_{\varPhi(t)}$ at non-equilibrium points away from the
    infinite-period set. Then, for Theorem \ref{20241021-yb-theorem-OptimalGrowthForCurves}, this expansion is sufficient to establish the growth of the length of a curve.  
    For Theorem \ref{20250103-yb-theorem-MixingScale}, we use the aforementioned expansion, as well as the aforementioned coordinate system, to study the deformation of the orbit-relative complement of  a target domain, which reflects the mixing scale of this domain. 
\end{itemize}

\subsection{Related works}
\label{sec:RelatedWorks}

Mixing problems arise in many applications, including atmospheric frontogenesis \cite{doswell1984kinematic}, solute transport in microfluidic channels 
\cite{abraham2002chaotic}, chemical engineering \cite{manu2016how}, and microfluidic applications in biology
\cite{beebe2002biology}. For incompressible transport without diffusion,
the $L^2$ norm of a passive scalar is conserved. Consequently,
mixing can be quantified through weaker norms or through
geometric properties of transported sets and interfaces. General
accounts of quantitative mixing and its interaction with diffusion
can be found in \cite{thiffeault2012multiscale,zelati2024mixing}.

Beyond the aforementioned applications, mixing also plays a fundamental role in the mathematical theory of partial differential equations and dynamical systems. 
In particular, the study of mixing is closely connected with ergodic theory and the long-time behavior of dynamical systems, where it provides a framework for understanding how stretching and folding mechanisms lead to the mixing of scalar quantities.

An important distinction is between designing velocity fields that
mix efficiently under prescribed constraints and determining the
mixing properties of a given flow. The first direction includes
constructions of efficient and universal mixers
\cite{alberti2016exponential,yao2017mixing,elgindi2019universal}
and optimization of stirring protocols
(e.g.~\cite{mathew2007optimal, liu2008mixing, foures2014optimal,  d1999control, zheng2023numerical, hu2023feedback, hu2018boundaryNS}).
The present work belongs primarily to the second direction:
we study a two-dimensional autonomous Hamiltonian flow and relate its mixing and deformation properties to the variation of the periods of its trajectories.

We organize the discussion around three complementary observables
of mixing and material deformation: (i) growth of interfaces;
(ii) geometric mixing scales; and (iii) functional mixing scales.
We briefly review the literature on each below.


\textit{(A) Growth of interfaces.}
The length of a material interface is a natural measure of the
deformation produced by an incompressible flow. 
    A typical setting considers a domain occupied by two distinct materials, which are often represented by binary values (e.g., $\pm 1$). 
    The discontinuity set between these labels defines the interface. 
    As the materials are advected by a flow and undergo mixing, it is natural to expect that the interface length increases over time due to stretching and folding mechanisms. 
    This observation motivates the study of the growth of interfaces. 
It has been used
both as a diagnostic of mixing
\cite{chakravarthy1996mixing,vikhansky2002enhancement}
and as an objective in the optimization of stirring.
Li and Zuazua \cite{li2026hamiltonian}, for example, formulate
interface-length optimization through a reduced Hamiltonian
control problem.


\textit{(B) Geometric mixing scales.}
Geometric mixing scales quantify local interpenetration through
averages over small spatial neighborhoods. Bressan's formulation
\cite[p. 101]{Bressan-2006} and subsequent work, including
Yao and Zlato\v{s} \cite{yao2017mixing}, employ a prescribed
accuracy or volume-fraction parameter.
The constructions of Alberti, Crippa, and Mazzucato
\cite{alberti2016exponential} illustrate how attainable mixing
rates depend on regularity constraints and on the construction
of the flow. Their self-similar scaling analysis yields
finite-time perfect mixing below the critical Sobolev index $s=1$,
exponential mixing at the critical index, and polynomial mixing
above it, under the corresponding assumptions.
The polynomial rate in the last regime is a property of this
construction, rather than a general obstruction to faster mixing.

For autonomous flows, Crippa, Luc\`a, and Schulze
\cite{crippa2019polynomial} study a stationary radial Euler flow
on the disk. For continuous initial data with zero average on
almost every circular trajectory, they establish a
$t^{-1}$ upper bound for the geometric mixing scale.
In a substantially less regular setting, Bonicatto and Marconi
\cite{bonicatto2021regularity} prove Lusin--Lipschitz estimates
for planar autonomous divergence-free BV flows in their setting, with constants
growing at most linearly in time. They deduce lower bounds
of order $(1+t)^{-1}$ for both geometric and analytical mixing
in their setting.

Our geometric quantity (introduced in \eqref{20250103-yb-ExactDefinitionOfMixingScale}) is inspired by these notions but differs
from the fixed-volume-fraction scales used in those works.
It is measured relative to the orbit saturation $Orbit(A)$,
and the density parameter  is not prescribed in advance.
The matching bounds proved here therefore concern this
orbit-relative notion and should be distinguished from estimates
for geometric mixing at a fixed accuracy.

\textit{(C) Functional mixing scales.}
Negative Sobolev norms provide a quantitative way to measure weak
homogenization. The multiscale approach of Mathew, Mezi\'c, and
Petzold \cite{mathew2005multiscale} and the use of the
$H^{-1}$ norm in stirring optimization
\cite{lin2011optimal} have motivated extensive analytical
and numerical work; see also \cite{thiffeault2012multiscale}.
Lower bounds under constraints on the velocity gradient were
developed in
\cite{crippa2008estimates,seis2013maximal,iyer2014lower}.
The relationship between geometric and functional mixing scales
requires care: the two notions are not equivalent in general,
and their comparability, including the role of large-scale
components, is studied in \cite{zillinger2019scales}.

For an autonomous Hamiltonian flow with periodic trajectories,
the average along each trajectory is preserved. Thus the relevant
decaying quantity is the difference between the solution and its
orbit average. This distinction already appears in the radial Euler flow in  \cite{crippa2019polynomial}, which discusses both
convergence towards the circular average and polynomial
negative-Sobolev estimates under suitable regularity assumptions.
In particular, under  support assumptions and the condition of zero average
on almost every circle, that work establishes a
$\dot H^{-1}$ upper bound of order $t^{-\alpha/(\alpha+1)}$
for $C^{0,\alpha}$ initial data, with $\alpha\in(0,1]$.
The Fourier-based estimates of order
$t^{-\alpha}$ are also claimed for $\dot H^\alpha$ initial data.


Bru\`e, Coti Zelati, and Marconi \cite{brue2024enhanced}
study mixing and enhanced dissipation for two-dimensional
Hamiltonian flows on a compact 2D manifold using orbit-period functions and action-angle
coordinates. For the standard cellular flow, they obtain
a lower bound of order $t^{-1}$ and 
an upper bound of order $t^{-1/3+\varepsilon}$ for every sufficiently
small $\varepsilon>0$, with the latter estimate reflecting the
difficulty of controlling the hyperbolic regions.
These estimates concern the component with zero streamline
average; the global upper bound does not identify an optimal
asymptotic decay rate.


Against this background, the present work investigates the
common role of period variation in three observables:
decay towards the orbit average, orbit-relative geometric
mixing, and the growth of transported curves.
Under explicit separation and nondegeneracy assumptions,
we obtain sharp $(1+t)^{-1}$ estimates for the first two
quantities (see Theorems  \ref{third:optimal-negative-norm} and \ref{20250103-yb-theorem-MixingScale}) and a first-order expansion for curve length (Theorem \ref{20241021-yb-theorem-OptimalGrowthForCurves}).
The comparison with earlier work thus concerns the precise
hypotheses, the form of the quantitative estimates, and the
connection between these observables.

\subsection{Organization}

The rest of this paper is organized as follows.
Section~\ref{section-OrbitAndPeriod} studies the properties of the orbit period function and related estimates used throughout the paper.  Section~\ref{20241018-yb-JacobianOfVelocityField}
derives the moving-frame representation, large-time expansion, and
singular-value asymptotics of the flow Jacobian. 
Section~\ref{20241018-yb-TimeInvariantVelocityField} verifies Theorem \ref{20241021-yb-theorem-OptimalGrowthForCurves}.  Section~\ref{20250313-yb-subsection-MixingScale-DeformationOfOpenSet}
establishes a deformation estimate for orbit-relative complements, and 
Section~\ref{20241018-yb-section-MixingScale} proves Theorem \ref{20250103-yb-theorem-MixingScale}. 
Section \ref{section-ThirdMainTheorem} is devoted to the proof of Theorem \ref{third:optimal-negative-norm}.  Section~\ref{section-ExamplesOnAssumptions}
provides explicit counterexamples illustrating the roles of the hypotheses in main theorems. 
Section~\ref{20250313-yb-subsection-MixingScale-ProofsOfCorollaries}
derives consequences for mixing time and set separation and
Section~\ref{section-NumericalSimulations} presents numerical illustrations.
Section~\ref{section-conclusion} summarizes the common mechanism, separates
the content of the three main results, and records several directions for
further study.  The Appendix contains the auxiliary results and technical
proofs.

\section{Orbits and their periods}
\label{section-OrbitAndPeriod}

In this section we study the orbits of solutions to equation \eqref{20240925-yubiao-Flow} and the associated properties of their orbit periods. 
The dynamical behavior of equation \eqref{20240925-yubiao-Flow} may vary significantly near different equilibria. To classify these equilibria, we introduce the following definition.

\begin{definition}\label{20250329-yb-ClassificationOfEquilibria}
With the notations in assumption (A2), any critical point of $H$ is called an equilibrium of $V$, and 
an equilibrium $x_k^*$ is said to be stable (of center type) if $S_k$ is positive or negative definite, and unstable (of saddle type) if $S_k$ is indefinite. 
The real number $m_k\ge1$ is called the order of vanishing of $V$ at $x_k^*$. 
\end{definition}

    For each $s>0$, define
    \begin{align}\label{20250323-yb-LargePeirodLayer}
        \Omega_{\ge s}
        :=
        \big\{x\in\overline{\Omega}  ~:~  d(x,\{T=+\infty\})\ge s\big\}
        \quad \text{and}\quad 
        \Omega_{<s}
        :=
        \overline{\Omega}\setminus\Omega_{\ge s}.
    \end{align}
The main result of this section is stated below. It concerns the regularity properties of the orbit-period function $T(\cdot)$ and several related estimates.

\begin{proposition}\label{20250322-yb-proposition-UsefulPropertiesOfOrbitPeriod}
Let $T(\cdot)$ be defined in \eqref{20241021-yb-PeriodOfOrbits}. 
Then, for each $\hat x_0\in\overline{\Omega}$, the function $T$ is $C^1$ at $\hat x_0$ if and only if $T(\hat x_0)\in(0,+\infty)$. 
Moreover, given $\varepsilon>0$, there exists a constant $C=C(\Omega,V,\varepsilon)>0$ such that for every 
$x_0\in\Omega_{\ge\varepsilon}$ with $V(x_0) \neq 0$,
\begin{align}\label{20250329-yb-GlobalUpperBoundsOfOrbitPeriodAndRatioOfVelocity}
\frac{1}{T(x_0)}
+&   T(x_0)\,\|J_V(x_0)\|_{\mathbb{R}^{2\times2}}
\nonumber\\
&\quad+ |V(x_0)|\, \Big( T(x_0) + |\nabla T(x_0)| \Big)
+ \sup_{t,s\in\mathbb{R}}
\frac{|V(x(t;x_0))|}{|V(x(s;x_0))|}
\le C .
\end{align}
\end{proposition}


The proof of Proposition \ref{20250322-yb-proposition-UsefulPropertiesOfOrbitPeriod} relies on several auxiliary lemmas. The first records the boundary values and local consequences of the Hamiltonian structure. Its proof is included in Appendix~\ref{appendix-HamiltonianFunction}.

\begin{lemma}\label{20241001-yb-lemma-Hamiltonian}
Use the same notations as in assumptions (A1)--(A2). 
Let $x_k^*$ be an equilibrium and let $P_k(y):=\frac12|y|^{m_k-1}y^\top S_k y$, $y\in \mathbb R^2$. Then, it holds that as $y \rightarrow 0$, 
\begin{align}\label{20250329-yb-AsymptoticsOfHamiltonianAndVelocity}
 \left\{
   \begin{array}{r@{~}l}
       H(x_k^*+y)&=H(x_k^*)+P_k(y)+o(|y|^{m_k+1}),
       \\
       \nabla H(x_k^*+y)&=\nabla P_k(y)+o(|y|^{m_k}),
       \\
       D^2H(x_k^*+y)&=D^2P_k(y)+o(|y|^{m_k-1}),
       \\
       V(x_k^*+y)&=\mathcal J\nabla P_k(y)+o(|y|^{m_k}),
       \\
       J_V(x_k^*+y)&=\mathcal J D^2P_k(y)+o(|y|^{m_k-1})
   \end{array}
  \text{ with }
  \mathcal J := \begin{pmatrix}
      0 & 1 \\
      -1 & 0
  \end{pmatrix}.
 \right.
\end{align}
In particular, there are constants $c_k,C_k>0$ such that for all sufficiently small $y$,
\begin{align}\label{ham:velocity-bounds}
c_k|y|^{m_k}\le|\nabla H(x_k^*+y)|\le C_k|y|^{m_k}
~\text{ and }~
\|D^2H(x_k^*+y)\|\le C_k|y|^{m_k-1}.
\end{align}
Furthermore, the same bounds hold for $|V|$ and $\|J_V\|$, respectively, and every stable equilibrium lies in $\Omega$.
\end{lemma}

The following lemma concerns the positional relationship between the set $\Omega_{\geq \varepsilon}$ and the unstable equilibria. The proof is provided in Appendix \ref{appendix-EquilibriaInFinitePeriodRegion}.

\begin{lemma}\label{20250329-yb-lemma-EquilibriaInFinitePeriodRegion}
    Let $\varepsilon>0$ and  $\Omega_{\geq \varepsilon}$ be defined by \eqref{20250323-yb-LargePeirodLayer}. Then, the closed set $\Omega_{\geq \varepsilon}$ contains no unstable equilibrium.
\end{lemma}

The next lemma concerns several basic properties of the orbit period function. Its proof is provided in Appendix \ref{appendix-PropertiesOfOrbitPeriod}.

\begin{lemma}\label{20250322-yb-propsotion-PropertiesOfOrbitPeriod}
    Let $x_0 \in \overline{\Omega}$. Then, the following statements hold:
    \begin{itemize}
        \item[(i)] $T(x_0) = 0$ if and only if $x_0$ is an equilibrium (i.e., $V(x_0)=0$).
        
        \item[(ii)] $T(x_0) = +\infty$ if and only if $x(\cdot;x_0)$ is not constant and there exist two (possibly identical) unstable equilibria $x_k^*, x_l^*$ such that
        \begin{align}\label{20250329-yb-SourcesOfSolutionsWithInfinitePeriod}
            \lim_{ t\rightarrow - \infty} x(t;x_0) = x_k^*
            \text{ and }
            \lim_{ t\rightarrow + \infty} x(t;x_0) = x_l^*.
        \end{align}
        Moreover, there are at most finitely many trajectories to equation \eqref{20240925-yubiao-Flow} verifying \eqref{20250329-yb-SourcesOfSolutionsWithInfinitePeriod}. 
        
        \item[(iii)] If $T(x_0) \in (0, +\infty)$, then $T(\cdot)$ is locally $C^1$ at $x_0$.
    \end{itemize}
\end{lemma}

The following lemma provides estimates for the orbit-period function around a stable equilibrium.

\begin{lemma}\label{20250329-yb-lemma-OrbitPeriodAroundEquilibra}
Let $x_k^*$ be a stable equilibrium of order $m_k$. 
Then, it holds that as $x\to x_k^*$,
\begin{align}
\label{20250329-yb-AsymptoticExpansionOfOrbitPeriodAroundEquilibrium}
\left\{
\begin{array}{r@{~}l}
    T(x)&\asymp |x-x_k^*|^{ -(m_k-1)},
    \vspace{0.5em}\\
    |\nabla T(x)|&\asymp |x-x_k^*|^{-m_k}
    \quad\text{if }m_k>1,
    \vspace{0.5em}\\
    |\nabla T(x)|&\asymp o( |x-x_k^*|^{-1} )
    \quad\text{if }m_k=1.
\end{array}
\right.
\end{align}
Moreover, for some $\delta,C>0$,
\begin{align}\label{ham:center-speed-ratio}
\sup_{t,s\in\mathbb R}
\frac{|V(x(t;x_0))|}{|V(x(s;x_0))|}
\le C ~\text{ when }  0 < |x_0  - x_k^* |  < \delta. 
\end{align}
\end{lemma}

\begin{proof}
 The key idea is that, near a stable equilibrium $x_k^*$, the level sets of the Hamiltonian are small closed curves of diameter comparable to $|x-x_k^*|$, while the velocity magnitude is of order $|x-x_k^*|^{m_k}$. Hence the orbit period is comparable to the ratio of orbit length to speed, which yields $T(x)\asymp |x-x_k^*|^{-(m_k-1)}$. The rest of the proof makes this scaling argument precise.

    Since $x_k^*$ is stable, Definition \ref{20250329-yb-ClassificationOfEquilibria} implies that $S_k > 0$ or $S_k <0$. Without loss of generality, assume that $S_k>0$; the argument for $S_k < 0$ is similar. 
    By assumption (A2), there exists a new orthonormal coordinate system $\{\vec v_1,\vec v_2\}$ centered at $x_k^*$, two positive numbers $a,b$, and a remainder function $R \in C^2(\mathbb R^2)$, satisfying
    \begin{align}\label{20250330-yb-NewGoodRemainder}
        \partial_{x_1}^{\alpha}\partial_{x_2}^{\beta} R(x_1,x_2)
        = o\big(|(x_1,x_2)|^{2-\alpha-\beta}\big)
        \quad \text{as } (x_1,x_2)\to 0,
    \end{align}
    for all $(\alpha,\beta)\in\mathbb N^2$ with $\alpha+\beta\le 2$, such that
    \begin{align*}
        H(x_k^* + x_1\vec v_1 + x_2\vec v_2) - H(x_k^*)
        =  |(x_1,x_2)|^{m_k-1}
        \big(a^2 x_1^2 + b^2 x_2^2 + R(x_1,x_2)\big),
        ~(x_1,x_2) \in \mathbb R^2.
    \end{align*}
    For simplicity, denote by $H(x_1,x_2)$ the function on the left-hand side of the above equality. Write
    \begin{align}\label{20250330-yb-SimplifiedHamiltonianNearStableEquilibrium}
        H(x_1,x_2)
        &= H_p(x_1,x_2) + |(x_1,x_2)|^{m_k-1} R(x_1,x_2) \nonumber\\
        &:= |(x_1,x_2)|^{m_k-1}\big(a^2x_1^2 + b^2x_2^2 + R(x_1,x_2)\big).
    \end{align}
    Let $(\bar x_1(\delta),\bar x_2(\delta))$ with $\delta\in[0,2\pi)$ be a parametrization of the curve
    \begin{align*}
        1 = H_p(x_1,x_2)
        = |(x_1,x_2)|^{m_k-1}(a^2x_1^2+b^2x_2^2),
        ~ (x_1,x_2)\in\mathbb R^2.
    \end{align*}
    Then, by \eqref{20250330-yb-SimplifiedHamiltonianNearStableEquilibrium} and \eqref{20250330-yb-NewGoodRemainder}, for sufficiently small $\varepsilon>0$, the level set $\{H=\varepsilon^{m_k+1}\}$ near $x_k^*$ admits the parametrization
    \begin{align}\label{20250330-yb-ParameterizedHamiltonianAroundStableEquilibrium}
        \varepsilon \big(
         \bar x_1(\delta) + o(\delta;\varepsilon),
        \bar x_2(\delta) + o(\delta;\varepsilon)
        \big),
        ~ \delta\in[0,2\pi),
    \end{align}
    where the remainder $o(\delta;\varepsilon)$ satisfies the estimate for $(\alpha,\beta) \in \mathbb N^2$ with $\alpha + \beta \leq 1$: 
    \begin{align}\label{SmallTerms-20260916}
          \partial_{\varepsilon}^{\alpha}  \partial_{\delta}^{\beta} o(\delta;\varepsilon)
         = o(  \varepsilon^{-\alpha})
         \text{ uniformly in } \delta
         \text{ as } \varepsilon \rightarrow 0^+.
    \end{align}
    
    Next, using \eqref{ham:representation}, \eqref{20250330-yb-SimplifiedHamiltonianNearStableEquilibrium}, and \eqref{20250330-yb-ParameterizedHamiltonianAroundStableEquilibrium}, we obtain that on the level set $\{ H =\varepsilon^{m_k+1} \}$, 
    \begin{align}\label{20250403-yb-LeadingTermInGradientOfHamiltonian}
        V^{\perp}(x_1,x_2)
        &= \nabla H(x_1,x_2) \nonumber\\
        &= |(x_1,x_2)|^{m_k-1}(2a^2x_1,2b^2x_2)  \nonumber\\
        &\quad + (m_k-1)|(x_1,x_2)|^{m_k-3}
        (a^2x_1^2+b^2x_2^2)(x_1,x_2)
        + o(|(x_1,x_2)|^{m_k}) \nonumber\\
        &= \varepsilon^{m_k}\nabla H_p(\bar x_1,\bar x_2) +  \varepsilon^{m_k} o(\delta;\varepsilon),
    \end{align}
    where the remainder $o(\delta;\varepsilon)$ also satisfies \eqref{SmallTerms-20260916}. 
    Consequently, from \eqref{20250330-yb-ParameterizedHamiltonianAroundStableEquilibrium} we compute the orbit period of the level set $\{H=\varepsilon^{m_k+1}\}$ around $x_k^*$:
    \begin{align}\label{20250330-yb-ComputeOrbitPeriod}
        T(\{H=\varepsilon^{m_k+1}\})
        &= \int_{\{H=\varepsilon^{m_k+1}\}}
        \frac{1}{|V(x)|}\,ds(x)  \nonumber\\
        &= \int_{\{H=\varepsilon^{m_k+1}\}}
        \frac{1}{|\nabla H(x_1,x_2)|}\,ds(x_1,x_2)  \nonumber\\
        &= \int_0^{2\pi}
        \frac{|(x_1'(\delta),x_2'(\delta))|}
        {|\nabla H(x_1(\delta),x_2(\delta))|}
        d\delta  \nonumber\\
        &= \varepsilon^{1-m_k}
        \int_0^{2\pi}
        \frac{|(\bar x_1'(\delta),\bar x_2'(\delta))|+ o(\delta;\varepsilon)}
        {|\nabla H_p(\bar x_1(\delta),\bar x_2(\delta))|+  o(\delta;\varepsilon)}
        d\delta  \nonumber\\
        &= \varepsilon^{1-m_k}
        \left(
        \int_{\{H_p=1\}}
        \frac{ds}{|\nabla H_p|}
        + o(\delta;\varepsilon)
        \right).
    \end{align}
    Here, $ds$ is the arc-length element and $o(\delta;\varepsilon)$ denotes some term satisfying \eqref{SmallTerms-20260916}.
    
    From \eqref{20250330-yb-ParameterizedHamiltonianAroundStableEquilibrium} we also have
    \[
    |x-x_k^*| \asymp\varepsilon
    \quad \text{as } x\to x_k^*.
    \]
    Combining this with \eqref{20250330-yb-ComputeOrbitPeriod} yields the first relation in
    \eqref{20250329-yb-AsymptoticExpansionOfOrbitPeriodAroundEquilibrium}. 
    The second and third relations follow from \eqref{20250403-yb-LeadingTermInGradientOfHamiltonian}, \eqref{20250330-yb-ComputeOrbitPeriod}, and the identity
    \[
    |\nabla T|
    = |V|\,\left|\frac{dT}{dH}\right|
    = |V| \frac{|dT/d\varepsilon|}{|dH/d\varepsilon|}.
    \]
    
    Finally, since a solution $x(\cdot; x_0)$ lies in a level set of the Hamiltonian function $H$, 
    \eqref{ham:center-speed-ratio} follows from \eqref{20250403-yb-LeadingTermInGradientOfHamiltonian} and \eqref{20250330-yb-SimplifiedHamiltonianNearStableEquilibrium}. 
    This completes the proof.
\end{proof}

Now we are in a position to prove Proposition~\ref{20250322-yb-proposition-UsefulPropertiesOfOrbitPeriod}.

\begin{proof}[Proof of Proposition~\ref{20250322-yb-proposition-UsefulPropertiesOfOrbitPeriod}]
    First, we  show that $T$ is $C^1$ at $\hat x_0$ if and only if $T(\hat x_0) \in (0,+\infty)$. 
    The sufficiency follows directly from (iii) of Lemma \ref{20250322-yb-propsotion-PropertiesOfOrbitPeriod}. 
    For the necessity, suppose that $T$ is $C^1$ at $\hat x_0$. Clearly,
    \begin{align}\label{20250403-yb-FinitePeriodAtAPoint}
        T(\hat x_0) < +\infty.
    \end{align}
    We claim that
    \begin{align}\label{20250403-yb-PositivePeriodAtAPoint}
        T(\hat x_0) > 0.
    \end{align}
    Indeed, if $T(\hat x_0)=0$, then by (i) of Lemma \ref{20250322-yb-propsotion-PropertiesOfOrbitPeriod}, $\hat x_0$ must be an equilibrium point. At a stable equilibrium, \eqref{20250329-yb-AsymptoticExpansionOfOrbitPeriodAroundEquilibrium} shows that the periods of nearby nonconstant orbits stay bounded away from zero.
    At an unstable equilibrium,
    \eqref{20250330-yb-InfinitePeriodAroundUnstableEquilibrium} in the appendix gives a sequence $x_n\to\hat x_0$ ($n \rightarrow+\infty$) with $T(x_n)=+\infty$.
    In either case, $T(\cdot)$ is not continuous at $\hat x_0$, since $T(\hat x_0)=0$. 
    Hence \eqref{20250403-yb-PositivePeriodAtAPoint} holds. Together with \eqref{20250403-yb-FinitePeriodAtAPoint}, this implies that $T(\hat x_0)\in (0,+\infty)$, completing the proof of the necessity.
    
    \vspace{1ex}
    
    It remains to prove \eqref{20250329-yb-GlobalUpperBoundsOfOrbitPeriodAndRatioOfVelocity}. 
    By Lemma \ref{20250329-yb-lemma-EquilibriaInFinitePeriodRegion}, the compact set $\Omega_{\geq \varepsilon}$ contains no unstable equilibria. Moreover, by (iii) of Lemma \ref{20250322-yb-propsotion-PropertiesOfOrbitPeriod}, the function $T(\cdot)$ is $C^1$ on the set $\{0<T<+\infty\}$. Hence $T$ is $C^1$ on $\Omega_{\geq \varepsilon}$ except possibly at finitely many stable equilibria. Near a stable equilibrium $x_k^*$, the estimates \eqref{20250329-yb-AsymptoticExpansionOfOrbitPeriodAroundEquilibrium} and \eqref{20250329-yb-AsymptoticsOfHamiltonianAndVelocity} give that as $x \rightarrow x_k^*$,
    \[
    T(x)^{-1}\asymp |x - x_k^* |^{m_k-1},~
    T(x)\|J_V(x)\|\lesssim1,   \text{ and }
    |V(x)|\, \Big( T(x) + |\nabla T(x)| \Big) \lesssim1. 
    \]
    Therefore, the estimates for all terms except the last one in \eqref{20250329-yb-GlobalUpperBoundsOfOrbitPeriodAndRatioOfVelocity} are established.

    Finally, we estimate the speed ratio (i.e., the last term) in \eqref{20250329-yb-GlobalUpperBoundsOfOrbitPeriodAndRatioOfVelocity}. Note that  $\Omega_{\geq \varepsilon}$ contains no unstable equilibria (see Lemma \ref{20250329-yb-lemma-EquilibriaInFinitePeriodRegion}). 
    On one hand, this estimate holds around each stable equilibrium (see \eqref{ham:center-speed-ratio}). On the other hand, the orbits of points away from the set $\{T=0,+\infty\}$ are still separated from this set (see  \eqref{BothAwayFromExtremeOrbitPeriods-20260926} in the appendix).     
    Thus, the estimate \eqref{20250329-yb-GlobalUpperBoundsOfOrbitPeriodAndRatioOfVelocity} holds. This completes the proof. 
\end{proof}

\section{Jacobian matrices of the flow}
\label{20241018-yb-JacobianOfVelocityField}

In this section, we investigate the Jacobian matrices $\{J_{\varPhi(t)}\}_{t\ge0}$ of the flow $\{\varPhi(t)\}_{t\ge0}$ defined in \eqref{20240926-yubiao-DefinitionOfFlow}. 
These matrices play a central role in the geometric interpretation of the main theorems and in the analysis of the deformation of curves and sets under the flow (see Proposition \ref{rev:prop52}). 

This section is organized as follows. 
Subsection \ref{20250324-yb-section-ExpressionOfJacobian} derives a useful representation of the Jacobian matrices along trajectories. 
Subsection \ref{20250324-yb-section-AsymptoticExpansionOfJacobian} establishes their leading asymptotic expansion  for large time. 
Subsection \ref{20250324-yb-section-SingularValues} analyzes the singular values of the Jacobian matrices and their implications for the geometric deformation under the flow.

\subsection{Representation of the flow Jacobian in moving frames}
\label{20250324-yb-section-ExpressionOfJacobian}

In this subsection, we focus on the Jacobian matrices $\{J_{\varPhi(t)}\}_{t\ge0}$ of the flow $\{\varPhi(t)\}_{t\ge0}$ and derive their representation in moving frames along the trajectories of equation \eqref{20240925-yubiao-Flow}. The main result is stated below.

\begin{theorem}\label{20241015-yb-Erlangen-theorem-FlowOfTangentVectorsForTimeInvariantVelocityField}
Let $x_0\in \overline{\Omega}$ be such that $V(x_0)\neq0$, and let $x(\cdot)$ be the solution $x(\cdot; x_0)$ to equation \eqref{20240925-yubiao-Flow}.  Then, it holds that for all $t \in \mathbb R$,
\begin{align}\label{20241018-yb-SpecialExpressionOfJacobianOfFlow}
    J_{\varPhi(t)}(x_0) 
    =&
    |V(x_0)|^{-2}\, V(x(t))  \otimes
    V(x_0)
    +
    |V(x(t))|^{-2}\, V^{\perp}(x(t))  \otimes
     V^{\perp}(x_0)
    \nonumber\\
    &+
    \Bigg[
    \int_0^t
    \left(
    \mathrm{curl}\,\frac{V(y)}{|V(y)|^2}
    \right)\Big|_{y=x(\tau)} d\tau
    \Bigg]
    \;  V(x(t))  \otimes  V^{\perp}(x_0)     .
\end{align}
\end{theorem}

To prove Theorem \ref{20241015-yb-Erlangen-theorem-FlowOfTangentVectorsForTimeInvariantVelocityField}, we first introduce the following auxiliary result. 
To state it, we define orthogonal frames along a trajectory $x(\cdot;x_0)$ of equation \eqref{20240925-yubiao-Flow}. 
Under assumptions (A1)--(A2), we introduce the orthogonal frame field
\[
U \in C^1\big(\overline{\Omega}\setminus\{x_k^*\}_{k=1}^K; \, SO(\mathbb R^{2\times2})\big)
\]
given by
\begin{align}\label{20241015-yb-OrthogonalFrame}
U := (\vec e \ \ \vec e^{\perp}) 
\quad \text{where}\quad
\vec e(x) := \frac{V(x)}{|V(x)|},
~
x \in \overline{\Omega}\setminus\{x_k^*\}_{k=1}^K .
\end{align}
Here, $\vec{e}$ stands for the direction along the velocity field or streamlines,   and $\vec e^{\perp}$  stands for the direction across streamlines.

\begin{proposition}\label{20241015-yb-Erlangen-theorem-FlowOfTangentVectors}
With the same notations as in Theorem \ref{20241015-yb-Erlangen-theorem-FlowOfTangentVectorsForTimeInvariantVelocityField}, define
\begin{align}\label{20241020-yb-NormalizedExpressionForJacobianOfVelocityField}
A(t;x_0) :=
\begin{pmatrix}
V \cdot \nabla \ln |V| &
|V|\,\mathrm{curl}\,\vec e
+ \vec e \times \nabla |V|
\\
0 &
-\, V \cdot \nabla \ln |V|
\end{pmatrix}
\Big|_{x=x(t)},
~ t \in \mathbb R.
\end{align}
Then
\begin{align}\label{20241018-yb-ExpressionOfJacobianOfFlow}
J_{\varPhi(t)}(x_0)
=
U\big(x(t)\big)\,
\varPsi(t;x_0)\,
U(x(0))^{-1},
~ t \in \mathbb R,
\end{align}
where $\{\varPsi(t;x_0)\}_{t\ge0}$ is the evolution system generated by the  matrices $\{A(t;x_0)\}_{t\ge0}$, i.e.,
\begin{align}\label{20241019-yb-PrincipalMatrixSolutionOfJacobianOfVelocityField}
\psi'(t) = A(t;x_0)\psi(t),
~ t \in \mathbb R;
~~ \psi(0)=I_2,
\end{align}
where $I_2$ is the identity matrix in $\mathbb R^{2\times2}$.
\end{proposition}

With the help of Proposition \ref{20241015-yb-Erlangen-theorem-FlowOfTangentVectors}, we now prove Theorem \ref{20241015-yb-Erlangen-theorem-FlowOfTangentVectorsForTimeInvariantVelocityField}. 
The proof of Proposition \ref{20241015-yb-Erlangen-theorem-FlowOfTangentVectors} will be given later.

\begin{proof}[Proof of Theorem \ref{20241015-yb-Erlangen-theorem-FlowOfTangentVectorsForTimeInvariantVelocityField}]
Since $V(x_0)\neq0$, the trajectory $x(\cdot)$ is not an equilibrium solution, and therefore
\[
V(x(t))\neq0
~ \text{for all } t \in \mathbb R.
\]
Hence the vector field $\vec e = V/|V|$ is well defined along the trajectory $x(\cdot)$.
Define
\begin{align}\label{20241018-yb-CoefficientsForJacobianOfFlow}
B(\tau;x_0)
:=
\Big(
|V|\,\mathrm{curl}\,\vec e
+
\vec e \times \nabla |V|
\Big)\Big|_{x=x(\tau)},
~ \tau \in \mathbb R.
\end{align}
A direct computation shows that
\begin{align}\label{20250324-yb-DirectComputationsOfAngularMomentum}
B(\tau;x_0)
=\Big[
|V|^2
\mathrm{curl}\,\big(|V|^{-2}V\big)
\Big]
\Big|_{x=x(\tau)},
~ \tau \in \mathbb R.
\end{align}
Observe that the matrix $A(t;x_0)$ in \eqref{20241020-yb-NormalizedExpressionForJacobianOfVelocityField} is upper triangular. 
Solving the linear system \eqref{20241019-yb-PrincipalMatrixSolutionOfJacobianOfVelocityField} (as well as \eqref{20241020-yb-NormalizedExpressionForJacobianOfVelocityField}), we obtain from \eqref{20241018-yb-CoefficientsForJacobianOfFlow} that 
\begin{align*}
\varPsi(t;x_0)
=
\begin{pmatrix}
R(t;x_0) &
R(t;x_0)\displaystyle\int_0^t
\frac{B(\tau;x_0)}{R^2(\tau;x_0)}\,d\tau
\\
0 &
R(t;x_0)^{-1}
\end{pmatrix},
~ t \in \mathbb R,
\end{align*}
where for all $\tau \in \mathbb R$,
\begin{align*}
R(\tau;x_0)
:=&
\exp\!\left[
\int_0^{\tau}
\big(
V\cdot\nabla\ln|V|
\big)\big|_{x=x(\eta)}\,d\eta
\right]
\nonumber\\
=& \exp \Big( \int_0^{\tau}
\frac{d}{ds}\big(\ln|V(x(s))|\big)\,ds \Big)
= \frac{|V(x(\tau))|}{|V(x_0)|}.
\end{align*}
This, together with \eqref{20241018-yb-ExpressionOfJacobianOfFlow} (as well as
\eqref{20250324-yb-DirectComputationsOfAngularMomentum} and \eqref{20241015-yb-OrthogonalFrame}),  yields \eqref{20241018-yb-SpecialExpressionOfJacobianOfFlow}. 
The proof is completed.
\end{proof}

This proposition expresses the Jacobian of the flow in a moving frame aligned with the velocity field, revealing that local deformation consists of speed variation along trajectories and cumulative shear between neighboring streamlines. The first diagonal term 
of $A(t; x_0)$ describes acceleration or deceleration along streamlines, and the second diagonal term, with its negative sign, ensures volume preservation due to incompressibility of the flow. The off-diagonal term 
measures shear or the total transverse deformation between neighboring streamlines. 

The remainder of this subsection is devoted to the proof of Proposition 
\ref{20241015-yb-Erlangen-theorem-FlowOfTangentVectors}, which relies on two lemmas. 
The first lemma shows that the Jacobian matrices 
$\{J_{\varPhi(t)}(x_0)\}_{t\ge0}$ satisfy a linear time-varying ordinary differential equation.

\begin{lemma}\label{20250313-yb-lemma-ComputeJacobianByODE}
Let $x_0\in\overline{\Omega}$ and $w_0\in\mathbb{R}^2$. 
Let $w(\cdot;x_0,w_0)$ denote the solution of the linear time-varying system
\begin{align}\label{20240925-yubiao-FlowForVectorFields}
w'(t)=J_V(x(t;x_0))\,w(t), ~ t \in \mathbb R; 
~~ w(0)=w_0,
\end{align}
where $x(\cdot;x_0)$ is the solution of equation \eqref{20240925-yubiao-Flow}. 
Then, for all $t \in \mathbb R$,
\begin{align}\label{20241015-yb-ComputationOnFlowOfTangentVectors}
J_{\varPhi(t)}(x_0)\,w_0
=
w(t;x_0,w_0).
\end{align}
Equivalently, the family of Jacobian matrices 
$\{J_{\varPhi(t)}(x_0)\}_{ t \in \mathbb R }$ is the evolution system generated by the time-varying matrices 
$\{J_V(x(t;x_0))\}_{ t \in \mathbb R }$.
\end{lemma}

\begin{proof}
The proof follows from differentiating  equation 
\eqref{20240925-yubiao-Flow} with respect to the initial point $x_0$, 
and is therefore omitted.
\end{proof}

The second lemma is stated below. It provides an expression for the Jacobian matrix of $V$ in the orthogonal frame field introduced in \eqref{20241015-yb-OrthogonalFrame}.

\begin{lemma}
Let $\vec e$ and $U$ be defined in \eqref{20241015-yb-OrthogonalFrame}. 
Then, for each $x \in \overline{\Omega} \setminus\{x_k^*\}_{k=1}^K $,
\begin{align}\label{20241015-yb-Erlangen-NewJacobian}
U(x)^{-1} J_V(x) U(x)
=
\begin{pmatrix}
V \cdot \nabla \ln |V|  &  \vec e \times \nabla |V|   \\
|V|\, \mathrm{curl}\,\vec e      &  - V \cdot \nabla \ln |V|
\end{pmatrix}.
\end{align}
Here, $\mathrm{curl}\,V := \nabla \times V$.
\end{lemma}

\begin{proof}
For simplicity, set $\vec n := \vec e^{\perp}$. Then
\begin{align}\label{20241015-yb-NewJacobian-InnerProducts}
U(x)^{-1} J_V(x) U(x)
=
\begin{pmatrix}
\langle \vec e, J_V \vec e \rangle_{\mathbb R^2}
&
\langle \vec e, J_V \vec n \rangle_{\mathbb R^2}
\\
\langle \vec n, J_V \vec e \rangle_{\mathbb R^2}
&
\langle \vec n, J_V \vec n \rangle_{\mathbb R^2}
\end{pmatrix}.
\end{align}
It therefore suffices to compute each entry of this matrix. The proof proceeds in three steps.

\vskip 5pt
\noindent\textit{Step 1. Computation of $\langle \vec e, J_V \vec e \rangle_{\mathbb R^2}$ and $\langle \vec n, J_V \vec n \rangle_{\mathbb R^2}$.}

First note that
\begin{align}\label{RelationBetweenGradientAndJacobian-20260926}
\nabla |V| = J_V^{\top} \vec e .
\end{align}
Hence
\begin{align*}
\langle \vec e, J_V \vec e \rangle_{\mathbb R^2}
=
\vec e^{\top} J_V \vec e
=
(\nabla |V|)\cdot \vec e
=
(|V|\,\nabla \ln |V|)\cdot \vec e
=
(\nabla \ln |V|)\cdot V.
\end{align*}
Next, since $\{\vec e,\vec n\}$ forms an orthonormal basis of $\mathbb R^2$,
\[
\mathrm{tr}\,J_V
=
\langle \vec e, J_V \vec e \rangle_{\mathbb R^2}
+
\langle \vec n, J_V \vec n \rangle_{\mathbb R^2},
\]
which yields
\begin{align*}
\langle \vec n, J_V \vec n \rangle_{\mathbb R^2}
=
\mathrm{tr}\,J_V
-
\langle \vec e, J_V \vec e \rangle_{\mathbb R^2}
=
\mathrm{div}\,V
-
V\cdot\nabla\ln|V|
=-V\cdot\nabla\ln|V|.
\end{align*}
This completes Step 1.

\vskip 5pt
\noindent\textit{Step 2. Computation of $\langle \vec e, J_V \vec n \rangle_{\mathbb R^2}$.}

It follows from \eqref{RelationBetweenGradientAndJacobian-20260926} that 
\begin{align*}
\langle \vec e, J_V \vec n \rangle_{\mathbb R^2}
=
\vec e^{\top} J_V \vec n
=
(\nabla |V|)\cdot \vec n
=
\vec n\cdot \nabla |V|
=
\vec e \times \nabla |V|.
\end{align*}

\vskip 5pt
\noindent
\noindent\textit{Step 3. Computation of $\langle \vec n, J_V \vec e \rangle_{\mathbb R^2}$.}

Using the cyclic property of the trace, $\mathrm{tr}(AB)=\mathrm{tr}(BA)$, we have
\begin{align*}
\langle \vec n, J_V \vec e \rangle_{\mathbb R^2}
&=
\vec n^{\top} J_V \vec e
=
\mathrm{tr}\!\left(\vec n^{\top} J_V \vec e\right)
=
\mathrm{tr}\!\left(J_V\,\vec e\,\vec n^{\top}\right)
\\
&=
\mathrm{tr}\!\left(J_V\,\vec n\,\vec e^{\top}\right)
+
\mathrm{tr}\!\left(J_V(\vec e\,\vec n^{\top}-\vec n\,\vec e^{\top})\right).
\end{align*}
Note that $\mathrm{tr}\!\left(J_V\,\vec n\,\vec e^{\top}\right)
=\mathrm{tr}\!\left(\vec e^{\top} J_V\,\vec n\right)
= \langle \vec e, J_V \vec n \rangle_{\mathbb R^2}$ and 
\[
\vec e\,\vec n^{\top}-\vec n\,\vec e^{\top}
=
\begin{pmatrix}
0 & 1\\
-1 & 0
\end{pmatrix}.
\]
Thus
\begin{align*}
\langle \vec n, J_V \vec e \rangle_{\mathbb R^2}
&=
\langle \vec e, J_V \vec n \rangle_{\mathbb R^2}
+
\mathrm{curl}\,V.
\end{align*}
Since $V=|V|\vec e$, we compute
\[
\mathrm{curl}\,V
=
\mathrm{curl}(|V|\vec e)
=
\nabla |V| \times \vec e
+
|V|\,\mathrm{curl}\,\vec e.
\]
Substituting the result of Step 2 then yields
\[
\langle \vec n, J_V \vec e \rangle_{\mathbb R^2}
=
\vec e\times\nabla|V|
+
\big(\nabla|V|\times\vec e+|V|\mathrm{curl}\vec e\big)
=
|V|\,\mathrm{curl}\vec e .
\]
\vskip 5pt
Finally, combining \eqref{20241015-yb-NewJacobian-InnerProducts} with the results obtained in Steps 1--3 gives \eqref{20241015-yb-Erlangen-NewJacobian}. This completes the proof. 
\end{proof}

Now we are in a position to prove Proposition \ref{20241015-yb-Erlangen-theorem-FlowOfTangentVectors}.

\begin{proof}[Proof of Proposition \ref{20241015-yb-Erlangen-theorem-FlowOfTangentVectors}]
We use \eqref{20241015-yb-ComputationOnFlowOfTangentVectors} to compute $J_{\varPhi(t)}(x_0)$. 
For this purpose, let $w$ be a solution of equation \eqref{20240925-yubiao-FlowForVectorFields}. 
Recall the orthogonal frame field $U=(\vec e,\vec e^{\perp})$ from \eqref{20241015-yb-OrthogonalFrame}. 
For simplicity, set
\[
\vec e_1(t):=\vec e(x(t)) 
~\text{ and }~
\vec e_2(t):=\vec e^{\perp}(x(t)),
~ t \in \mathbb R.
\]
Then, for each $t \in \mathbb R$, 
\[
\langle \vec e_i(t),\vec e_j(t)\rangle_{\mathbb R^2}
=
\delta_{ij},
~
i,j=1,2.
\]
Direct computation shows that
\begin{align}\label{ODE-MovingFrame-20260926}
    \frac{d}{dt}\vec e_1=k\,\vec e_2
    ~\text{ and }~
    \frac{d}{dt}\vec e_2=-k\,\vec e_1,
    ~ t \in \mathbb R,
\end{align}
where
\begin{align}\label{20241016-yb-Curvature}
k(t)
&:=
\Big\langle
\frac{d}{dt}\vec e_1(t),\vec e_2(t)
\Big\rangle_{\mathbb R^2}
=
\Big\langle
 \big( x'(t) \cdot\nabla \big)\vec e (t),
\vec e^{\perp}(t)
\Big\rangle_{\mathbb R^2}
\nonumber\\
&=
\vec e\times (V\cdot\nabla)\vec e
=
|V|\,\vec e\times\big[(\vec e\cdot\nabla)\vec e\big]
=
|V|\,\mathrm{curl}\,\vec e.
\end{align}

Next write
\begin{align}\label{20241016-yb-DecompositionOfTangentVectors}
w(t)=w_1(t)\vec e_1(t)+w_2(t)\vec e_2(t),
~ t \in \mathbb R,
\end{align}
where $w_1,w_2\in C^1( \mathbb R )$.  By \eqref{ODE-MovingFrame-20260926}, differentiating on both sides of \eqref{20241016-yb-DecompositionOfTangentVectors} yields
\[
w'=(w_1'-kw_2)\vec e_1+(w_2'+kw_1)\vec e_2.
\]
Since $w$ solves \eqref{20240925-yubiao-FlowForVectorFields}, we obtain
\begin{align*}
\frac{d}{dt}
\begin{pmatrix}
w_1(t)\\
w_2(t)
\end{pmatrix}
+
\begin{pmatrix}
0&-k\\
k&0
\end{pmatrix}
\begin{pmatrix}
w_1(t)\\
w_2(t)
\end{pmatrix}
=
\begin{pmatrix}
\langle\vec e_1,J_V\vec e_1\rangle_{\mathbb R^2} &
\langle\vec e_1,J_V\vec e_2\rangle_{\mathbb R^2}
\\
\langle\vec e_2,J_V\vec e_1\rangle_{\mathbb R^2} &
\langle\vec e_2,J_V\vec e_2\rangle_{\mathbb R^2}
\end{pmatrix}
\begin{pmatrix}
w_1(t)\\
w_2(t)
\end{pmatrix}.
\end{align*}
This, along with \eqref{20241015-yb-Erlangen-NewJacobian} and 
\eqref{20241016-yb-Curvature}, gives 
\begin{align*}
    \frac{d}{dt}
    \begin{pmatrix}
        w_1(t) \\
        w_2(t)
    \end{pmatrix}
    =
    \begin{pmatrix}
        V \cdot \nabla \ln |V|  &
        |V|  \, \text{curl}\, \vec{e}
        +   \vec{e} \times \nabla |V|
        \\
        0&          - V \cdot \nabla \ln |V|
    \end{pmatrix}
    \begin{pmatrix}
        w_1(t) \\
        w_2(t)
    \end{pmatrix}
= A(t; x_0)
\begin{pmatrix}
    w_1(t) \\
    w_2(t)
\end{pmatrix},
\end{align*}
where $A(\cdot;x_0)$ is defined in 
\eqref{20241020-yb-NormalizedExpressionForJacobianOfVelocityField}.
Then, by \eqref{20241019-yb-PrincipalMatrixSolutionOfJacobianOfVelocityField}, we obtain
\[
\begin{pmatrix}
w_1(t)\\
w_2(t)
\end{pmatrix}
=
\varPsi(t;x_0)
\begin{pmatrix}
w_1(0)\\
w_2(0)
\end{pmatrix},
~ t \in \mathbb R.
\]
Using \eqref{20241016-yb-DecompositionOfTangentVectors} twice, we get that when $t \in \mathbb R$, 
\begin{align*}
w(t)
&=
\begin{pmatrix}
\vec e_1(t)&\vec e_2(t)
\end{pmatrix}
\begin{pmatrix}
w_1(t)\\
w_2(t)
\end{pmatrix}
\\
&=
\begin{pmatrix}
\vec e_1(t)&\vec e_2(t)
\end{pmatrix}
\varPsi(t;x_0)
\begin{pmatrix}
\vec e_1(0)&\vec e_2(0)
\end{pmatrix}^{-1}
w(0)
\\
&=
U(x(t))\,\varPsi(t;x_0)\,U(x(0))^{-1}\,w(0).
\end{align*}
Combining this with \eqref{20241015-yb-ComputationOnFlowOfTangentVectors} yields 
\eqref{20241018-yb-ExpressionOfJacobianOfFlow}. This completes the proof.
\end{proof}

\subsection{Asymptotic expansion}
\label{20250324-yb-section-AsymptoticExpansionOfJacobian}

This subsection establishes an asymptotic expansion for the Jacobian matrices of the flow $\{\varPhi(t)\}_{t\in\mathbb{R}}$. The result is based on the representation derived in Theorem \ref{20241015-yb-Erlangen-theorem-FlowOfTangentVectorsForTimeInvariantVelocityField}. 
It is closely related to the growth of the singular values of these Jacobian matrices, which will be studied in Theorem \ref{20250326-yb-theorem-SingularValuesOfJacobian}.

\begin{theorem}\label{20241023-yb-proposition-AsymptoticForJacobianOfVeolocityField} 
Let $T$ be defined by \eqref{20241021-yb-PeriodOfOrbits}. Let $\varepsilon>0$ and let $\Omega_{\geq\varepsilon}$ be defined by \eqref{20250323-yb-LargePeirodLayer}. 
Take $x_0\in\Omega_{\geq\varepsilon}$ such that $V(x_0)\neq0$, and denote by $x(\cdot)$ the solution $x(\cdot;x_0)$ of equation \eqref{20240925-yubiao-Flow}. 
Then there exists a constant $C=C(\Omega,V,\varepsilon)>0$ such that
\begin{align}\label{20241021-yb-AsymptoticExpansionForJacobianOfFlow}
\sup_{t\in\mathbb{R}} 
\Big|
J_{\varPhi(t)}(x_0)w_0
+
t\,\big(w_0\cdot\nabla\ln T(x_0)\big)\,V(x(t))
\Big|
\le C|w_0|,
~~ w_0 \in \mathbb R^2.
\end{align}
In particular, there is a constant $C=C(\Omega,V,\varepsilon)>0$ such that for each $t \in \mathbb R$, 
\begin{align}\label{20260823-LowerBoundOfJacobianOfFlow}
    C^{-1}  \big( 1+   |t|  \big)^{-1}
    | w_0 |     \leq | J_{\varPhi(t)}(x_0) w_0 |
    \leq C  \big( 1+   |t|  \big)  | w_0|,
    ~w_0 \in \mathbb R^2.
\end{align}
\end{theorem}

Before proving Theorem \ref{20241023-yb-proposition-AsymptoticForJacobianOfVeolocityField}, we present two auxiliary lemmas. 
The first lemma concerns derivatives of the orbit period function $T(\cdot)$.

\begin{lemma}
It holds that 
\begin{align}\label{20241021-yb-DerivativeOfPeriodField}
J_{\varPhi(T)} = I_2 - V \otimes \nabla T
~
\text{ in the open set } \{0<T<+\infty\}.
\end{align}
\end{lemma}

\begin{proof}
By the definition of the orbit period $T$ in \eqref{20241021-yb-PeriodOfOrbits}, we have
\[
x_0 = x(T(x_0);x_0)
~
\text{whenever } T(x_0)\in(0,+\infty).
\]
Since $T$ is $C^1$ on the open set $\{0<T<+\infty\}$ (see Proposition \ref{20250322-yb-proposition-UsefulPropertiesOfOrbitPeriod}), we differentiate both sides of the above identity with respect to $x_0$. This gives
\[
I_2
=
x'(T(x_0);x_0) \otimes \nabla T(x_0)
+
J_{\varPhi(T(x_0))}(x_0)
=
V(x_0) \otimes \nabla T(x_0)
+
J_{\varPhi(T(x_0))}(x_0).
\]
Rearranging the terms yields the desired identity \eqref{20241021-yb-DerivativeOfPeriodField}. 
\end{proof}

The following lemma provides an asymptotic estimate for the flow Jacobian  $\{J_{\varPhi(t)}\}_{t\ge0}$.

\begin{lemma}
With the same notations as in Theorem \ref{20241023-yb-proposition-AsymptoticForJacobianOfVeolocityField}, there exists a constant $C=C(\Omega,V,\varepsilon)>0$ such that for each $x_0\in \Omega_{\ge\varepsilon}\setminus\{V=0\}$,
\begin{align}\label{20241021-yb-AverageOfDeflectedTangentVector}
		& \sup_{t \in \mathbb R}    \Big|
		| V(x(t)) |  \cdot  |V(x_0)|   \cdot \int_0^t   	\Big( 
		\text{curl}\,   \frac{V(y)} {\|V(y)\|^2}
		\Big)\Big|_{y=x(\tau)}   d\tau  
		\nonumber\\
	   &  \quad\quad\quad
	   	+    t  
			| V(x(t)) | \cdot |V(x_0)|^{-1} \cdot \big( 
                 V(x_0)   \times    \nabla  \ln T(x_0)
		\big) 
		\Big|
		\leq C.
		\end{align}
\end{lemma}

\begin{proof}
Fix $x_0\in\Omega_{\ge\varepsilon}\setminus\{V=0\}$.  
The proof is divided into two steps.

\vskip6pt
\noindent
\textit{Step 1. Evaluation of the period integral}

We first show that
\begin{align}\label{20250325-yb-IntegralAndOrbitPeriod}
A:=
\int_0^{T(x_0)}
\Big(
\mathrm{curl}\,\frac{V(y)}{|V(y)|^2}
\Big)\Big|_{y=x(\tau)} d\tau
=
-
\Big(|V|^{-2}V\times\nabla T\Big)\Big|_{x=x_0}.
\end{align}
By the definition of the period function $T(\cdot)$ in \eqref{20241021-yb-PeriodOfOrbits}, we have $x(T(x_0))=x_0$, and hence
\[
V(x(T(x_0)))=V(x_0).
\]
Combining this with \eqref{20241018-yb-SpecialExpressionOfJacobianOfFlow}, we obtain for each $w_0\in\mathbb R^2$
\begin{align*}
J_{\varPhi(T(x_0))}(x_0)w_0
=
w_0
+
\Big[
\int_0^{T(x_0)}
\Big(
\mathrm{curl}\,\frac{V(y)}{|V(y)|^2}
\Big)\Big|_{y=x(\tau)} d\tau
\,\langle w_0,V^\perp(x_0)\rangle
\Big]V(x_0).
\end{align*}
Using \eqref{20241021-yb-DerivativeOfPeriodField} and taking
$
w_0=V^\perp(x_0)/|V(x_0)|,
$
we deduce
\begin{align*}
\int_0^{T(x_0)}
\Big(
\mathrm{curl}\,\frac{V(y)}{|V(y)|^2}
\Big)\Big|_{y=x(\tau)} d\tau
=
- \bigg( \frac{ w_0 \cdot \nabla T }{ |V| } \bigg)\bigg|_{x=x_0}
=
-
\Big(V\times\nabla T/|V|^2\Big)\Big|_{x=x_0},
\end{align*}
which proves \eqref{20250325-yb-IntegralAndOrbitPeriod}.

\vskip6pt
\noindent
\textit{Step 2. Proof of \eqref{20241021-yb-AverageOfDeflectedTangentVector}}

It suffices to prove \eqref{20241021-yb-AverageOfDeflectedTangentVector} for $t\ge0$, since the case $t<0$ can be treated similarly.  
Fix $t\ge0$ and set $k_0 := [t/T(x_0)]$ (i.e., the largest integer  $\leq t/T(x_0)$).  
Since $x(T(x_0))=x_0$, relation \eqref{20250325-yb-IntegralAndOrbitPeriod} implies that for each $k\in\mathbb N$,
\begin{align*}
\int_{kT(x_0)}^{(k+1)T(x_0)}
\Big(
\mathrm{curl}\,\frac{V(y)}{|V(y)|^2}
\Big)\Big|_{y=x(\tau)} d\tau
=
-
\Big(|V|^{-2}V\times\nabla T\Big)\Big|_{x=x_0}.
\end{align*}
With $A$ given in \eqref{20250325-yb-IntegralAndOrbitPeriod}, we have 
\begin{align*}
\int_0^t
\Big(
\mathrm{curl}\,\frac{V(y)}{|V(y)|^2}
\Big)\Big|_{y=x(\tau)} d\tau
=
k_0A
+
\int_{k_0T(x_0)}^{t}
\Big(
\mathrm{curl}\,\frac{V(y)}{|V(y)|^2}
\Big)\Big|_{y=x(\tau)} d\tau .
\end{align*}
Since $| t - k_0 T(x_0) | < T(x_0)$, we denote the left-hand side of \eqref{20241021-yb-AverageOfDeflectedTangentVector} by $B$ and obtain
\begin{align}\label{20250329-yb-InitialEstimateOnB}
B
\le
|V(x(t))|\,|V(x_0)|
\left(
|A|
+ T(x_0)
\sup_{\tau\in[0,T(x_0)]}
\left|
\mathrm{curl}\,\frac{V}{|V|^2}(x(\tau))
\right|
\right).
\end{align}

On the other hand, by Proposition \ref{20250322-yb-proposition-UsefulPropertiesOfOrbitPeriod}, 
there exists a constant $C_1=C_1(\Omega,V,\varepsilon)>0$ such that for all $t,s\in \mathbb R$,
\begin{align*}
\frac{|V(x(t))|}{|V(x(s))|}
+
\Big(
|V|\,|\nabla T|
+
T\|J_V\|_{\mathbb R^{2\times2}}
\Big)\Big|_{x=x(t)}
\le C_1 .
\end{align*}
Combining this with \eqref{20250329-yb-InitialEstimateOnB} yields \eqref{20241021-yb-AverageOfDeflectedTangentVector}.  
The proof is completed.
\end{proof}

We are now ready to prove Theorem \ref{20241023-yb-proposition-AsymptoticForJacobianOfVeolocityField}.

\begin{proof}[Proof of Theorem \ref{20241023-yb-proposition-AsymptoticForJacobianOfVeolocityField}]
By Proposition \ref{20250322-yb-proposition-UsefulPropertiesOfOrbitPeriod}, there exists a constant $C_1=C_1(\Omega,V,\varepsilon)>0$ such that for each $x_0\in\Omega_{\ge\varepsilon}\setminus\{V=0\}$,
\[
\sup_{t,s\in\mathbb R}
\frac{|V(x(t))|}{|V(x(s))|}
\le C_1.
\]
Meanwhile, using that $V \cdot \nabla \ln T =0$, we obtain
\begin{align*}
    w_0 \cdot \nabla \ln T  =  ( V \times \nabla \ln T ) ~ ( w_0 \cdot V^{\perp}) |V|^{-2},
    ~ w_0 \in \mathbb R^2. 
\end{align*}
Now, we determine from Theorem \ref{20241015-yb-Erlangen-theorem-FlowOfTangentVectorsForTimeInvariantVelocityField} and 
\eqref{20241021-yb-AverageOfDeflectedTangentVector} that  
\eqref{20241021-yb-AsymptoticExpansionForJacobianOfFlow} holds.  

Next,  the upper bound in \eqref{20260823-LowerBoundOfJacobianOfFlow} is a direct consequence of \eqref{20241021-yb-AsymptoticExpansionForJacobianOfFlow} as well as \eqref{20250329-yb-GlobalUpperBoundsOfOrbitPeriodAndRatioOfVelocity}.  For the lower bound in \eqref{20260823-LowerBoundOfJacobianOfFlow}, note that $\det J_{\varPhi(t)} (x_0) =1$. Write $\sigma_{sing} \in (0,1]$ and $\sigma_{sing}^{-1} \in [1,+\infty)$ for the  singular values of $J_{\varPhi(t)} (x_0)$. The upper bound in \eqref{20260823-LowerBoundOfJacobianOfFlow} shows 
\begin{align*}
   \max\{ \sigma_{sing}, \sigma_{sing}^{-1}  \}   \leq C(1+|t|).
\end{align*}
Thus, we obtain
\begin{align*}
    | J_{\varPhi(t)} (x_0) w|     \geq \sigma_{sing} |w|  
    \geq  C^{-1} (1+|t|)^{-1}  |w|, ~  w \in \mathbb R^2.
\end{align*}
This gives the lower bound in \eqref{20260823-LowerBoundOfJacobianOfFlow}. 
The proof is completed. 
\end{proof}

\subsection{Growth of singular values}
\label{20250324-yb-section-SingularValues}

In this subsection, we analyze the singular values and the corresponding singular vectors of the Jacobian matrices $\{J_{\varPhi(t)}\}_{t\ge0}$ for large time.

\begin{theorem}\label{20250326-yb-theorem-SingularValuesOfJacobian}
Let $E$ be a nonempty relatively compact subset of the open set
\begin{align*}
    \big\{  x \in \overline{\Omega}   ~:~
    0 < T(x) < +\infty,~
    \nabla T(x) \neq 0   \big\}.
\end{align*}
Then, for each $t \in \mathbb R$ and $x_0\in E$, the matrix $J_{\varPhi(t)}(x_0)$ has two positive singular values
\[
\lambda_{\mathrm{sing}}^{-1}(t,x_0)\in[1,+\infty)
~\text{ and }~
\lambda_{\mathrm{sing}}(t,x_0)\in(0,1],
\]
which satisfy the following uniform estimate on $E$:
\begin{align}\label{20250325-yb-LocallyUniformEstimateOnSmallerSingularValues}
\lim_{t\to+\infty}
\Big[
\lambda_{\mathrm{sing}}(t,x_0)\,
\big(
t \, |\nabla\ln T(x_0)|\,|V(x(t;x_0))|
\big)
\Big]
=1.
\end{align}
Moreover, the singular vector $V_{\mathrm{sing}}$ corresponding to $\lambda_{\mathrm{sing}}$ can be chosen such that for some $C=C(\Omega, V, E) > 0$, 
\begin{align}\label{20250325-SingularVectorAlmostEqualsToV}
|V_{\mathrm{sing}}(t,x_0)-V(x_0)|
\le
C\big( (1+ |t| ) |\nabla\ln T(x_0)|\big)^{-1},
~ t  \in \mathbb R.
\end{align}
\end{theorem}

\begin{remark}\label{remark-SingularValues}
    Variation of the orbit period causes the anisotropic growth of singular values of the flow-map Jacobian $J_{\varPhi(t)}$ (see \eqref{20250325-yb-LocallyUniformEstimateOnSmallerSingularValues}): the larger singular value grows linearly, and its associated singular vector asymptotically aligns with $V^{\perp}$; conversely, the smaller singular value decays at the reciprocal rate, with its singular vector asymptotically aligning with $V$. This is responsible for the following: (i) the upper and lower bounds for $J_{\varPhi(t)}$  in \eqref{20260823-LowerBoundOfJacobianOfFlow}, (ii) the stretching-compression effect (i.e., stretch along the ``bigger" singular vector and compress along the ``smaller" singular vector),  and (iii) the growth of curves and the deformation for open sets under the flow.
\end{remark}

\begin{proof}
Let $t \in \mathbb R$ and $x_0\in E$. For simplicity, write $x(\cdot)$ for the solution $x(\cdot;x_0)$. 
Let $\vec e$ and $U$ be defined in \eqref{20241015-yb-OrthogonalFrame}. 
By \eqref{20241018-yb-SpecialExpressionOfJacobianOfFlow}, we have
\begin{align*}
J_{\varPhi(t)}(x_0)
=
U(x(t))
\begin{pmatrix}
a & b\\
0 & a^{-1}
\end{pmatrix}
U(x_0)^{-1},
\end{align*}
where $a=a(t, x_0)$ and $b=b(t, x_0)$ are given by
\begin{align}\label{20250325-yb-TwoNumbers-a-b}
a:=\frac{|V(x(t))|}{|V(x_0)|},
~\text{ and }~
b:=|V(x(t))|\,|V(x_0)|
\int_0^t
\Big(
\mathrm{curl}\,\frac{V(y)}{|V(y)|^2}
\Big)\Big|_{y=x(\tau)}d\tau .
\end{align}
Consequently,
\begin{align}\label{20250325-yb-SquareOfJacobianMatrix}
J_{\varPhi(t)}^{\top}(x_0)J_{\varPhi(t)}(x_0)
=
U(x_0)
\begin{pmatrix}
a^2 & ab\\
ab & a^{-2}+b^2
\end{pmatrix}
U(x_0)^{-1}.
\end{align}

Since the determinant of this symmetric matrix equals  $1$, the matrix $J_{\varPhi(t)}(x_0)$ has two positive singular values, denoted by $\lambda_{\mathrm{sing}}^{-1}(t,x_0)\in[1,+\infty)$ and $\lambda_{\mathrm{sing}}(t,x_0)\in(0,1]$. 
The smaller singular value satisfies
\[
\lambda^4-(a^2+a^{-2}+b^2)\lambda^2 +1=0,
\]
which yields
\begin{align}\label{20250325-yb-ExpressionOfSmallerSingularValue}
\lambda_{\mathrm{sing}}^2
=
\frac{2}{
(a^2+a^{-2}+b^2)
+
\sqrt{(a^2+a^{-2}+b^2)^2-4}
}.
\end{align}
By \eqref{20250325-yb-TwoNumbers-a-b}, \eqref{20250329-yb-GlobalUpperBoundsOfOrbitPeriodAndRatioOfVelocity}, and \eqref{20241021-yb-AverageOfDeflectedTangentVector}, the quantities $a$ and $a^{-1}$ are uniformly bounded for $x_0\in E$ and $t \in \mathbb R$, while
\begin{align}\label{20250325-yb-AsymptoticOfIntegral}
b^2\sim
\big(
|t| \, |\nabla\ln T(x_0)|\,|V(x(t;x_0))|
\big)^2
~
( |t| \to +\infty)
\end{align}
uniformly over $E$. 
Substituting this into \eqref{20250325-yb-ExpressionOfSmallerSingularValue} gives \eqref{20250325-yb-LocallyUniformEstimateOnSmallerSingularValues}.

Next we determine the corresponding singular vector. 
Assume
\begin{align}\label{20250325-yb-DefineSingularVector}
V_{\mathrm{sing}}(t,x_0)
=
V(x_0)+\alpha V^{\perp}(x_0),
\end{align}
where $\alpha=\alpha(t,x_0)\in\mathbb R$ is to be determined. 
Substituting this into \eqref{20250325-yb-SquareOfJacobianMatrix} yields
\[
\alpha=\frac{\lambda_{\mathrm{sing}}^2 - a^2}{ab}
~~ (|t| \to +\infty).
\]
Using \eqref{20250325-yb-LocallyUniformEstimateOnSmallerSingularValues}, 
\eqref{20250325-yb-AsymptoticOfIntegral}, and \eqref{20250325-yb-TwoNumbers-a-b}, we obtain
\[
|\alpha|
\asymp 
|ab^{-1}|
\asymp
\big(
|t| \, |\nabla\ln T(x_0)|\,|V(x_0)|
\big)^{-1}
~ (|t| \to +\infty),
\]
uniformly for $x_0\in E$. 
Together with \eqref{20250325-yb-DefineSingularVector}, this proves \eqref{20250325-SingularVectorAlmostEqualsToV} for large $t$. For the bounded time $t$, \eqref{20250325-SingularVectorAlmostEqualsToV} can be directly checked. The proof is completed. 
\end{proof}

\section{Proof of Theorem \ref{20241021-yb-theorem-OptimalGrowthForCurves}}
\label{20241018-yb-TimeInvariantVelocityField}

This section is devoted to the proof of Theorem \ref{20241021-yb-theorem-OptimalGrowthForCurves}. 
The argument requires computing the length $|\gamma_t|$ and analyzing the matrices $J_{\varPhi(t)}(\cdot)$ along the initial curve $\gamma_0$. 
With the help of Theorem \ref{20241023-yb-proposition-AsymptoticForJacobianOfVeolocityField}, we now present the proof of Theorem \ref{20241021-yb-theorem-OptimalGrowthForCurves}.

\begin{proof}[Proof of Theorem \ref{20241021-yb-theorem-OptimalGrowthForCurves}]

First, by \eqref{20250402-yb-FintiePeriodAssumptionOnCurve}, it is clear that
\begin{align*}
\gamma_0 \subset \Omega_{\ge D},
\end{align*}
where $\Omega_{\ge D}$ is defined in \eqref{20250323-yb-LargePeirodLayer} (with $s$ replaced by $D$).
Next, we prove \eqref{20241023-yb-AuxsiliaryEstimateForLengthOfCurves}. 
By Proposition \ref{20250322-yb-proposition-UsefulPropertiesOfOrbitPeriod}, there exists a constant $C_1=C_1(\Omega,V,D_*)>0$ (with $D_*=\min\{1,D\}$) such that for each $x_0\in \Omega_{\ge D}\setminus\{V=0\}$,
\begin{align}
T^{-1}(x_0)
+|V(x_0)|\,|\nabla T(x_0)|
+\sup_{t\in\mathbb R}\frac{|V(x(t;x_0))|}{|V(x_0)|}
\le C_1 .
\end{align}
This implies that when $x_0\in\Omega_{\ge D}\setminus\{V=0\}$,
\begin{align*}
|\nabla\ln T(x_0)|\,|V(x(t;x_0))|
&=T^{-1}(x_0)
\big(|\nabla T(x_0)|\,|V(x_0)|\big)
\frac{|V(x(t;x_0))|}{|V(x_0)|}
\le C_1^3,
\end{align*}
which  yields \eqref{20241023-yb-AuxsiliaryEstimateForLengthOfCurves}.

\vspace{4pt}

We now prove \eqref{20241023-yb-AsymptoticForLengthOfCurves}. 
Define the measurable set
\begin{align}\label{20250402-yb-NodalSetOfInitialCurve}
E_0
:=
\big\{\alpha\in[0,1] ~:~ T(\gamma_0(\alpha))=0\big\}
=
\big\{\alpha\in[0,1] ~:~ V(\gamma_0(\alpha))=0\big\},
\end{align}
where the equality of the two sets follows from (i) of Lemma \ref{20250322-yb-propsotion-PropertiesOfOrbitPeriod}.
We claim that
\begin{align}\label{20250402-yb-ZeroDerivativeOfInitialCurve}
\gamma_0'(\alpha)=0
~\text{ for a.e. }\alpha\in E_0 .
\end{align}
To see this, take a Lebesgue point $\alpha_0\in E_0$ of $E_0$  such that $\gamma_0$ is differentiable at $\alpha_0$. Since the set of equilibria of $V$ is finite (see assumption (A2)), these equilibria are isolated. 
Then, by the continuity of $\gamma_0$, it follows from \eqref{20250402-yb-NodalSetOfInitialCurve} that there exists $\delta_0>0$ such that
\begin{align*}
    \gamma_0(\alpha)
    \equiv\text{constant over}\ 
    (\alpha_0-\delta_0,\alpha_0+\delta_0)\cap E_0.
\end{align*}
Since $\gamma_0$ is differentiable at $\alpha_0$, the above implies that $\gamma_0'(\alpha_0)=0$. So almost every point $\alpha_0$ in $E_0$ has this property. Thus, \eqref{20250402-yb-ZeroDerivativeOfInitialCurve} holds.


\vspace{4pt}

Next, for $\alpha\in[0,1]$ we compute
\begin{align}\label{20250402-yb-DerivativesOfEvolvingCurves}
\gamma_t'(\alpha)
=
\frac{d}{d\alpha}\varPhi(t)(\gamma_0(\alpha))
=
J_{\varPhi(t)}(\gamma_0(\alpha))\,\gamma_0'(\alpha).
\end{align}
For a.e.\ $\alpha\in[0,1]\setminus E_0$, applying 
\eqref{20241021-yb-AsymptoticExpansionForJacobianOfFlow}
with
\[
(x_0,x(t;x_0),w_0)
\mapsto
(\gamma_0(\alpha),\gamma_t(\alpha),\gamma_0'(\alpha))
\]
gives a constant $C_2=C_2(\Omega,V,D_*)>0$ such that
\begin{align}\label{20250402-yb-GrowthRateOnNonEquilibrium}
\sup_{t\ge0}
\left|
\gamma_t'(\alpha)
+
t\big(\gamma_0'(\alpha)\cdot\nabla\ln T(\gamma_0(\alpha))\big)
V(\gamma_t(\alpha))
\right|
\le
C_2 |\gamma_0'(\alpha)|.
\end{align}

Finally, using \eqref{20250402-yb-DerivativesOfEvolvingCurves} together with
\eqref{20250402-yb-ZeroDerivativeOfInitialCurve} and
\eqref{20250402-yb-NodalSetOfInitialCurve}, we obtain
\begin{align*}
|\gamma_t|
&=
\int_0^1|\gamma_t'(\alpha)|\,d\alpha
=
\int_0^1
\chi_{\{T\neq0\}}(\gamma_0(\alpha))
|\gamma_t'(\alpha)|\,d\alpha .
\end{align*}
Combining this with \eqref{20250402-yb-GrowthRateOnNonEquilibrium} yields
\eqref{20241023-yb-AsymptoticForLengthOfCurves}. The proof is completed.
\end{proof}

\section{Deformation of the orbit-relative complement}
\label{20250313-yb-subsection-MixingScale-DeformationOfOpenSet}

In this section, we study the deformation of the transported orbit-relative complement $Orbit(A)\setminus A$ of a subdomain $A$ satisfying a uniform-orbit cone condition. As time evolves, this complement becomes increasingly elongated, like a piece of dough being gradually
stretched into a long, thin noodle (see Figure \ref{fig:r2-onecircle-evolution} in Subsection \ref{subsection-SecondTheorem}).



The main result of this section is the following.

\begin{proposition}    \label{rev:prop52}
Let $A\subset\Omega$ be a domain whose closure is not $V$-invariant. 
Assume that $A$ satisfies the following uniform-orbit cone condition: there are numbers $r_0 >0 $ and $\beta_0 \in (0,\pi/2)$ such that for each $x\in A$, there is a point $p \in A\cap Orbit(x)$ and a direction $\nu_p \in \mathbb S^1$ such that
\begin{align}\label{20260908-UniformInteriorConeCondition}
    \Bigl\{
    p+s\xi  ~:~  0<s \leq r_0,  ~\xi\in\mathbb S^1,    ~\xi\cdot\nu_p>\cos\beta_0
    \Bigr\}
    \subset A. 
\end{align}
Suppose also that
        \begin{align}\label{20260901-AssumptionInDeformationTheorem}
            d:=d\big(A,\{T=0, +\infty\}\big)>0
            ~\text{ and }~ 
            \kappa := \inf_{x \in A}  | \nabla T(x)| > 0.
        \end{align}
Write $E:= Orbit(A) \setminus A$ and $E(t):=\varPhi(t)(E)$ ($t \geq 0$). 
Then there is a constant $C=C(\Omega,V,d,\kappa,r_0,\beta_0)>0$ such that
        \begin{align}\label{lip:main}
            C^{-1} (1+t)^{-1}   \sup_{x_0\in E} d(x_0,E^\complement)      \leq
            \rho\bigl(E(t)\bigr)
            \leq
            \rho_A\bigl(E(t)\bigr)
            \leq  C (1+t)^{-1},
            ~~ t\geq0.
        \end{align}
        Here, the functions $\rho(\cdot)$ and $\rho_{A}(\cdot)$ are defined as follows:
        \begin{align*}
            \rho(F):=\sup_{x\in F}d\bigl(x,F^\complement\bigr)
            ~\text{ and }~
            \rho_A(F)
            :=
            \sup_{x\in F}d\bigl(x, F^\complement \cap Orbit(A)  \bigr),
            ~\emptyset \neq F  \subset  Orbit(A). 
        \end{align*}
    
\end{proposition}

\begin{proof}
\csname @show@reffalse\endcsname 
First of all, $E$ is not empty; otherwise, $Orbit(A) = A$ and $A$ is $V$-variant, which leads to a contradiction. 
The proof is organized into the following three steps.

    \medskip
    \noindent\textit{Step 1. To show that for some $C_1=C_1(\Omega,V,d,\kappa) > 0 $, 
    \begin{align}\label{20260901-weight}
        \inf_{x\in Orbit(A)}  
            ~  | \nabla T^{-1}(x) | 
        >  C_1
    \end{align}
    }
    
%
    Observe that
    \begin{align*}
        | \nabla T^{-1} | =  | V | \cdot  |\nabla \ln T |   \cdot ( T |V|)^{-1},
    \end{align*}
    and that $T|V|$ is bounded above by \eqref{20250329-yb-GlobalUpperBoundsOfOrbitPeriodAndRatioOfVelocity}. These yield  \eqref{20260901-weight}. 

    \medskip
    \noindent\textit{Step 2. To  prove the first inequality in \eqref{lip:main}, i.e., for some $C_2=C_2(\Omega,V,d,\kappa)>0$,
    \begin{align}\label{20250326-yb-LowerBoundOfMaximalWidthOfEvolvingSubdomain}
         C_2(1+t)^{-1}
        \sup_{x_0\in E} d(x_0,E^\complement)
        \le
        \sup_{x\in E(t)} d\big(x,E(t)^\complement\big),
        ~ t\ge0
    \end{align}
}
    We estimate the distance from an arbitrary point of $E(t)$ to the boundary of $E(t)$. Fix $\hat x_0\in E(t)$, and let $\hat x_1$ be any boundary point of $E(t)^\complement$. Along the segment from $\hat x_0$ to $\hat x_1$, define $x_*$ to be the first point at which one exits $E(t)$, i.e.,
    \begin{align*}
        x_*:= ( 1 - \alpha_* ) \hat x_0+ \alpha_*\hat x_1
        ~\text{with }~
        \alpha_*:=\inf\Big\{
        \alpha\in[0,1]  ~:~  (1-\alpha) \hat x_0+\alpha \hat x_1\in E(t)^\complement
        \Big\}.
    \end{align*}
    Since $\varPhi(t)$ is a homeomorphism, there exists a Lipschitz curve $\gamma_0:[0,\alpha_*)\to E$ connecting $\varPhi(t)^{-1}(\hat x_0)$ to $\varPhi(t)^{-1}(x_*)$ (a boundary point of $E$) such that
    \begin{align*}
        \varPhi(t)(\gamma_0(\alpha))
        =
        (1-\alpha)\hat x_0+\alpha\hat x_1,
        ~
        \alpha\in[0,\alpha_*).
    \end{align*}
    Hence
    \begin{align}\label{20250326-yb-LowerBoundOfAlternativeDistance}
        d(\hat x_0,\hat x_1)
        \ge
        d(\hat x_0,x_*)
        =
        |\varPhi(t)(\gamma_0)|
        =
        \int_0^{\alpha_*}
        |J_{\varPhi(t)}(\gamma_0(\alpha))\gamma_0'(\alpha)|\,d\alpha.
    \end{align}
    By \eqref{20260823-LowerBoundOfJacobianOfFlow}, there is a constant $C_{11} = C_{11}(\Omega, V, d)  >0$ such that
    \begin{align*}
        |J_{\varPhi(t)}(\gamma_0(\alpha))\gamma_0'(\alpha)|
        \ge
        C_{11} ( 1+ t )^{-1} |\gamma_0'(\alpha)|,
        ~
        \forall\,\alpha\in[0,\alpha_*).
    \end{align*}
    From this and  \eqref{20250326-yb-LowerBoundOfAlternativeDistance}, we have 
    \begin{align*}
        d(\hat x_0,\hat x_1)
        \geq
        C_{11} |\gamma_0| ( 1+ t )^{-1}
        \geq
        C_{11}
        d\Big(\varPhi(t)^{-1}(\hat x_0),E^\complement\Big)
        ( 1+ t )^{-1} .
    \end{align*}
    This implies
    \begin{align*}
        &\sup_{x\in E(t)} d\big(x,E(t)^\complement\big)
        =
        \sup_{\hat x_0\in E(t)}\inf_{\hat x_1\in E(t)^\complement} d(\hat x_0,\hat x_1)
        \\
        \ge& C_{11}   ( 1+ t )^{-1}  \sup_{\hat x_0\in E(t)}
        d\Big(\varPhi(t)^{-1}(\hat x_0),E^\complement\Big)
        = 
        C_{11}(1+t)^{-1}
        \sup_{x_0\in E} d(x_0,E^\complement),
    \end{align*}
    which proves \eqref{20250326-yb-LowerBoundOfMaximalWidthOfEvolvingSubdomain}.

    \vskip 5pt
    \noindent \medskip
    \noindent\textit{Step 3. To prove the last inequality in \eqref{lip:main}, i.e., for some $C_3=C_3(\Omega,V,d,\kappa,r_0,\beta_0)>0$,
        \begin{align}\label{20250326-yb-UpperBoundOfMaximalWidthOfEvolvingSubdomain}
            \sup_{x\in E(t)} d\Big(x,E(t)^\complement\cap Orbit(A)\Big)
            \le
            C_3(1+t)^{-1},
            ~ t\ge0
        \end{align}
    }
    
Fix an arbitrary $x_0 \in E$. Since $E = Orbit(A) \setminus A$, we deduce from \eqref{20260908-UniformInteriorConeCondition} that there is a $y_0 \in Orbit(x_0) \cap A$ and a direction $\nu \in \mathbb S^1$, making an angle of at least $\beta_0  /3 $ with $V(y_0)$, such that 
     \begin{align}\label{20260901-ExteriorLine}
        \gamma_{y_0} := \big\{  y_0 + s \nu ~:~  s \in [0, r_0]  \big\}
        \subset  A =  E^\complement  \cap  Orbit(A). 
     \end{align}
     Without loss of generality, we assume that the angle between $\nu$ and $V$ along the curve $\gamma_{ y_0 }$ is at least $\beta_0/6$, i.e., 
     \begin{align} \label{20260901-UniformAngleAlongExteriorSegment}
          \big| \nu \cdot V( \gamma_{y_0}(s) )  \big|  \leq  
          \cos \frac{\beta_0}{6}  ~  | V( \gamma_{y_0}(s) ) |
          ~\text{ for all }~  s \in [0,r_0].
     \end{align}
     This can be done by shrinking $r_0$. 
     Indeed, by the first inequality in \eqref{20260901-AssumptionInDeformationTheorem} as well as (i) of Lemma \ref{20250322-yb-propsotion-PropertiesOfOrbitPeriod},  we know that $V$ is nonzero on the set $\big\{ x\in \Omega ~:~ d(x, \{T=0,+\infty\}) \geq d \big\}$, and that there is a $C_0 = C_0(\Omega, V, d)>0$ such that
     \begin{align*}
        \Big| \frac{d}{ds}  V( \gamma_{y_0}(s) )  \Big| =  | J_V \gamma_{y_0}'(s) |
        \leq  \| J_V( \gamma_{y_0}(s) )  \|_{\mathbb R^{2 \times 2}}
        \leq C_0 \inf_{ \tau \in [0,s]}  | V( \gamma_{y_0}(\tau) )|,
     \end{align*}
     which implies \eqref{20260901-UniformAngleAlongExteriorSegment} by shrinking $r_0$ to $C_0^{-1}( \cos \frac{\beta_0}{6}  - \cos \frac{\beta_0}{3} )$. 
     
     We now apply the coordinate system in Lemma \ref{rev:uniform-action-angle} to  $Orbit(A)$ (including $\gamma_{ y_0 }$). Since $A$ is connected, so is  $Orbit(A)$ because the flow consists of homeomorphisms $\{ \varPhi(t)\}_{t\in \mathbb R}$. At the same time, by Lemma \ref{20260903-EquivalentDistanceHypothesis} (in the appendix) and \eqref{20260901-AssumptionInDeformationTheorem}, we have
     \begin{align*}
         d\big(Orbit(A),\{T=0, +\infty\}\big)>0
         ~\text{ and }~ 
          \inf_{x \in Orbit(A)}  | \nabla T(x)| > 0.
     \end{align*}
     Thus, the orbit $Orbit(A)$ is contained in a connected subdomain of the set $\mathcal R$ defined in \eqref{20260901-NonTrivalPeriodRegion}. 
     Now, by Lemma \ref{rev:uniform-action-angle} (in the appendix), we can find a $C^1$ coordinate map $\psi$ from   $Orbit(A)$ to $(a,b) \times \mathbb T$ (where $\mathbb T := \mathbb R / (2\pi\mathbb Z)$) for some $a,b \in \mathbb R$. 
     
    In this coordinate chart $(\psi, Orbit(A))$, we write
    \begin{align}\label{Coordinate-20260914}
        \psi(x_0) = (h(0), \alpha_0)   ~\text{ and }~
        \psi(y_0 + s \nu)=(h(s), \alpha(s)), ~ s\in [0,r_0]
    \end{align}
    (here $h(\cdot)$ and $\alpha(\cdot)$ are two $C^1$ functions). For each $t\geq0$, by \eqref{lip:aa} in the appendix, we consider the difference in the second coordinate between $x_0$ and points on $\gamma_{y_0}$ under the action of $\varPhi(t) $: 
    \begin{align}\label{20260901-DefinitionOfPhaseDifference}
        \Theta(s; t) :=& \big( \alpha(s) + \omega(h(s)) t \big) - 
        \big( \alpha_0 + \omega(h(0)) t \big)
        \nonumber\\
        =& \big( \alpha(s)  - \alpha_0 \big)  + 
        \big[ \omega(h(s)) - \omega(h(0))\big] t, ~~ s\in [0,r_0], 
    \end{align}
    where $\omega(\cdot)$ is given by \eqref{lip:aa}. 
    
    Let $C_1$ be given by \eqref{20260901-weight}. Since $\omega(h)= 2\pi/T(x)$ in \eqref{lip:aa}, we have
    \begin{align}\label{20260902-LowerBoundOfFrequency}
        \inf_{ x \in Orbit(A) }   | \omega'( H(x)) |  \cdot  | V(x)| 
        =  2\pi \inf_{ x \in Orbit(A) }   | \nabla T^{-1}(x) | 
         \geq 2 \pi C_1.
    \end{align}
    Fix a sufficiently large time $t$ such that
    \begin{align}\label{20260902-DefineLargeTime}
        t \geq T_{large} :=  \big( 2 \pi  + r_0  \| \alpha' \|_{ L^\infty( (0,r_0) )}  \big) / 
       \big( 2 \pi C_1 r_0 \sin( \beta_0 / 6 )  \big).
    \end{align}
    It follows from  the expression for $J_{ \psi }$ in \eqref{lip:aab}, \eqref{20260901-UniformAngleAlongExteriorSegment} and \eqref{20260902-LowerBoundOfFrequency} that when $s\in[0, r_0]$, 
    \begin{align*}
        | \Theta'(s ; t) | \geq& t | \omega'( h(s)) h'(s)|  - |\alpha'(s)|
        = t | \omega'( h(s)) |  \cdot  | \nu \cdot V^{\perp}(\gamma_{ y_0 }(s) ) |
        - |\alpha'(s)|
        \nonumber\\
        \geq&  t | \omega'( h(s) ) | \cdot |V(\gamma_{ y_0 }(s) ) |
         \sin \frac{\beta_0}{6} 
         -  |\alpha'(s)|
         \nonumber\\
         \geq& \Big( t \sin \frac{\beta_0}{6}  \Big)   
         \inf_{ x \in Orbit(A) }  \Big( | \omega'( H(x) ) | \cdot |V(x) |\Big)
         -  \| \alpha' \|_{ L^\infty( (0, r_0) )} 
         \nonumber\\
         \geq& 2 \pi C_1  \Big( t \sin \frac{\beta_0}{6}  \Big) - \| \alpha' \|_{ L^\infty( (0, r_0) )} 
         \geq  \frac{ 2\pi t }{ r_0 T_{large} }  . 
    \end{align*}
Since $\Theta(\cdot; t)$ is of class $C^1$, the above implies that it is monotone on $[0, r_0]$ and that its range contains an interval of length $\geq 2 \pi$. Take the smallest $s^* \in (0,  r_0 T_{large} t^{-1}]$ such that
    \begin{align}\label{20260901-2PI}
        \Theta( s^* ; t)  \in 2\pi \mathbb Z,
    \end{align}
    which says that the points $x_0$ and $\gamma_{y_0}(s^*)$ have the same angle $\alpha_0 + \omega(h(0)) t $ on the torus $\mathbb T$ under the action of $\varPhi(t)$.

    Next, by \eqref{20260901-2PI},  the coordinates of the points $\varPhi(t)(x_0)$ and $\varPhi(t)( \gamma_{y_0}(s^*) )$ are connected by the following segment along the $h$-axis: 
    \begin{align*}
        \gamma_1(s) := \big( h(s), \alpha_0 + \omega(h(0)) t  \big),  ~  s\in [0, s^*]. 
    \end{align*}
    Write $y(s):= \psi^{-1}( \gamma_1(s) )$ for $ s \in [0, s^*]$. We obtain from the estimate \eqref{20260902-Upper-Lower-Bound-OfCoordinateMap} in the appendix and \eqref{Coordinate-20260914} that for some $C=C(\Omega,V,d, \kappa)$, 
    \begin{align*}
        | \varPhi(t)(x_0) -  \varPhi(t)( \gamma_{y_0}(s^*) ) | \leq& | \psi^{-1}( \gamma_1 ) |
        = \int_0^{s^*}  | J_{ \psi^{-1} } ( \gamma_1(s) )  \gamma_1'(s)|  ds
        \nonumber\\
        \leq C & \int_0^{s^*}   |h'(s)|  ds
        = C |  h(s^*) - h(0)  |
        = C | H(y_0 + s^* \nu) - H(y_0) |.
    \end{align*}
    Combined with  \eqref{20260901-2PI}, the above yields 
    \begin{align*}
        | \varPhi(t)(x_0) -  \varPhi(t)( \gamma_{y_0}(s^*) ) |  \leq  
        C \| \nabla H \|_{L^\infty (\Omega)} s^*
        \leq  C  \| V \|_{L^\infty (\Omega)}  r_0 T_{large} t^{-1}. 
    \end{align*}
    Since $\gamma_{ y_0 } \subset E^\complement \cap Orbit(A)$ (see \eqref{20260901-ExteriorLine}), we  obtain 
    \begin{align*}
        d\big( \varPhi(t)(x_0), E(t)^\complement \cap Orbit(A) \big) \leq | \varPhi(t)(x_0) - \varPhi(t)( \gamma_{y_0}(s^*) )  | 
        \leq   C  \| V \|_{L^\infty (\Omega)}  r_0 T_{large} t^{-1} .
    \end{align*}
    Since $x_0 \in E$ was arbitrary, the above leads to \eqref{20250326-yb-UpperBoundOfMaximalWidthOfEvolvingSubdomain} when $t\geq T_{large}$. 
    
    The conclusion \eqref{20250326-yb-UpperBoundOfMaximalWidthOfEvolvingSubdomain} for small times $t< T_{large}$ follows from the continuity of $\varPhi(t)$. This completes the proof. 
\end{proof}


\begin{remark}\label{20250404-remark-MechanismForDeformationOfSet}
The basic mechanism for the deformation of a transported orbit-relative complement  is as follows: variation of the orbit period supplies the transverse shear between neighboring
periodic trajectories (see the second inequality in assumption  \eqref{20260901-AssumptionInDeformationTheorem}), causing  the larger and smaller singular values of the flow Jacobian to be comparable to $1+t$ and $(1+t)^{-1}$, respectively (see  Remark \ref{remark-SingularValues}).  Roughly speaking, a transported set is stretched along the singular directions associated with the larger singular values and compressed along those associated with the smaller singular values. As a result, the set becomes progressively longer and thinner. See Subsection \ref{subsection-SecondTheorem} for numerical illustrations. 
\end{remark}

\section{Proof of Theorem \ref{20250103-yb-theorem-MixingScale}}
\label{20241018-yb-section-MixingScale}

This section is devoted to the proof of Theorem \ref{20250103-yb-theorem-MixingScale}. We start with the following lemma. 
\begin{lemma}
    \label{rev:core-transition} 
    Let $A\subset\Omega$ be a nonempty open set whose closure is not $V$-invariant. Then the orbit-relative complement $E_A := Orbit(A)\setminus A$ has a nonempty interior and satisfies 
    \begin{align}\label{rev:core-transition-sandwich}
        MixingScale(t,A) =&      
      \sup_{x\in Orbit(A)  }  d\bigl(x, \varPhi(t)(A)\bigr)
      \nonumber\\
    =&        \sup_{x\in \varPhi(t)(E_A)}d\bigl(x,  Orbit(A) \setminus \varPhi(t)(E_A)  \bigr),
        ~~ t \geq 0. 
    \end{align}
   Here, $MixingScale(t,A)$ is defined in  \eqref{20250103-yb-ExactDefinitionOfMixingScale}. 
\end{lemma}

\begin{proof}
    First of all, we have $ Orbit(A)\setminus\overline A  \neq \emptyset$. This follows from Lemma \ref{lemma-20260122-CharacterizeInvariantSet} (in the appendix) and the assumption that the closure of $A$ is not $V$-invariant. 
    
Next, the set $E_A$ has a nonempty interior. In fact, since every $\varPhi(t)$ $(t\geq0)$ is a homeomorphism and $A$ is open, the orbit $Orbit(A) = \cup_{t \in \mathbb R} \varPhi(t)(A)$ is also open. The set $ Orbit(A)\setminus\overline A$ is also open and, as shown above, nonempty. It is clear that $ Orbit(A)\setminus\overline A$ is contained in the interior of $E_A= Orbit(A) \setminus A$. 
    
    It remains to prove \eqref{rev:core-transition-sandwich}. Fix $t\geq 0$. 
    We first verify the second equality in \eqref{rev:core-transition-sandwich}. Since $E_A= Orbit(A) \setminus A$, it is clear that
    \begin{align*}
        Orbit(A) = A  \cup  E_A
        ~\text{ and }~
        A \cap E_A =\emptyset.
    \end{align*}
    Because $Orbit(A)$ is $V$-invariant, the above implies
    \begin{align*}
        Orbit(A) = \varPhi(t)(A)  \cup  \varPhi(t)(E_A)
        ~\text{ and }~
        \varPhi(t)(A) \cap \varPhi(t)(E_A) = \emptyset.
    \end{align*}
    This leads to the second equality in \eqref{rev:core-transition-sandwich}. 
    
    Finally, we prove the first equality in \eqref{rev:core-transition-sandwich}. By \eqref{20250103-yb-ExactDefinitionOfMixingScale} in Definition~\ref{20250103-yb-DefinitionOfMixingScale}, there are positive numbers $\{ \varepsilon_n \}_{n=1}^\infty$ such that for each $n$, the set $\varPhi(t)(A)$ is mixed to scale $\varepsilon_n$ over $Orbit(A)$ and that 
    \begin{align}\label{ApproximateMixingScale-20260909}
       MixingScale(t,A) = \lim_{ n \rightarrow +\infty}  \varepsilon_n.
    \end{align}
    Then, by (ii) of Definition~\ref{20250103-yb-DefinitionOfMixingScale}, we find that for each $n$, 
    \begin{align*}
        | B_{\varepsilon_n}(x) \cap \varPhi(t)(A) | > 0,
        ~~\forall\, x \in Orbit(A),
    \end{align*}
    which implies 
    \begin{align*}
        d(x, \varPhi(t)(A))  \leq  \varepsilon_n,
        ~~\forall\, x \in Orbit(A). 
    \end{align*}
    This, together with \eqref{ApproximateMixingScale-20260909}, gives
    \begin{align}\label{LowerBoundOfMixingScale-20260909}
        \sup_{x\in Orbit(A)  }  d\bigl(x, \varPhi(t)(A)\bigr)
        \leq MixingScale(t,A). 
    \end{align}
    
    Take an arbitrary number $r$ such that
    \begin{align}\label{UpperBoundOfMixingScale-20260909}
       3 \varepsilon := r  - \sup_{x\in Orbit(A)  }  d\bigl(x, \varPhi(t)(A)\bigr)
       > 0. 
    \end{align}
    Then, for each $x \in Orbit(A)$, the ball $B_{r-2\varepsilon}(x)$ must contain at least one point $y_x$ in the set $\varPhi(t)(A)$. This implies
    \begin{align*}
         B_r(x) \cap \varPhi(t)(A)   \supset   B_{\varepsilon}(y_x) \cap \varPhi(t)(A) , 
         ~~\forall\, x \in Orbit(A). 
    \end{align*}
    Then, we apply Lemma \ref{20250322-yb-lemma-ZeroMixingScaleForOpenSets} (with $\mathcal O$ replaced by $\varPhi(t)(A)$) to obtain that for each $x \in Orbit(A)$, 
    \begin{align*}
        | B_r(x) \cap \varPhi(t)(A) |   \geq  | B_{\varepsilon}(y_x) \cap \varPhi(t)(A) |
        \geq \inf_{ y \in \varPhi(t)(A) }   | B_{\varepsilon}(y) \cap \varPhi(t)(A) |
        > 0.
    \end{align*}
    From this and (ii) of Definition~\ref{20250103-yb-DefinitionOfMixingScale}, we know that $\varPhi(t)(A)$ is mixed to scale $r$ over $Orbit(A)$. Then, by \eqref{20250103-yb-ExactDefinitionOfMixingScale} in Definition~\ref{20250103-yb-DefinitionOfMixingScale}, we obtain
    \begin{align*}
        MixingScale(t,A) \leq r.
    \end{align*}
    Because of the arbitrariness of $r$ in \eqref{UpperBoundOfMixingScale-20260909}, the above implies
    \begin{align*}
        MixingScale(t,A)  \leq
        \sup_{x\in Orbit(A)  }  d\bigl(x, \varPhi(t)(A)\bigr). 
    \end{align*}
    This, together with \eqref{LowerBoundOfMixingScale-20260909}, leads to the first equality in \eqref{rev:core-transition-sandwich}. The proof is completed.     
\end{proof}

We are now in a position to prove
Theorem~\ref{20250103-yb-theorem-MixingScale}, using
Lemma~\ref{rev:core-transition}  and 
Proposition \ref{rev:prop52}. 

\begin{proof}[Proof of Theorem~\ref{20250103-yb-theorem-MixingScale}]
Note that the set $A$ is assumed to be a Lipschitz domain. The uniform-orbit cone condition \eqref{20260908-UniformInteriorConeCondition} holds for such a domain $A$. Furthermore, \eqref{20260901-AssumptionInDeformationTheorem} is exactly the assumption \eqref{rev:regular}. Thus, the conclusion \eqref{lip:main} in Proposition \ref{rev:prop52} holds for the set $A$. Then, \eqref{20250112-yb-OrderOfMixingScale} follows from \eqref{rev:core-transition-sandwich} (in Lemma~\ref{rev:core-transition}) and \eqref{lip:main} (in Proposition \ref{rev:prop52}). This completes the proof. 
\end{proof}


\section{Proof of Theorem \ref{third:optimal-negative-norm}}
\label{section-ThirdMainTheorem}

This section is devoted to the proof of Theorem \ref{third:optimal-negative-norm}. We begin with Lemma \ref{20260902-RegularityOfProjection}, which establishes the regularity of the orbit average function defined in \eqref{third:orbit-average}. 

\begin{lemma}\label{20260902-RegularityOfProjection}
    The map $f \mapsto \bar f$ (given by \eqref{third:orbit-average}) is an orthogonal projection on $L^2(\Omega)$. Furthermore, if  a set $A\subset \Omega$ satisfies
    $
        d\big( A, \{T=0, +\infty\} \big) > 0
    $
    and     $f\in H^1(\Omega)$ with supp\,$f \subset A$, then $\bar f \in H^1(\Omega)$. 
\end{lemma}

\begin{proof}
Fix an $f\in C^{\infty}(\overline{\Omega})$. By \eqref{third:orbit-average}, direct computation shows that when $0 < T(x) < +\infty$, 
\begin{align*}
    \big( \bar f(x) \big)^2 =&  \left( 
    \int_{ Orbit(x)  }   f(y)  \frac{d \sigma(y) }{ |T(x)| |V(y)| }
    \right)^2
    \nonumber\\
    \leq&  \int_{ Orbit(x)  }   f(y)^2  \frac{d \sigma(y) }{ |T(x)| |V(y)| }
    \int_{ Orbit(x) }    \frac{d \sigma(y) }{ |T(x)| |V(y)| }
    = \overline{f^2}(x).
\end{align*}
Observe that the set $\big\{ 0 < T(x) < +\infty \big\}$ coincides with almost all of the domain $\Omega$ except for finitely many infinite-period orbits and equilibria (see Lemma \ref{20250322-yb-propsotion-PropertiesOfOrbitPeriod}). 
Then, the last inequality yields 
\begin{align}\label{20260903-SquareNormOfProjection}
    \int_{\Omega} \bar f(x)^2 dx \leq  \int_{\Omega}   \overline{f^2}(x) dx
    = \int_{\Omega}   f^2(x) dx. 
\end{align}
By the density of $C^{\infty}(\overline{\Omega})$ in $L^2(\Omega)$, the map $f \mapsto \bar f$ is linear and bounded. 
At the same time, it follows from \eqref{third:orbit-average} that $\bar{ \bar f}= \bar f$. Thus, the map $f \mapsto \bar f$ is an orthogonal projection on $L^2(\Omega)$. 

Next, we see from \eqref{third:orbit-average} that when $0 < T(x) < +\infty$, 
\begin{align}\label{20260903-EstimateOfH1}
   \nabla \bar f (x) =& \nabla \int_0^1 f\big( \varPhi(s T(x)) (x)\big) ds
   \nonumber\\
   =&  \int_0^1 
   \Big(
        s \nabla T(x) \otimes V\big( \varPhi(s T(x)) (x) \big)    + J_{ \varPhi(s T(x)) }^{\top}(x)
   \Big)
   \nabla f\big( \varPhi(s T(x)) (x)\big) 
    ds.
\end{align}
At the same time, since  $d := d\big( A, \{T=0, +\infty\} \big) > 0$, we see from Lemma \ref{20260903-EquivalentDistanceHypothesis} (in the appendix) that $Orbit(A)$ is also separated from the set $\big\{ T=0, +\infty \}$. This, together with \eqref{20250329-yb-GlobalUpperBoundsOfOrbitPeriodAndRatioOfVelocity} and \eqref{20260823-LowerBoundOfJacobianOfFlow}, implies that there is a constant $C=C(\Omega, V, d)$ such that
\begin{align*}
    \sup_{ y \in Orbit(A) }  \Big[
    \| \nabla T(y) \otimes V(y)  \|_{ \mathbb R^{2\times2} }   +  
    \sup_{0 \leq t \leq T(y)} \|J_{ \varPhi(t) } (y) \|_{ \mathbb R^{2\times2} } 
    \Big] 
    \leq C.
\end{align*}
Since supp\,$f \subset A$, the above, together with \eqref{20260903-EstimateOfH1}, yields
\begin{align*}
    |  \nabla \bar f (x)  |  \leq   C \int_0^1  \big| \nabla f\big( \varPhi(s T(x)) (x)\big)  \big| 
    ds
    =  \displaystyle\frac{C}{T(x)}  \int_0^{T(x)}
    |\nabla f| \bigl(\varPhi(t)(x)\bigr)\,dt.
\end{align*}
Then, by an argument similar to that in \eqref{20260903-SquareNormOfProjection}, we obtain
\begin{align*}
    \int_{\Omega}  |  \nabla \bar f (x)  |^2  dx
    \leq C^2  \int_{\Omega}  |  \nabla  f (x)  |^2  dx. 
\end{align*}
This, along with \eqref{20260903-SquareNormOfProjection}, yields
\begin{align*}
    \| \bar f \|_{ H^1(\Omega) }    \leq  (1+C)  \|  f \|_{ H^1(\Omega) }. 
\end{align*}
Therefore, by the density of the space $C^{\infty}(\overline{\Omega})$ in $H^1(\Omega)$, $\bar f \in H^1(\Omega)$ holds for general $f \in H^1(\Omega)$ with supp\,$f\subset A$. This completes the proof.     
\end{proof}

We are now in a position to prove Theorem \ref{third:optimal-negative-norm}. 
\begin{proof}[Proof of Theorem \ref{third:optimal-negative-norm}]
 First, we enlarge $A$ to a $V$-invariant open and connected subset of the closure $\overline{\Omega}$ containing $Orbit(A)$ as a relatively compact subset. 
 Indeed,  since $A$ is connected, the assumption \eqref{20260902-domain} also holds for the connected sets $\overline{A}$ and  $Orbit(\overline{A})$ (by Lemma \ref{20260903-EquivalentDistanceHypothesis} in the appendix).  Take a subdomain  $A_1 \subset \overline{\Omega}$ to be the $\varepsilon$-neighborhood of $Orbit(\overline{A})$ for small $\varepsilon>0$. Then, by  Lemma \ref{20260903-EquivalentDistanceHypothesis}, \eqref{20260902-domain} also holds with $A$ replaced by the connected orbit $\mathcal O := Orbit(A_1)$, i.e., 
 \begin{align}\label{20260903-EstimateOfO}
    \tilde{d} := d\big(\mathcal O,\{T=0, +\infty\}\big)>0
     ~\text{ and }~
     \tilde{\kappa} := \inf_{ x \in \mathcal O}  | \nabla T(x) | > 0.
 \end{align}
 The rest of the proof is organized into the following two steps.

\vskip 5pt
\noindent\textit{Step 1. To show the lower bound in \eqref{third:sharp-rate}}

Take an arbitrary function $\varphi \in C_0^{\infty}( \mathcal O)$. Define
\begin{align*}
    \varphi_t := \varphi \circ \varPhi(-t)  \in  H_0^1( \mathcal O). 
\end{align*}
By \eqref{20260823-LowerBoundOfJacobianOfFlow} as well as \eqref{20260903-EstimateOfO}, there is a constant $C_1 = C_1(\Omega, V, d) >0 $ such that
\begin{align}\label{20260902-EstimateOfSobolevFunction}
    \int_{\Omega}   | \nabla \varphi_t |^2 dx 
    \leq  \sup_{x \in \mathcal O} \| J_{\varPhi(-t)} (x) \|^2  
    \int_{\Omega}   | \nabla \varphi |^2 dx
    \leq  C_1^2 (1+t)^2 \int_{\Omega}   | \nabla \varphi |^2 dx.
\end{align}
Since the flow $\{ \varPhi(t)\}_{ t\in \mathbb R}$ preserves measure, we obtain 
\begin{align*}
    \int_{\Omega}  \big( \theta_0(x) - \bar\theta_0(x) \big) \varphi(x) dx
    =& \int_{\Omega}  \big[ \theta_0\big( \varPhi(-t)(x) \big)  - \bar\theta_0(x) \big] \varphi\big(  \varPhi(-t)(x) \big) dx
    \nonumber\\
    =& \int_{\Omega}  \big[ \theta(t,  x)  - \bar\theta_0(x) \big] \varphi_t(x) dx
    \leq  \| \theta(t)  - \bar\theta_0 \|_{ H^{-1}(\Omega)}   
    \| \varphi_t  \|_{ H_0^1(\Omega) } .
\end{align*}
Then, it follows from \eqref{20260902-EstimateOfSobolevFunction} that
 for each $\varphi \in C_0^{\infty}( \mathcal O)$, 
\begin{align*}
    \int_{\Omega}  \big( \theta_0(x) - \bar\theta_0(x) \big) \varphi(x) dx
    \leq  C_1 (1+t) \| \theta(t)  - \bar\theta_0 \|_{ H^{-1}(\Omega)}   
    \| \varphi  \|_{ H_0^1(\Omega) }.
\end{align*}
Because supp\,$(\theta_0 - \bar\theta_0)   \Subset \mathcal O$, the above implies
 \begin{align*}
     \| \theta_0  - \bar\theta_0 \|_{ H^{-1}(\Omega)}  
     \leq  C_1 (1+t)  \| \theta(t)  - \bar\theta_0 \|_{ H^{-1}(\Omega)}.
 \end{align*}
 This proves the lower bound in \eqref{third:sharp-rate}. 

\vskip 5pt
\noindent\textit{Step 2. To show the upper bound in \eqref{third:sharp-rate}}

We apply the coordinate system introduced in Lemma \ref{rev:uniform-action-angle} (in the appendix) to $E=\mathcal O$. Let the coordinate chart $(\psi, \mathcal O)$ from $\mathcal O$ onto $(a,b) \times \mathbb T$ be given by Lemma \ref{rev:uniform-action-angle}. 

Since $V \in C^2( \overline{ \Omega} )$, by an argument similar to that in (iii) of Lemma \ref{20250322-yb-propsotion-PropertiesOfOrbitPeriod}, one can show that the orbit-period function $T(\cdot) $ is of class $C^2$ at $x \in \overline{\Omega}$ with $T(x) \in (0, +\infty)$. Then, the function $\omega$ in \eqref{lip:aa} is of class $C^2$.  Furthermore, we derive from \eqref{20260902-Upper-Lower-Bound-OfCoordinateMap} (in the appendix) and \eqref{20260903-EstimateOfO} that
\begin{align}\label{20260903-UpperLowerBounds}
    \| J_{\psi^{-1}} \|_{ C( (a,b) \times \mathbb T)} +
    \| \omega \|_{ C^2((a,b))}   < +\infty
    ~\text{ and }~
    \inf_{ h \in (a,b)} \min\big\{ |\omega(h)|  , |\omega'(h)| \big\} > 0.
\end{align}

Denote by $\hat f$ the representation of $f$ in the coordinate chart $(\psi, \mathcal O)$. 
 The following formula holds for each $f \in L^2(\mathcal O)$: when $x=\psi^{-1}(h,\theta) $,
\begin{align}\label{20260902-AnEqualityInAngle}
    \int_{\mathbb T} \hat f(h, \theta_1) d\theta_1 =
    |\omega(h)|
    \int_{ H^{-1}(h) \cap \mathcal O }
    \frac{ f(y) }{ |V(y)| }
    d\sigma(y) 
    =|\omega(h)| \int_0^{ T(x)}  f\big( \varPhi(t)(x) \big) dt.
\end{align}
In fact, we obtain from \eqref{lip:aa}--\eqref{lip:aab} (in the appendix) that for each $g \in C_0^{\infty}(a,b)$, 
\begin{align*}
    \int_a^b\Big( \int_{\mathbb T} \hat f(h, \theta_1) d\theta_1 \Big)
    g(h) dh =& \int_a^b \int_{ \mathbb T}  \hat f (h,\theta_1)  g(h)  d \theta_1 dh
    = \int_{ \mathcal O} f(x)  g(H(x))
    \cdot |\det J_\psi(x)|   dx
    \nonumber\\
    =& \int_a^b |\omega(h)| g(h) 
    \int_{ H^{-1}(h) \cap \mathcal O }
    f(x)
    \frac{ d\sigma(x)  }{ |\nabla H(x)| }
    dh
    \nonumber\\
    =& \int_a^b |\omega(h)| g(h) 
    \int_{ H^{-1}(h) \cap \mathcal O }
    f(x)
    \frac{ d\sigma(x)  }{ |V(x)| }
    dh.
\end{align*}
This gives \eqref{20260902-AnEqualityInAngle}. 

Now, we estimate $\theta -\bar\theta_0$ in the coordinate chart $(\psi, \mathcal O)$. For this purpose, let $t \geq 1$ and fix an arbitrary $ \varphi \in C_0^{\infty}( \mathcal O)$. 
We see from \eqref{lip:aa}--\eqref{lip:aab} (in the appendix) that 
\begin{align}\label{20260902-ExpressionInActionAngleForm}
    \int_{\Omega}  (\theta - \bar\theta_0)(t,x)  \varphi(x) dx 
    =& \int_a^b  \int_{ \mathbb T}   
    \widehat{(\theta - \bar\theta_0)}(t, h, \theta)  
    \widehat\varphi 
    (h,\theta) |\det J_{\psi^{-1}} | d \theta dh
    \nonumber\\
    =& \int_a^b  \int_{ \mathbb T}   \widehat{(\theta_0 - \bar\theta_0)}
    \big( h, \theta - t \omega(h)  \big)  
        \widehat{\varphi}(h,\theta)  
    |\omega^{-1}(h)|
    d \theta dh.
\end{align}
We apply the Fourier transform to the variable $\theta$. Write
\begin{align} \label{20260902-FourierSeriesOfTheta}
 \widehat{(\theta_0 - \bar\theta_0)}(h,\theta) =  \sum_{ k \in \mathbb Z \setminus \{0\}}  a_k(h) e^{ i k \theta} 
 ~\text{ and }~
 \widehat{\varphi} (h,\theta)
 = \sum_{ k \in \mathbb Z  }  b_k(h) e^{ i k \theta},
 ~~(h,\theta)  \in (a,b) \times  \mathbb T,  
\end{align}
where $a_k(\cdot), b_k(\cdot)$ are the Fourier coefficients depending on the variable $h$. Here, the index $ k \in \mathbb Z \setminus \{0\}$ is taken due to the following fact: 
\begin{align*}
    a_0(h) = \frac{1}{2 \pi}\int_{\mathbb T} \widehat{(\theta_0 - \bar\theta_0)}(h,\theta) d\theta
    = \frac{ |\omega(h)| }{ 2\pi } \int_0^{ T(x)}  (\theta_0 - \bar\theta_0)\big( \varPhi(t)(x) \big) dt
    =0
\end{align*}
(here we used \eqref{20260902-AnEqualityInAngle} in the second equality and \eqref{third:orbit-average} in the last equality).

Now, it follows from \eqref{20260902-ExpressionInActionAngleForm} and \eqref{20260902-FourierSeriesOfTheta} that
\begin{align*}
    \int_{\Omega}  ( \theta(t) - \bar\theta_0) \varphi dx  
    =&  2\pi \sum_{ k \in \mathbb Z \setminus \{0\} }
      \int_a^b
    a_k(h)  b_{-k}(h)  |\omega^{-1}(h)|  e^{ - i k \omega(h) t }
    dh
    \nonumber\\
    =& 2\pi \sum_{ k \in \mathbb Z \setminus \{0\} }
    \int_a^b
    a_k(h)  b_{-k}(h)    \frac{i}{kt |\omega(h)| \omega'(h)} \partial_h e^{ - i k \omega(h) t }
    dh
    \nonumber\\
    =& 2\pi\sum_{ k \in \mathbb Z \setminus \{0\} }   
          \frac{i a_k b_{-k} e^{ - i k \omega  t }}{kt \omega'  |\omega|}
    \Big|_{h=a}^b
     + 2\pi \sum_{ k \in \mathbb Z \setminus \{0\} }
    \int_a^b
    \partial_h \Big( a_k  b_{-k}    \frac{1}{ikt  \omega'  |\omega|} 
    \Big)
      e^{ - i k \omega t }
    dh.
\end{align*}
At the same time, since $\varphi \in C_0^{\infty}(\mathcal O)$ vanishes on $\{ H(x)=a,b \}$, it follows from \eqref{20260902-AnEqualityInAngle} that when $h=a,b$, 
\begin{align*}
    b_{-k}(h) = \frac{1}{ 2\pi } \int_{ \mathbb T}  \widehat\varphi(h,\theta) e^{ik \theta} d\theta
   = \frac{ |\omega(h)| }{ 2\pi }
   \int_{ H^{-1}(h) \cap \mathcal O }
   \frac{ \varphi(x) e^{ik \theta(x)} }{ |V(x)| }
   d\sigma(x) 
   =0, ~~ k \in \mathbb Z.
\end{align*}
Therefore, the last two equalities give 
\begin{align}\label{20260902-EqualityOfFrequency}
    \int_{\Omega}  ( \theta(t) - \bar\theta_0) \varphi dx 
    =   2\pi\sum_{ k \in \mathbb Z \setminus \{0\} }
    \frac{1}{ikt  } \int_a^b
    \partial_h \Big( a_k  b_{-k}    \frac{1}{  |\omega| \omega' } 
    \Big)
    e^{ - i k \omega t }
    dh.
\end{align}

We compute that for each $k \in \mathbb Z \setminus \{0\}$, 
\begin{align*}
     \Big|  \partial_h \Big( a_k b_{-k}   \frac{1}{ |\omega| \omega'}   \Big)
     \Big| 
    \leq  \Big( |a_k'|  |b_{-k}|  + |a_k| | b_{-k}' | \Big) 
   \frac{2}{ |(\omega^2)' | } 
   + |a_k| |b_k|   \frac{2 | (\omega^2)''| }{ |(\omega^2)'  |^2 }. 
\end{align*}
Thus, we obtain from \eqref{20260902-EqualityOfFrequency} that for each $k \in \mathbb Z \setminus \{0\}$, 
\begin{align*}
    \int_{\Omega}  ( \theta(t) - \bar\theta_0) \varphi dx 
    \leq&  \frac{4 \pi}{t  } \sup_{h \in (a, b) }
    \frac{ |\omega| \big( |(\omega^2)' | + | (\omega^2)''| \big)  }{ |(\omega^2)'  |^2 } 
    \nonumber\\
    & \times 
    \left(
       \int_a^b  \Big(\| (a_k(h))_{k\in \mathbb Z}\|^2   +
       \| (a_k'(h))_{k\in \mathbb Z}\|^2
       \Big)  \frac{dh}{ |\omega(h)| }
    \right)^{1/2}
    \nonumber\\
    &  \times 
    \left(
    \int_a^b  \Big(\| (b_k(h))_{k\in \mathbb Z}\|^2   +
    \| (b_k'(h))_{k\in \mathbb Z}\|^2
    \Big) \frac{dh}{ |\omega(h)| }
    \right)^{1/2}.
\end{align*}
This, along with \eqref{20260903-UpperLowerBounds} and \eqref{20260902-FourierSeriesOfTheta}, yields that for some $C=C(\Omega, V, d, \kappa)$,
\begin{align*}
    \int_{\Omega}  ( \theta(t) - \bar\theta_0) \varphi dx 
    \leq&  \frac{C}{t  }   \| \theta_0 - \bar\theta_0 \|_{ H^1(\Omega) }
    \| \varphi \|_{ H^1(\Omega) }.
\end{align*}
Since $\varphi \in C_0^{\infty}( \mathcal O)$ and $Orbit(A) \Subset \mathcal O$, the above leads to the upper bound in \eqref{third:sharp-rate} when $t\geq 1$. The case $t \in [0,1)$ is clear. This completes the proof. 
\end{proof}

\section{Examples illustrating the roles of the assumptions}
\label{section-ExamplesOnAssumptions}

In this section we give  counterexamples concerning several assumptions in
the three main theorems.    The conclusions in
the corresponding theorems fail when the indicated hypothesis is
removed, while the other hypotheses listed there are kept.  These counterexamples are not
intended to assert that every assumption is logically necessary for
every individual velocity field. 

This section is organized as follows: Subsections \ref{subsec:assumption-FirstMainTheorem}, \ref{subsec:assumption-SecondMainTheorem}, and \ref{subsec:assumption-ThirdMainTheorem}  justify the assumptions in Theorems~\ref{third:optimal-negative-norm}, \ref{20250103-yb-theorem-MixingScale}, and \ref{20241021-yb-theorem-OptimalGrowthForCurves}, respectively.

\subsection{Assumptions in Theorem~\ref{third:optimal-negative-norm}}
\label{subsec:assumption-FirstMainTheorem}

We will present some counterexamples for specific velocity fields to illustrate the failure of Theorem~\ref{third:optimal-negative-norm} when one of its assumptions is removed. 

\begin{example}[Vanishing lower bound for the gradient of orbit periods]
    \label{example:VanishingPeriodGradient-MixingNorm-20260910}
    We provide a counterexample to \eqref{third:sharp-rate} when the second inequality in assumption \eqref{20260902-domain} fails. 
    Let $\Omega := B_1(0) $.  Let $r_* \in (0,1)$ and let $g \in C^{\infty}([0,+\infty))$ be positive except at the origin: 
    \begin{align}\label{SpecialPeriod-20260910}
        g'\big( (r_*+r)^2 \big) \asymp r^{1}    ~\text{ and }~
        g(r^2) \asymp r^{m}   \text{ as } r \rightarrow 0^+
        ~\text{ for some }
        m \in \mathbb N. 
    \end{align} 
    Consider the following Hamiltonian flow:
    \begin{align*}
        H(x) := \frac{1}{2} \int_{|x|^2}^1 g( r ) dr, ~ x \in \mathbb R^2  ~\text{ and }~
        V := - \nabla^{\perp} H |_{ \overline{\Omega} }. 
    \end{align*} 
    By direct computation, $V$ has only one (stable) equilibrium $0$ of order $m+1$, and 
    \begin{align}\label{OrbitPeriod:VanishingPeirodGradient}
        T(x) = \frac{ 2\pi }{ g(|x|^2) }
        ~\text{ and }~
        | \nabla T(x) | = \frac{ 4\pi |x| \cdot  | g'(|x|^2) |  }{ g^2(|x|^2) } , ~ x\in \Omega\setminus\{0\}. 
    \end{align}

    We construct an initial datum in the following way: fix a sufficiently small number $\varepsilon>0$ and let  $A:= \{ x \in \Omega ~:~ r_* - \varepsilon < |x| < r_* + \varepsilon\}$. By \eqref{OrbitPeriod:VanishingPeirodGradient} and  \eqref{SpecialPeriod-20260910}, it is clear that 
    \begin{align*}
        \inf_{x \in A} | \nabla T(x) | =0.
    \end{align*}
     Take a function $a \in C_0^{\infty} \big( (r_* - \varepsilon, r_* + \varepsilon) \big)$ with $a(r_*) \neq 0$. Consider the following initial data supported in $A$:
    \begin{align*}
        \theta_0(x) := a( |x| )  \cos (\arg x)
        ~\text{ and }~
        \theta_1(x) := a( |x| )  \sin (\arg x),  ~ x\in \Omega. 
    \end{align*}
    By \eqref{third:orbit-average}, the orbit average of both $\theta_0,\theta_1$ are zero. 
    Write $\omega(x) := 2 \pi / T(x) = g(|x|^2)$ (by \eqref{OrbitPeriod:VanishingPeirodGradient}). 
    By direct computation, when $ t \geq 0$, 
    \begin{align*}
        \theta(t; \theta_0 ) = a( |x| )  \cos \big( \arg x - t \omega(x) \big)
        ~\text{ and }~
        \theta(t; \theta_1 ) = a( |x| )  \sin \big(\arg x - t \omega(x) \big),  ~ x\in \Omega,
    \end{align*}
    which implies
    \begin{align*}
        \int_{\Omega}  \big[ \theta(t; \theta_0 ) + i \theta(t; \theta_1 ) \big] 
        \cdot  \big[ \theta_0  - i  \theta_1 \big]
        dx
        =& \int_{\Omega}   a(|x|) e^{ i (\arg x - t \omega(x) )}  
        a(|x|) e^{ - i \arg x} dx.
    \end{align*}
    This yields that as $t \rightarrow +\infty$, 
    \begin{align*}
         \| \theta(t; \theta_0 ) + i \theta(t; \theta_1 ) \|_{ H^{-1}(\Omega; \mathbb C)}
         \gtrsim  
         \bigg| \int_{ r_* - \varepsilon}^{r_* + \varepsilon}  a^2(r)  e^{ - i  t g(r^2) }  rdr
         \bigg|
         \asymp t^{ - \frac{1}{2} }.
    \end{align*}
    Here, we used \eqref{SpecialPeriod-20260910}.  Therefore, the upper bound in \eqref{third:sharp-rate} cannot hold in the current situation.     
\end{example}

\begin{example}[Approach to a stable equilibrium]
    \label{Counterexample-StableEquilibrium:MixingNorm}
We provide a counterexample to \eqref{third:sharp-rate} when the support of an initial datum approaches a stable equilibrium. 
Let $\Omega := B_1(0) $. Take a number $m > 3$ and consider the following Hamiltonian flow:
\begin{align*}
    H(x) := |x|^{m+1} -1, ~ x \in \mathbb R^2  ~\text{ and }~
    V := - \nabla^{\perp} H. 
\end{align*} 
The velocity field $V|_{\overline{\Omega} }$ has only one (stable) equilibrium (the origin) of order $m$. The assumptions (A1)-(A2) hold in the current case. Furthermore,
\begin{align}\label{ExpressionOfOrbitPeriod-StableEquilibrium}
    T(x) = \frac{ 2\pi }{ (m+1) |x|^{m-1} }
    \text{ and }
    |\nabla T(x)| =  \frac{(m-1)2\pi}{ (m+1)|x|^m },
    ~x\in \Omega \setminus\{0\}.
\end{align}
The following set 
\begin{align*}
    A := \big\{
    x \in \Omega ~:~  |\arg x|  <  \pi/ 2
    \big\}
\end{align*}
satisfies $\inf_{ x \in A}  | \nabla T(x)| > 0$. It has the stable equilibrium $0$ as a boundary point. 

Let $\varepsilon\in(0,1)$. We choose an initial datum in the following way: take two nonzero functions $f \in C_0^{\infty}((1/2, 1))$ and $\alpha \in C_0^{\infty}((-\pi/2, \pi/2))$ with $\int_{\mathbb R} \alpha(x) dx =0$. Choose the following initial datum supported in the set $A$: 
\begin{align*}
    \theta_{0, \varepsilon} (x) :=  f( |x| / \varepsilon ) \, \alpha( \arg x ),  ~ x \in \mathbb R^2. 
\end{align*}
Here, $\arg x$ denotes the angle of the vector $x$ from the first coordinate axis. 
By \eqref{third:orbit-average}, one can check that 
$ \bar\theta_{0, \varepsilon} = 0$. 

In what follows, we extend equation \eqref{20240925-yubiao-MainEquation} to $\Omega = \mathbb R^2$.  
Note that $V$ is homogeneous of order $m$. Then, the function  $(t,x) \mapsto \theta( \varepsilon^{-(m-1)} t, \varepsilon x; \theta_{0, \varepsilon})$ satisfies  equation \eqref{20240925-yubiao-MainEquation} with the initial datum $\theta_{0,1}$ and $\Omega = \mathbb R^2$. 
In particular,
\begin{align*}
    \theta( \varepsilon^{-(m-1)} , \varepsilon x; \theta_{0, \varepsilon})
    = \theta(1, x; \theta_{0,1}),  ~ x\in \mathbb R^2. 
\end{align*}
Write $t_{\varepsilon} := \varepsilon^{-(m-1)} $. By direct computation, we have
\begin{align*}
    \| \theta( t_{\varepsilon}, \cdot; \theta_{0, \varepsilon}) \|_{ H^{-1}( B_1(0))} = \varepsilon^2 \| \theta(1, \cdot; \theta_{0,1})  \|_{ H^{-1}( B_{ \frac{1}{\varepsilon} }(0) ) }
    ~\text{ and }~
    \| \theta_{0, \varepsilon} \|_{ H_0^{1}( B_1(0) )} = \|  \theta_{0,1}  \|_{ H_0^{1}( B_{ \frac{1}{\varepsilon} }(0) ) }.
\end{align*}
Since $\Omega = B_1(0)$, we then obtain
\begin{align*}
    \| \theta( t_{\varepsilon}, \cdot; \theta_{0, \varepsilon}) \|_{ H^{-1}(\Omega)} / \| \theta_{0, \varepsilon} \|_{ H^{1}(\Omega)}
    \asymp  \varepsilon^2
    ~\text{ as }
    \varepsilon \rightarrow 0^+. 
\end{align*}
Because $t_{\varepsilon}^{-1} = \varepsilon^{m-1}$ and $m > 3$, the above implies that the upper bound in \eqref{third:sharp-rate} does not hold when the support of an initial datum approaches the stable equilibrium. 
\end{example}

\begin{example}[Approach to infinite-period orbits]
    \label{Counterexample-InfinitePeriod:MixingNorm}
    Let
    $\Omega := (0,\pi)^2$. Let $m \in \{1\} \cup (2,+\infty)$ and consider the cellular-type flow
    \[
    H(x)
    :=|(\sin x_1, \sin x_2)|^{m-1}\sin x_1\sin x_2,
    ~x=(x_1,x_2) \in \mathbb R^2
    ~\text{ and }~
    V:=- \nabla^{\perp} H|_{ \overline{\Omega}}.
    \]
    The only critical points of $H$ in
    $\overline\Omega$ are the center $(\pi/2,\pi/2)$ and the four corners.  
    The assumptions (A1)--(A2) hold with $m_k=1$ at the
    center (stable equilibrium) and $m_k=m$ at the corners (unstable equilibria).  The boundary of $\Omega$, except for its corners, consists of the infinite-period orbits $\{ T=+\infty \}$. 
    
    We compute the period of the orbit $\{ H(x) = h \}$  with sufficiently small $h \in (0,1)$. Take two numbers $\alpha_h \asymp h$ and $\beta_h  \asymp h^{ \frac{1}{m+1} }$ in $ (0, \frac{\pi}{2} )$ satisfying
    \begin{align*}
        H(\frac{\pi}{2}, \alpha_h) = h   ~\text{ and }~
        H(\beta_h, \beta_h) = h.
    \end{align*}
    By the symmetry of $H$, the period  of the level set $\{ H(x) = h \}$ equals $8$ times the time required to travel on the segment from point $(\beta_h, \beta_h)$ to point $(\frac{\pi}{2}, \alpha_h)$: 
    \begin{align*}
        T(x) =& 8 \int_{ \beta_h }^{ \frac{\pi}{2} } \frac{ \sqrt{1+|\frac{ dx_2 }{ dx_1 }|^2} dx_1 }{ |V(x_1,x_2)|}
        = 8 \int_{ \beta_h }^{ \frac{\pi}{2} } \frac{  dx_1 }{ | \partial_2 H(x_1,x_2) | }
        \nonumber\\
        =&  8 h^{-1}  \int_{ \beta_h }^{ \frac{\pi}{2} }     
        \frac{ \sin^2 x_1 + \sin^2 x_2 }{ \sin^2 x_1 + m \sin^2 x_2 }
        \tan x_2 
        d x_1
        \asymp   h^{-1}  \int_{ \beta_h }^{ \frac{\pi}{2} }     
        \tan x_2 
        d x_1.
    \end{align*}
    On this segment, $\tan x_2 \sim  \sin x_2  \asymp h/\sin^m x_1$ for small $h$. 
    By direct computation, we obtain that as $h=H(x) \rightarrow 0^+$, 
    \begin{align}\label{OrbitPeriod:InfiniteOrbit-FirstTheroem-20260911}
        T(x) \asymp \begin{cases}
           h^{ - \frac{m-1}{m+1} }, &~ m >1,
            \vspace{0.5em}\\
            \ln \frac{1}{h}, &~m=1
        \end{cases}
        ~\text{ and }~
        \Big| \frac{ d T(x) }{ dH(x)} \Big|  \asymp h^{ - \frac{2m}{m+1} }. 
    \end{align}

    The following set with small $\delta$
    \begin{align}\label{SetA:InfinitePeriod-FirstTheorem-20260912}
        A:= \big( \frac{\pi}{3}, \frac{2\pi}{3} \big) \times (0, \delta)
    \end{align}
    includes a part of the infinite-period orbits (i.e., $(0,\pi) \times \{0\}$)  as its boundary. 
    The construction of a sequence of initial data is as follows. Set  
    \begin{align*}
        x_h := \big( \frac{\pi}{2}, \alpha_h \big), ~ t_h:=T(x_h),
        \text{ and }
        v_h := \frac{ V(x_h) }{ |V(x_h)| }.
    \end{align*}
    Since $ V \cdot \nabla T =0$, it follows from   \eqref{20241021-yb-DerivativeOfPeriodField} that 
    \begin{align*}
        v_h + |V(x_h)| \nabla T(x_h) = \big[ J_{\varPhi(t_h)}(x_h)^{\top} \big]^{-1}  v_h.
    \end{align*}
    After small perturbations, there is a small neighborhood $U_h\Subset A$ of $x_h$ such that
    \begin{align}\label{PerburbationProperties-20260911}
        \big| \big[ J_{\varPhi(t_h)}(x)^\top \big]^{-1}  v_h \big|   >  \frac{ |V(x_h)| \cdot | \nabla T(x_h) |  }{2}
        ~\text{and}~
        | V(x) \cdot v_h| > \frac{ |V(x_h)| }{2} ,
        ~ x \in U_h.
    \end{align}
    Take $a \in C_0^{\infty}(U_h) \setminus \{0\}$ and define a sequence of real initial data $\{(\theta_{0,k}, \theta_{1,k})\}_{k=1}^{\infty}$ supported in $A$ as follows:
    \begin{align}\label{SpecialInitialData:InifiteOrbit-FirstTheorem-20260911}
        (\theta_{0,k} + i \theta_{1,k})(x) := V(x)\cdot \nabla \big( a(x) e^{i k v_h\cdot x } \big),
        ~ x \in \Omega. 
    \end{align}
   By \eqref{third:orbit-average}, one can directly check that their orbit averages are all $0$ (i.e., $\bar\theta_{0,k} = \bar\theta_{1,k} =0$). 
    
 The following equality holds: for each $f,\varphi\in C_0^{\infty}(\Omega) $ with $\nabla \varphi \neq 0$ on supp\,$f$, 
 \begin{align*}
     \lim_{ k \rightarrow +\infty } \|  \big( k e^{i k \varphi } \big) f \|_{H^{-1}(\Omega; \mathbb C)}^2  =   
     \int_{\Omega}  f^2 | \nabla \varphi |^{-2} dx.
 \end{align*}
 The key observation behind this equality is that the difference $ - \Delta \big( k^{-1} e^{ik\varphi} f / |\nabla\varphi|^2 \big) - \big( k e^{i k \varphi } \big) f$ is   $o(1)$ in $H^{-1}(\Omega; \mathbb C)$ (as $k \rightarrow +\infty$), which implies
  \begin{align*}
     \lim_{ k \rightarrow +\infty } \big\|  \big( k e^{i k \varphi } \big) f \big\|_{H^{-1}(\Omega; \mathbb C)}  =   
     \lim_{ k \rightarrow +\infty }  \big\| k^{-1} e^{ik\varphi} f / |\nabla\varphi|^2 \big\|_{H_0^1(\Omega; \mathbb C)}
     = \big\| f / |\nabla\varphi| \big\|_{L^2(\Omega)},
 \end{align*}
 here the details are omitted. Now we apply this equality to \eqref{SpecialInitialData:InifiteOrbit-FirstTheorem-20260911} to obtain
 \begin{align*}
     \lim_{ k \rightarrow +\infty }  \| \theta_{0,k} + i \theta_{1,k} \|_{ H^{-1}(\Omega; \mathbb C)}^2
     =& \int_{\Omega}  | a(x) V(x) \cdot v_h |^2 dx,
     \nonumber\\
     \lim_{ k \rightarrow +\infty }  \| \theta(t_h;\theta_{0,k}) + i \theta(t_h; \theta_{1,k}) \|_{ H^{-1}(\Omega; \mathbb C)}^2
     =& \lim_{ k \rightarrow +\infty }  \| ( \theta_{0,k} + i \theta_{1,k}) \circ \varPhi(t_h)^{-1} \|_{ H^{-1}(\Omega; \mathbb C)}^2
     \nonumber\\
     =& \int_{\Omega}  \frac{   | a(x) V(x) \cdot v_h |^2 }{
        \big| \big[ J_{\varPhi(t_h)}(x)^\top \big]^{-1}  v_h \big|^2
     } dx.
 \end{align*}
 This, together with \eqref{PerburbationProperties-20260911} and \eqref{OrbitPeriod:InfiniteOrbit-FirstTheroem-20260911}, yields that as $h = H(x_h) \rightarrow 0^+$, 
 \begin{align*}
     \lim_{ k \rightarrow +\infty }  \frac{
        \| \theta(t_h;\theta_{0,k}) + i \theta(t_h; \theta_{1,k}) \|_{ H^{-1}(\Omega; \mathbb C)}
     }{
        \| \theta_{0,k} + i \theta_{1,k} \|_{ H^{-1}(\Omega; \mathbb C)}
     }
     \leq   \frac{2}{ |V(x_h)| \cdot | \nabla T(x_h) | } 
     \asymp  h^{ \frac{2m}{m+1} } ,
 \end{align*}
 which is of smaller order than $t_h^{-1}=\frac{1}{ T(x_h) }$. Therefore, the lower bound in \eqref{third:sharp-rate} cannot hold in the current situation.         
\end{example}

\subsection{Assumptions in Theorem~\ref{20250103-yb-theorem-MixingScale}}
\label{subsec:assumption-SecondMainTheorem}

We will present some counterexamples for specific velocity fields to illustrate the failure of Theorem~\ref{20250103-yb-theorem-MixingScale} when one of its assumptions is removed. 

\begin{example}[Vanishing lower bound for the gradient of orbit periods]
    \label{example:VanishingPeriodGradient-MixingScale-20260910}
    We provide a counterexample to \eqref{20250112-yb-OrderOfMixingScale} when the second inequality in assumption \eqref{rev:regular} fails. Consider the same velocity field $V$ (with the same notations) as in Example \ref{example:VanishingPeriodGradient-MixingNorm-20260910}. Set
    \begin{align*}
        A_1 := \Big\{  x \in \Omega  ~:~   
        r_* - \varepsilon < |x| <r_* + \varepsilon,
        ~ |\arg x| <  \frac{ \pi }{ 2 }  
        \Big\}. 
    \end{align*}
    Let $ t \geq 1$.     According to \eqref{OrbitPeriod:VanishingPeirodGradient} and  \eqref{SpecialPeriod-20260910}, we can take $\delta_t \asymp t^{- \frac{1}{2} }  \in (0, \frac{r_*}{2})$ such that the angular velocity in the annulus $\mathfrak R_{\delta_t} := \big\{  r_* - \delta_t < |x| <r_* + \delta_t \big\}$ minus $\omega_* := g(r_*^2)$ is less than $\frac{\pi}{3t}$, i.e.,
    \begin{align*}
        \Big| \frac{2\pi}{ T(x) }  - \omega_*  \Big| = \big| g( |x|^2 )  -g(r_*^2) \big| <  \frac{\pi}{3t}
        ~\text{ when }   x \in \mathfrak R_{\delta_t}.
    \end{align*}
    Thus, the set $\varPhi(t)(A_1)$ does not enter the region $\mathfrak R_{\delta_t} \cap \big\{  \frac{5\pi}{6} < \arg x - t \omega_* < \frac{7\pi}{6} \big\}$. So
    \begin{align*}
        B_{ \delta_t }  \big( (r_* \cos( \pi + t \omega_* ), r_* \sin( \pi + t \omega_* ) ) \big)  \cap  \varPhi(t)(A_1) = \emptyset. 
    \end{align*}
    Then, by (iii) of Definition \ref{20250103-yb-DefinitionOfMixingScale}, we know
    \begin{align*}
        MixingScale(t,A_1) \geq  \delta_t \asymp t^{ - \frac{1}{2} }
        ~\text{ as }~  t\rightarrow +\infty.
    \end{align*}
    This implies that the upper bound in \eqref{20250112-yb-OrderOfMixingScale} (with $A$ replaced by $A_1$) does not hold in the current situation. 
\end{example}

\begin{example}[Approach to a stable equilibrium]
    We provide a counterexample to \eqref{20250112-yb-OrderOfMixingScale} when the set $A$ approaches a stable equilibrium. 
    Consider the same velocity field $V$ and the set $A$ as in Example \ref{Counterexample-StableEquilibrium:MixingNorm}.  Let $ t \geq 1$ and define
    \begin{align*}
        r_t := \frac{1}{2} \Big( \frac{ \pi }{ 3(m+1)t }  \Big)^{ \frac{1}{m-1} }.
    \end{align*}
    By \eqref{ExpressionOfOrbitPeriod-StableEquilibrium}, the angular velocity in the ball $B_{2 r_t}(0)$ is less than $\frac{\pi}{3t}$, i.e.,
    \begin{align*}
        \frac{2\pi}{ T(x) } = (m+1) |x|^{m-1}  <  \frac{\pi}{3t}
        ~\text{ when }  |x| < 2 r_t.
    \end{align*}
    Thus, the set $\varPhi(t)(A)$ does not enter the region $B_{2r_t}(0) \cap \big\{  \frac{5\pi}{6} < \arg x < \frac{7\pi}{6} \big\}$. So
    \begin{align*}
        B_{ \frac{r_t}{2} }  \big( (-r_t,0) \big)  \cap  \varPhi(t)(A) = \emptyset. 
    \end{align*}
    Then, by (iii) of Definition \ref{20250103-yb-DefinitionOfMixingScale}, we know
    \begin{align*}
        MixingScale(t,A) \geq  \frac{r_t}{2} \asymp t^{ - \frac{1}{m-1} }
        ~\text{ as }~  t\rightarrow +\infty.
    \end{align*}
    Since $m>3$, the above implies that the upper bound in \eqref{20250112-yb-OrderOfMixingScale} does not hold when the set $A$ approaches the stable equilibrium $0$. 
\end{example}

\begin{example}[Approach to infinite-period orbits]
    We provide a counterexample to \eqref{20250112-yb-OrderOfMixingScale} when the set $A$ approaches an infinite-period orbit. 
    Consider the same velocity field $V$ and the set $A$ (see \eqref{SetA:InfinitePeriod-FirstTheorem-20260912}) as in Example \ref{Counterexample-InfinitePeriod:MixingNorm} and let $m>2$.  The corner $0$ of $\Omega$ is an unstable equilibrium of order $m>2$. By Lemma \ref{20241001-yb-lemma-Hamiltonian}, we have that $|V(x)| \asymp |x|^m$ as $x \rightarrow 0$. 
    There are numbers $r>0$ and $C>0$ such that
    \begin{align}\label{OrderOfVelocity-20260912}
        | V(x) |  \leq  C |x|^m, ~ x \in B_{r}(0) \cap \Omega.
    \end{align}
    Let $t\geq 1$. Choose a number $r_t$ in the following way:
    \begin{align*}
        r_t := \min\Big\{
            \frac{1}{2}  \big( C(m-1)t \big)^{ -\frac{1}{m-1} }, \,
            \frac{\pi}{6}, \,  \frac{r}{2}
        \Big\}. 
    \end{align*}
    We estimate the evolution of points in $B_{r_t}(0)$. 
    Write $R(s;x) := | \varPhi(-s)(x)|$, $s\geq 0$, $x \in B_{r_t}(0) \cap \Omega$. Then, it follows from \eqref{OrderOfVelocity-20260912} that $R'(s;x) \leq C R(s;x)^m$ in $B_r(0)$ and that
    \begin{align*}
        R(s;x) \leq \frac{ |x| }{
            \big[ 1 - C(m-1) |s| |x|^{m-1}  \big]^{ \frac{1}{m-1} }
        }
        < 2 r_t
        \leq \frac{\pi}{3},
        ~ s\in(0, t),~x\in B_{r_t}(0) \cap \Omega. 
    \end{align*}
    Then, we see from the definition of $A$ in \eqref{SetA:InfinitePeriod-FirstTheorem-20260912} that
    \begin{align*}
        \varPhi(-t) B_{r_t}(0) \subset B_{ \frac{\pi}{3} }(0) 
        \subset A^\complement,
        ~\text{ i.e., }
         B_{r_t}(0)  \cap \varPhi(t)(A) = \emptyset. 
    \end{align*}
    Meanwhile, since the boundary of  $Orbit(A)$ includes $\partial\Omega$, one has $B_{ \frac{r_t}{3} }\big( ( \frac{r_t}{3}, \frac{r_t}{3}) \big) \subset Orbit(A)$ for large $t$. 
    Now, by (iii) of Definition \ref{20250103-yb-DefinitionOfMixingScale}, we know
    \begin{align*}
        MixingScale(t,A) \geq \frac{ r_t }{3} \asymp t^{ - \frac{1}{m-1} }
        ~\text{ as }~  t\rightarrow +\infty.
    \end{align*}
    Since $m>2$, the above implies that the upper bound in \eqref{20250112-yb-OrderOfMixingScale} does not hold in the current situation. 
\end{example}

\subsection{Assumption in Theorem~\ref{20241021-yb-theorem-OptimalGrowthForCurves}}
\label{subsec:assumption-ThirdMainTheorem}

We will present some counterexample for specific velocity fields to illustrate the failure of Theorem~\ref{20241021-yb-theorem-OptimalGrowthForCurves} when the curve $\gamma_0$ lies on the infinite-period orbits. 

\begin{example}
    \label{Counterexample-InfinitePeriod:Length}
    Consider the same velocity field $V$  as in Example \ref{Counterexample-InfinitePeriod:MixingNorm}.  Let $m=1$. We are in the cellular flow case: 
    \begin{align*}
        V(x_1, x_2) = (\sin x_1 \cos x_2,  -\cos x_1 \sin x_2), ~(x_1,x_2) \in \overline{ \Omega }.
    \end{align*}
    In particular, $V(x_1, 0) = (\sin x_1, 0)$ on the $x_1$-axis. Each point on this part of the axis evolves according to the ODE: $x_1'(t) = \sin x_1(t)$, whose solution satisfies $x(t) = f(t, x(0))$, where 
    \begin{align*}
        f(t,x) := 2 \arctan( e^t \tan(x/2) ), ~ t\geq 0, ~ x \in (0, \pi).
    \end{align*}
    Take the Lipschitz curve 
    \begin{align*}
        \gamma_0(s) := \big( s^2\big( 1 + \sin (\pi s^{-1})\big), 0 \big)
        ~s\in (0,1];
        ~~\gamma_0(0) =(0,0). 
    \end{align*}
    Furthermore, with $s_n:=( 3/2+ 2n)^{-1}$ and $p_n:=(1/2 + 2n)^{-1}$ ($n \in \mathbb N$), we have
    \begin{align*}
        \gamma_0(s_n)=(0,0) ~\text{and}~ \gamma_0(p_n) = (2 p_n^2,0),
        ~ n \in \mathbb N.
    \end{align*}
    Let $t\geq 2$. Over each $[s_{n}, s_{n-1}]$, the curve $\gamma_0$ goes from $0$ to a point at distance greater than $2p_n^2$ from the origin and then returns to $0$.     When $1 \leq n \leq \frac{1}{4} e^{t/2}$, 
    \begin{align*}
        |\varPhi(t)( \gamma_0 |_{ [s_{n}, s_{n-1}] } ) |
        \geq 2 f(t, 2 p_n^2)  \geq 4 \arctan(e^t p_n^2) \geq \pi. 
    \end{align*}
    In particular, at the fixed time $t$, the ratio of the length of a transported segment compared to its original one can be very large: 
    \begin{align*}
         \lim_{n\rightarrow+\infty}  \frac{  |\varPhi(t)( \gamma_0 |_{ [s_{n}, p_n] } ) | }{
         |\gamma_0 |_{ [s_{n}, p_n] }  |	
         } 
         = \lim_{n\rightarrow+\infty} \frac{  f(t, 2 p_n^2) }{
              2 p_n^2	
        } = e^t. 
    \end{align*}
    Now, we deduce that as $t\rightarrow +\infty$,
    \begin{align*}
         |\varPhi(t)( \gamma_0) |  \geq \sum_{ 1 \leq n \leq \frac{1}{4} e^{t/2} }  |\varPhi(t)( \gamma_0 |_{ [s_{n}, s_{n-1}] } ) |
          \gtrsim e^{t/2}.
    \end{align*}
    The length of $\gamma_0$ under evolution increases exponentially. Therefore, \eqref{20241023-yb-AsymptoticForLengthOfCurves} does not hold in the current situation.     
\end{example}

\section{Further studies on mixing time and set separation}
\label{20250313-yb-subsection-MixingScale-ProofsOfCorollaries}

In this section, we present two important corollaries of Theorem \ref{20250103-yb-theorem-MixingScale}. The first is an analog of Bressan's conjecture (see \cite[p.\,101]{Bressan-2006}) for stationary divergence-free vector fields (see Corollary \ref{20250110-yb-corollary-AnalogOfBressanConjecture}). Note that Bressan's original conjecture concerns time-dependent flows. The second concerns the extent to which two disjoint evolving sets, sharing the same orbit, approach each other as time becomes large (see Corollary \ref{20250110-yb-corollary-DistanceBetweenTwoSets}).

\begin{corollary}\label{20250110-yb-corollary-AnalogOfBressanConjecture}
Assume the same hypotheses as in Theorem \ref{20250103-yb-theorem-MixingScale}. Set
\begin{align}\label{20250301-yb-OptimalMixingTime}
    t^*_{\varepsilon} := \inf
    \big\{
        t > 0 ~:~ V \text{ mixes } A \text{ to scale } \varepsilon
        \text{ at time } t
    \big\},
    ~ \forall\, \varepsilon > 0
\end{align}
(here the infimum is understood to be $+\infty$ if the corresponding set is empty). Then, it holds that
\begin{align*}
t^*_{\varepsilon}   \asymp \varepsilon^{-1}  ~ \text{ as } \varepsilon \to 0^+ .
\end{align*}
\end{corollary}

This corollary indicates that if the mixing scale decays like 
$1/t$, then the time needed to reach scale $\varepsilon$ 
 is of order $\varepsilon^{-1}$. 

\begin{proof}[Proof of Corollary \ref{20250110-yb-corollary-AnalogOfBressanConjecture}]
    By Theorem \ref{20250103-yb-theorem-MixingScale}, we have
    \begin{align}\label{20250307-yb-proof-TwoSidedEstimateOfMixingScale}
        C_1 (1 + t)^{-1} \leq MixingScale(t,A) \leq C_2 (1 + t)^{-1},
        ~ t \geq 0,
    \end{align}
    where $C_1$ and $C_2$ are the constants in \eqref{20250112-yb-OrderOfMixingScale}.
    Let 
    \begin{align}\label{20250307-yb-RangeOfEpsilon}
        0 < \varepsilon < C_1 / 2.
    \end{align}
    Let $t_1 \in (0, C_1 /(2\varepsilon))$ be arbitrary. By \eqref{20250307-yb-proof-TwoSidedEstimateOfMixingScale} and \eqref{20250307-yb-RangeOfEpsilon},
    \begin{align*}
        MixingScale(t_1, A) \geq C_1 (1 + t_1)^{-1} > \varepsilon.
    \end{align*}
    By \eqref{20250103-yb-ExactDefinitionOfMixingScale}, this implies that $V$ cannot mix $A$ to scale $\varepsilon$ at time $t_1$. Thus, by \eqref{20250301-yb-OptimalMixingTime},
    $
    t^*_{\varepsilon} \geq t_1.
    $
    Since $t_1$ was arbitrary in $(0, C_1 /(2\varepsilon))$, we conclude that
    \begin{align}\label{2025030y-yb-LowerBoundOfOptimalMixingTime}
        t^*_{\varepsilon} \geq C_1 /(2\varepsilon).
    \end{align}
    
    Next, set $t_2 := C_2 / \varepsilon$. Again by \eqref{20250307-yb-proof-TwoSidedEstimateOfMixingScale} and \eqref{20250307-yb-RangeOfEpsilon},
    \begin{align*}
        MixingScale(t_2, A) \leq C_2 (1 + t_2)^{-1} < \varepsilon.
    \end{align*}
    Then, by \eqref{20250103-yb-ExactDefinitionOfMixingScale} as well as \eqref{20250310-yb-DensityForMixing}, $V$ can mix $A$ to scale $\varepsilon$ at time $t_2$. Therefore, by \eqref{20250301-yb-OptimalMixingTime},
    \[
    t^*_{\varepsilon} \leq t_2 = C_2 / \varepsilon.
    \]
    Since both $C_1$ and $C_2$ are independent of $\varepsilon$, the above inequality together with \eqref{2025030y-yb-LowerBoundOfOptimalMixingTime} yields the desired conclusion.
\end{proof}

\begin{corollary}\label{20250110-yb-corollary-DistanceBetweenTwoSets}
Let $T$ be given by \eqref{20241021-yb-PeriodOfOrbits}. 	
Let $A,B\subset  \Omega$ be two disjoint Lipschitz subdomains (open and connected) sharing the same orbit, i.e.,
\[
Orbit(A)=Orbit(B).
\]
Assume the same hypotheses for $A$ as in Theorem \ref{20250103-yb-theorem-MixingScale}.
Then
\begin{align}\label{20250308-yb-DistanceBetweenTwoEvolvingSubdomains}
     d_H\big( \varPhi(t)(A), \varPhi(t)(B) \big) \asymp t^{-1}
    ~ \text{as } t \to +\infty.
\end{align}
Here, $d_H$ denotes the Hausdorff distance between sets, i.e.,
\begin{align*}
	d_H(E,F) := \max\Big\{ \sup_{x \in E} d(x,F),  \sup_{y \in F} d(y,E) \Big\}.
\end{align*}
\end{corollary}

%

\begin{proof}
To prove \eqref{20250308-yb-DistanceBetweenTwoEvolvingSubdomains}, it suffices to show that there exist constants $\widehat C_1, \widehat C_2>0$ such that
\begin{align}\label{20250308-yb-UpperBoundOfDistance}
	\widehat C_1 (1 + t)^{-1} \leq
	\sup_{x \in \varPhi(t)(B)} d(x, \varPhi(t)(A))
	\leq \widehat C_2 (1 + t)^{-1},
	~ t \geq 0.
\end{align}
Once this is established, we can obtain a similar two-sided inequality by exchanging $A$ and $B$ in \eqref{20250308-yb-UpperBoundOfDistance}, and \eqref{20250308-yb-DistanceBetweenTwoEvolvingSubdomains} then follows from these two inequalities. Here, we used that $B$ also satisfies the same hypotheses  as in Theorem \ref{20250103-yb-theorem-MixingScale}.  This holds due to the following:  because $A$ satisfies these, by Lemma \ref{20260903-EquivalentDistanceHypothesis}, so does $Orbit(B)=Orbit(A)$, which implies the desired conclusion. 

We begin with the upper bound in \eqref{20250308-yb-UpperBoundOfDistance}. By Theorem \ref{20250103-yb-theorem-MixingScale},
\begin{align}\label{20250307-yb-proof-CopyOfTwoSidedEstimateOfMixingScale}
	MixingScale(t,A) \leq C_2 (1 + t)^{-1},
	~ t \geq 0,
\end{align}
where $C_2$ is the constant in \eqref{20250112-yb-OrderOfMixingScale}.
Fix $t \geq 0$. By \eqref{20250307-yb-proof-CopyOfTwoSidedEstimateOfMixingScale} and \eqref{20250103-yb-ExactDefinitionOfMixingScale}, the vector field $V$ mixes $A$ to scale $2 C_2 / (1+t)$ at time $t$. Hence, by \eqref{20250310-yb-DensityForMixing},
\begin{align*}
	\inf_{x \in Orbit(A)} | B_{2 C_2 / (1+t)}(x) \cap \varPhi(t)(A) | > 0.
\end{align*}
Since $Orbit(B)=Orbit(A)$ and $\varPhi(t)(B)\subset Orbit(B)$, it follows that
\begin{align*}
	\inf_{x \in \varPhi(t)(B)} | B_{2 C_2 / (1+t)}(x) \cap \varPhi(t)(A) |
	&\geq
	\inf_{x \in Orbit(B)} | B_{2 C_2 / (1+t)}(x) \cap \varPhi(t)(A) |
	\\
	&=
	\inf_{x \in Orbit(A)} | B_{2 C_2 / (1+t)}(x) \cap \varPhi(t)(A) |
	> 0.
\end{align*}
Since $A$ is open and $\varPhi(t)$ is a homeomorphism, the set $\varPhi(t)(A)$ is open. Therefore, the last inequality implies that each point $x \in \varPhi(t)(B)$ lies at distance at most $2 C_2 (1+t)^{-1}$ from $\varPhi(t)(A)$, i.e.,
\begin{align*}
	\sup_{x \in \varPhi(t)(B)} d(x, \varPhi(t)(A)) \leq 2 C_2 (1 + t)^{-1}.
\end{align*}
Since $t \geq 0$ was arbitrary, this proves the upper bound in \eqref{20250308-yb-UpperBoundOfDistance} with $\widehat C_2 = 2C_2$.

Next, we turn to the lower bound in \eqref{20250308-yb-UpperBoundOfDistance}. Observe that $B$ is open and nonempty. In particular, there is an open ball $B_r \subset B$. Since $A\cap B =\emptyset$, to prove \eqref{20250308-yb-UpperBoundOfDistance}, it suffices to show that for some $C_3>0$, 
\begin{align}\label{LowerBoundOfInradiusOfInnerB-20261009}
    \sup_{ z \in \varPhi(t)(B_r) }  d\big(z, [\varPhi(t)(B_r)]^\complement  \big)  \geq  C_3 (1+t)^{-1},
    ~ t \geq 0.
\end{align}
To this end, note that  $A$ satisfies assumption \eqref{rev:regular} in Theorem \ref{20250103-yb-theorem-MixingScale}. By Lemma \ref{20260903-EquivalentDistanceHypothesis}, the set $Orbit(B) = Orbit(A)$ also satisfies \eqref{rev:regular}. Then, by an approach similar to that in the proof of
\eqref{20250326-yb-LowerBoundOfMaximalWidthOfEvolvingSubdomain}, the estimate \eqref{LowerBoundOfInradiusOfInnerB-20261009} is established. This ends the proof. 
\end{proof}

\section{Numerical simulations}
\label{section-NumericalSimulations}

In this section, we present numerical simulations illustrating Theorems~\ref{third:optimal-negative-norm}, \ref{20250103-yb-theorem-MixingScale}, and \ref{20241021-yb-theorem-OptimalGrowthForCurves}, respectively. All ODE systems are integrated with the second-order implicit midpoint scheme, which is symplectic for Hamiltonian systems.

\subsection{Example illustrating Theorem~\ref{third:optimal-negative-norm}}
\label{subsection-MixNorm}

The purpose of this example is to illustrate the inverse-time decay in
Theorem~\ref{third:optimal-negative-norm}. Let
$\Omega=(0,1)\times(0,1)$ and consider the Hamiltonian function
\begin{align}\label{eq:r2-cellular-flow}
H(x_1,x_2) :=\sin(\pi x_1)\sin(\pi x_2), ~(x_1, x_2) \in \mathbb R^2
~\text{ and }~
V :=-\nabla^\perp H|_{ \overline{\Omega}}.
\end{align}


We take the initial distribution
\begin{align*}
\theta_0(x_1,x_2)
:=\chi\bigl(H(x_1,x_2)\bigr)
\tanh\left(\frac{x_2-0.5}{0.01}\right),
~(x_1,x_2) \in \Omega,
\end{align*}
where $\chi\in C_c^\infty((0,1))$ is a smooth cutoff satisfying
$\chi(h)=1$ for $h\in[0.08,0.90]$. Thus $\theta_0$ is supported in a
compact regular energy band. Moreover, the reflection
\begin{align*}
R(x_1,x_2):=(x_1,1-x_2)
\end{align*}
preserves each orbit, while
$\theta_0\circ R=-\theta_0$. Hence, the orbit average satisfies $\bar\theta_0=0$, and
Theorem~\ref{third:optimal-negative-norm} predicts inverse-time decay of
$\|\theta(t)\|_{H^{-1}(\Omega)}$.

We compute the solution by a tensor sine Galerkin method. Writing
\begin{align*}
\theta_N(t,x_1,x_2)
&=\sum_{m,n=1}^N a_{mn}(t)\,
\Big( 2 \sin(m\pi x_1)\sin(n\pi x_2)  \Big),
~t \geq 0, 
~(x_1, x_2) \in \Omega, 
\end{align*}
we obtain the mix-norm
\begin{align*}
\|\theta_N(t)\|_{H^{-1}(\Omega)}^2
&=\sum_{m,n=1}^N
\frac{|a_{mn}(t)|^2}{\pi^2(m^2+n^2)},
~ t \geq 0.
\end{align*}
We use $N=768$ modes in each coordinate and $\Delta t=1/400$. The
Galerkin matrix is skew-symmetric, and no spectral filter is applied.

Figure~\ref{fig:r2-mixnorm-evolution} shows the evolution of the scalar at six different time instants. As time evolves, the initial transition is stretched into increasingly fine alternating filaments along the closed streamlines.

Figure~\ref{fig:r2-mixnorm-decay} shows the mix-norm and the fitted power law. A log--log least-squares fit of $Ct^{-p}$ over the interval
$3.00\leq t\leq10.00$ gives $C=0.07$ and $p=0.99$.
Thus the numerical mix-norm exhibits polynomial decay with an exponent close
to $-1$. The nearly linear tail in the double-logarithmic plot and the fitted exponent $-0.99$ are consistent with the inverse-time order in Theorem~\ref{third:optimal-negative-norm}.

\begin{figure}[H]
  \centering
  \begin{subfigure}[b]{0.16\textwidth}
    \includegraphics[width=\textwidth]{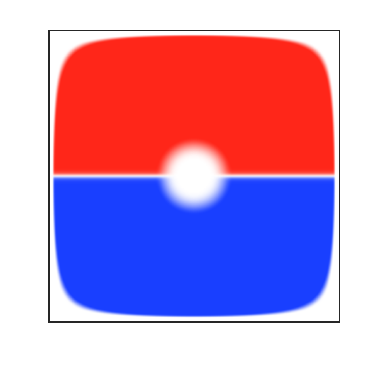}
    \caption{$t=0$}
  \end{subfigure}
  \begin{subfigure}[b]{0.16\textwidth}
    \includegraphics[width=\textwidth]{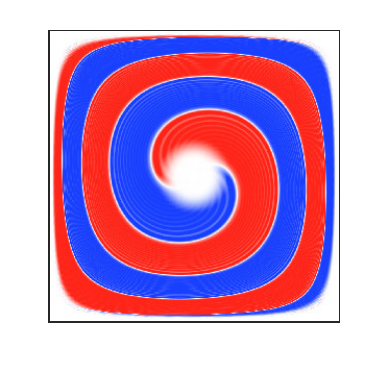}
    \caption{$t=2$}
  \end{subfigure}
  \begin{subfigure}[b]{0.16\textwidth}
    \includegraphics[width=\textwidth]{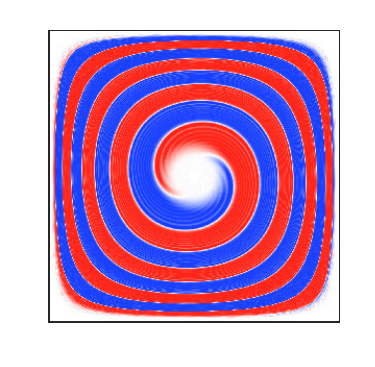}
    \caption{$t=4$}
  \end{subfigure}
  \begin{subfigure}[b]{0.16\textwidth}
    \includegraphics[width=\textwidth]{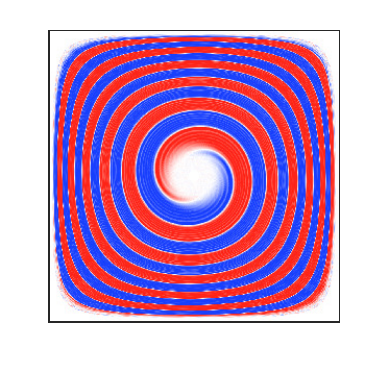}
    \caption{$t=6$}
  \end{subfigure}
  \begin{subfigure}[b]{0.16\textwidth}
    \includegraphics[width=\textwidth]{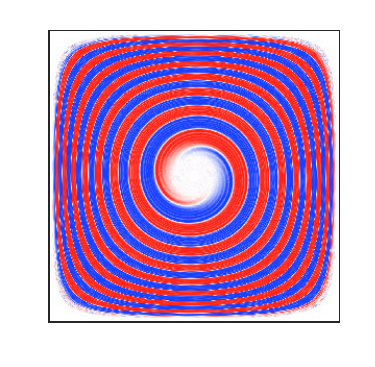}
    \caption{$t=8$}
  \end{subfigure}
  \begin{subfigure}[b]{0.16\textwidth}
    \includegraphics[width=\textwidth]{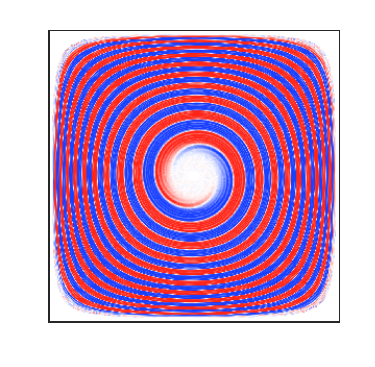}
    \caption{$t=10$}
  \end{subfigure}
  \caption{Evolution of the scalar field $\theta$ at different time instants}
  \label{fig:r2-mixnorm-evolution}
\end{figure}

\begin{figure}[htp]
  \centering
  \includegraphics[width=0.5\textwidth]{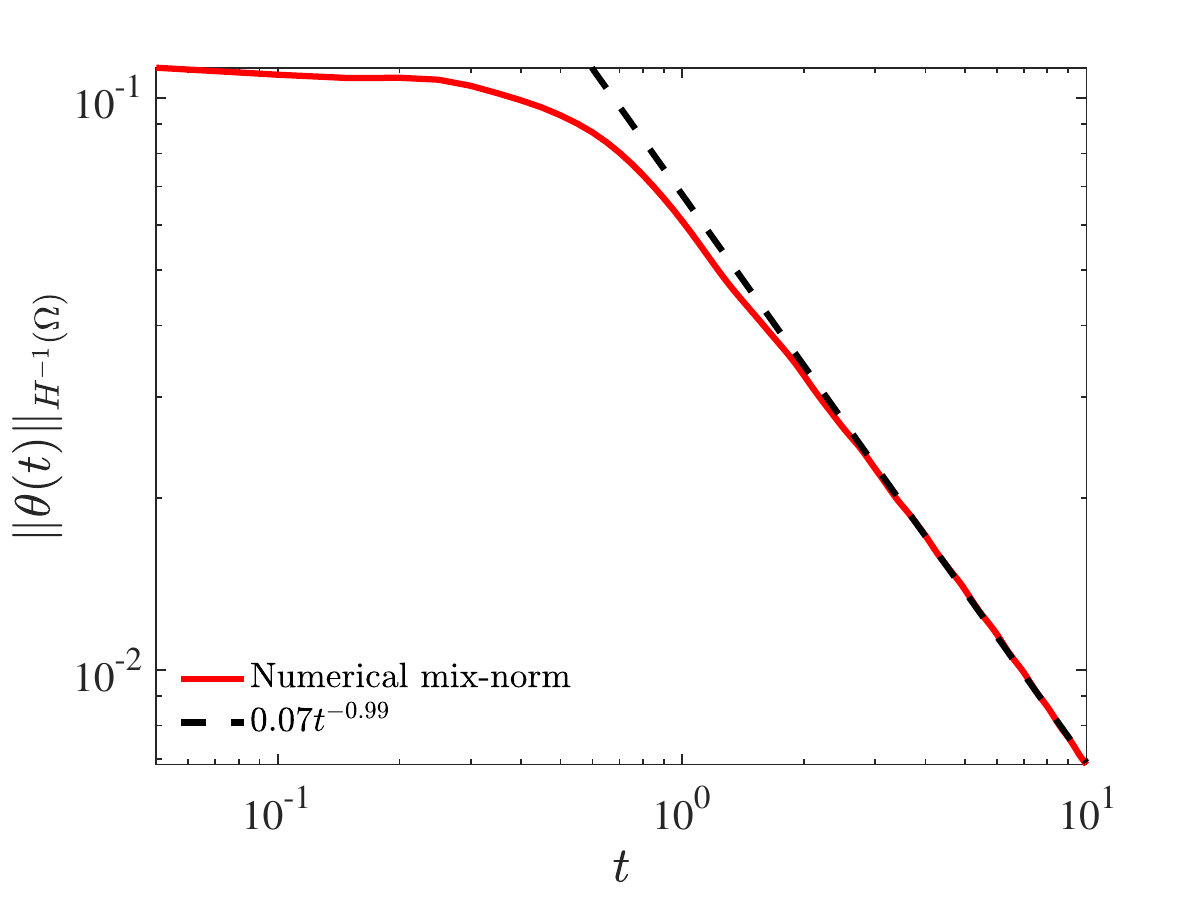}
  \caption{Polynomial decay of $\|\theta(t)\|_{H^{-1}(\Omega)}$.}
  \label{fig:r2-mixnorm-decay}
\end{figure}

\subsection{Example illustrating
Theorem~\ref{20250103-yb-theorem-MixingScale}}
\label{subsection-SecondTheorem}

The purpose of this example is to illustrate the inverse-time decay of the
mixing scale. We reuse the cellular flow in
\eqref{eq:r2-cellular-flow} and let $A$ be the open ball centered at
$(0.75,0.75)$ with radius $0.20$. 
 One can directly check that this example satisfies the hypotheses of
Theorem~\ref{20250103-yb-theorem-MixingScale}.

The mixing scale is the
one-sided distance from $Orbit(A)$ to
$\varPhi(t)(A)$. Since $\varPhi(t)(A)$ is a Jordan domain, its distance
from an exterior point is attained on the transported boundary. We therefore
compute the distance to
the resulting polygonal curve. In these coordinates the angular equation has
constant velocity, so the implicit midpoint scheme is exact. The computation
uses $1{,}048{,}576$ material points on $A$.

Figure~\ref{fig:r2-onecircle-evolution} shows the short-time evolution at
$t=0,1,\ldots,5$ only to illustrate the deformation of the
transported disk. The red region is the transported disk $\varPhi(t)(A)$, while the two black dashed curves are the boundary components of $\operatorname{Orbit}(A)$.
Figure~\ref{fig:r2-onecircle} reports the mixing scale on
$[0, 1000]$. A log--log least-squares fit of
$C(1+t)^{-p}$ over $300\leq t\leq1000$ gives $C=0.78$ and $p=0.99$. The black dashed line in Figure~\ref{fig:r2-onecircle-lnmixscale} extends this fitted law across the displayed interval; only the observations with $300\leq t\leq1000$ determine $C$ and $p$. 

\begin{figure}[htp]
  \centering
  \begin{subfigure}[b]{0.16\textwidth}
    \includegraphics[width=\textwidth]{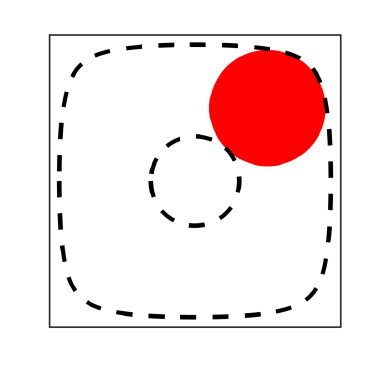}
    \caption{$t=0$}
  \end{subfigure}
  \begin{subfigure}[b]{0.16\textwidth}
    \includegraphics[width=\textwidth]{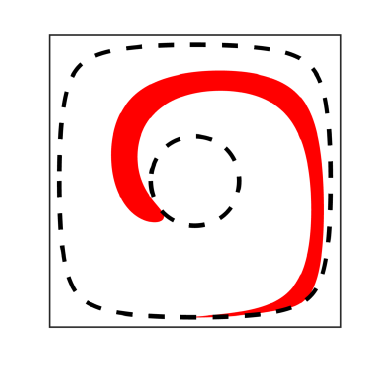}
    \caption{$t=1$}
  \end{subfigure}
  \begin{subfigure}[b]{0.16\textwidth}
    \includegraphics[width=\textwidth]{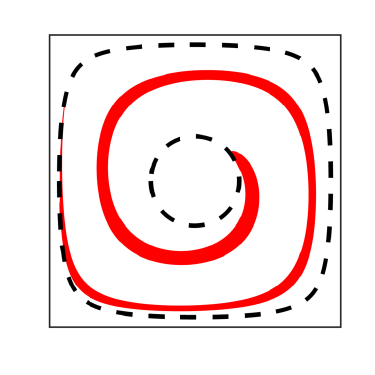}
    \caption{$t=2$}
  \end{subfigure}
  \begin{subfigure}[b]{0.16\textwidth}
    \includegraphics[width=\textwidth]{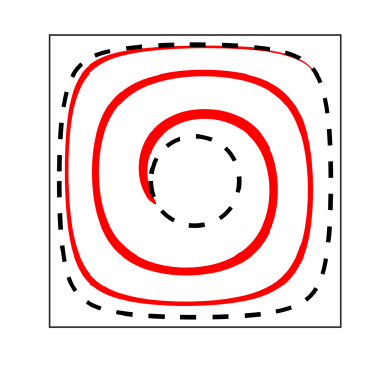}
    \caption{$t=3$}
  \end{subfigure}
  \begin{subfigure}[b]{0.16\textwidth}
    \includegraphics[width=\textwidth]{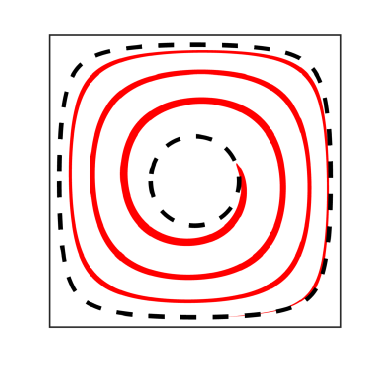}
    \caption{$t=4$}
  \end{subfigure}
  \begin{subfigure}[b]{0.16\textwidth}
    \includegraphics[width=\textwidth]{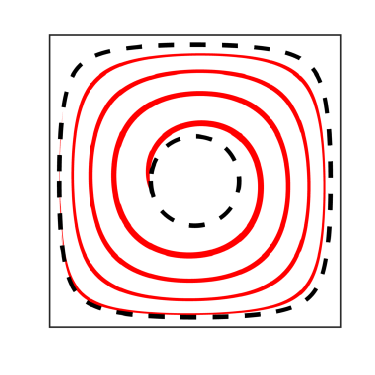}
    \caption{$t=5$}
  \end{subfigure}
  \caption{Short-time evolution of the set $A$ (in red), shown for
  illustration.
  The black dashed curves are the two boundary components of
  $\operatorname{Orbit}(A)$.}
  \label{fig:r2-onecircle-evolution}
\end{figure}

\begin{figure}[htp]
  \centering
  \begin{subfigure}[b]{0.49\textwidth}
    \includegraphics[width=\textwidth]{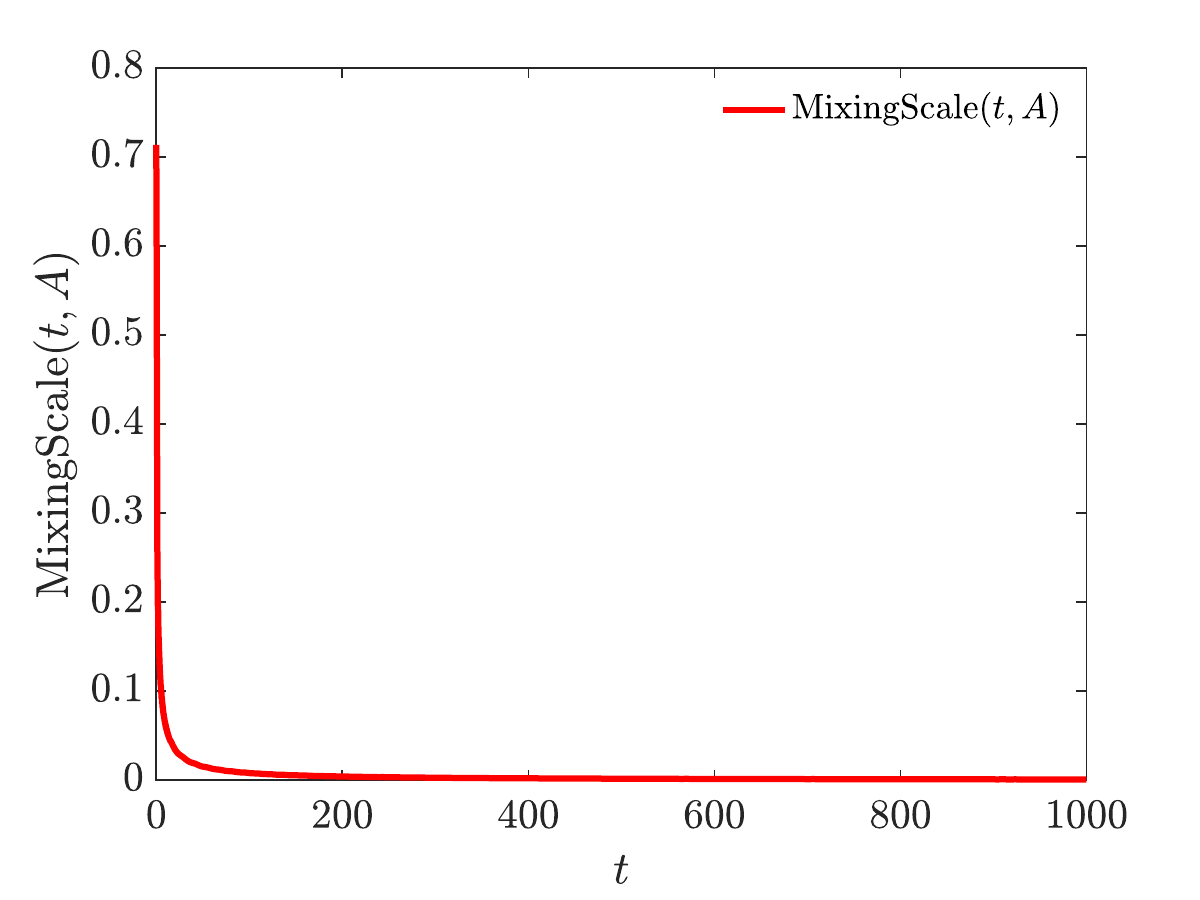}
    \caption{$t\mapsto\operatorname{MixingScale}(t,A)$}
    \label{fig:r2-onecircle-mixscale}
  \end{subfigure}
  \begin{subfigure}[b]{0.49\textwidth}
    \includegraphics[width=\textwidth]{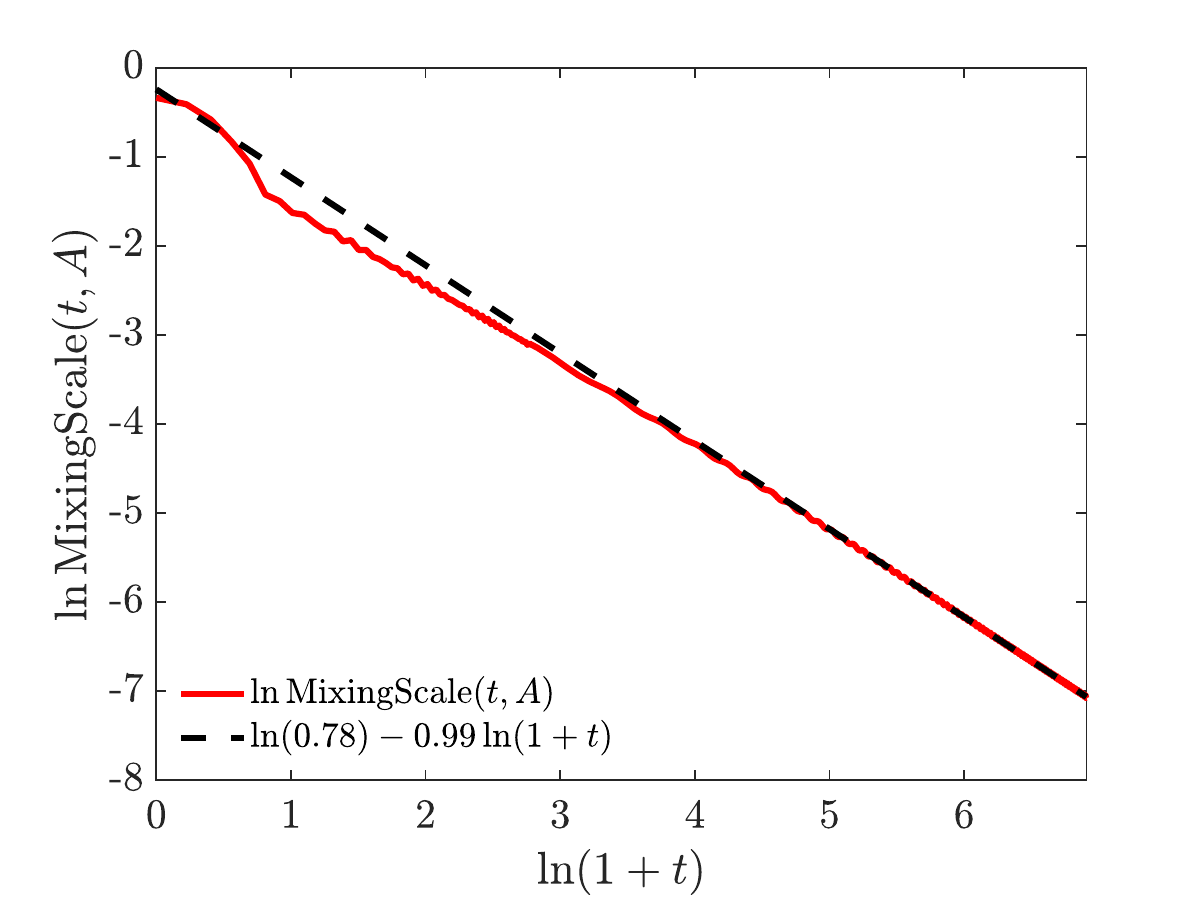}
    \caption{$\ln(1+t)\mapsto\ln\operatorname{MixingScale}(t,A)$}
    \label{fig:r2-onecircle-lnmixscale}
  \end{subfigure}
  \caption{Decay of $\operatorname{MixingScale}(t,A)$. In panel (b), the
  black dashed line is the power-law fit on $300\leq t\leq1000$.}
  \label{fig:r2-onecircle}
\end{figure}

Figure~\ref{fig:r2-onecircle-evolution} shows that $\varPhi(t)(A)$ becomes
longer and thinner while spreading through its orbit saturation. The scale
shows an overall decreasing trend in
Figure~\ref{fig:r2-onecircle-mixscale}. The transformed graph in
Figure~\ref{fig:r2-onecircle-lnmixscale} has an approximately linear tail,
and the fitted exponent $-0.99$ is consistent with the inverse-time order in
\eqref{20250112-yb-OrderOfMixingScale}.

\subsection{Example illustrating
Theorem~\ref{20241021-yb-theorem-OptimalGrowthForCurves}}
\label{subsection-FirstTheorem}

The purpose of this example is to compare the asymptotic coefficient in
Theorem~\ref{20241021-yb-theorem-OptimalGrowthForCurves} with the
computed growth of a transported curve. Let $\Omega := B_1(0)$ and consider
\begin{align*}
H(x):=|x|^4-1, ~  x \in \mathbb R^2
~\text{ and }~
V(x):=-\nabla^{\perp} H|_{ \overline{\Omega}}.
\end{align*}
For $0<|x|\leq1$, the orbit period and its logarithmic gradient are
\begin{align*}
T(x)=\frac{\pi}{2|x|^2}
~\text{ and }~
\nabla\ln T(x)=-2|x|^{-2}x,
~ x \in \overline{\Omega} \setminus \{0\}.
\end{align*}
Fix $\delta := 0.01$ and  take
\begin{align*}
\gamma_0(\alpha)
:=\bigl(\delta+(1-2\delta)\alpha,0\bigr),
~
0\leq\alpha\leq1. 
\end{align*}
Thus only a segment of length $0.01$ is removed from each end of the unit
radial segment. 
By \eqref{20241023-yb-AsymptoticForLengthOfCurves}, we should have
\begin{align}
\lim_{t\to+\infty}\frac{|\gamma_t|}{t}
=\frac{8}{3}\bigl((1-\delta)^3-\delta^3\bigr)
\approx2.59.
\label{eq:r2-exact-curve-length}
\end{align}

We transport uniformly distributed points on $\gamma_0$ and approximate
$|\gamma_t|$ by the length of the resulting polygonal curve. We use
$N_\gamma=160{,}001$ points and $\Delta t=1/800$ on $0\leq t\leq10$. Figure~\ref{fig:r2-length-evolution} displays the progressive winding of
the transported interface. Figure~\ref{fig:r2-length-length} shows that
the length becomes approximately linear in time over the computed
interval, while Figure~\ref{fig:r2-length-slope} shows $|\gamma_t|/t$
approaching the value $2.59$ in
\eqref{eq:r2-exact-curve-length}, indicated by the black dashed line across the
displayed interval. The computation is consistent with the
asymptotic formula in
Theorem~\ref{20241021-yb-theorem-OptimalGrowthForCurves}.

\begin{figure}[H]
  \centering
  \begin{subfigure}[b]{0.16\textwidth}
    \includegraphics[width=\textwidth]{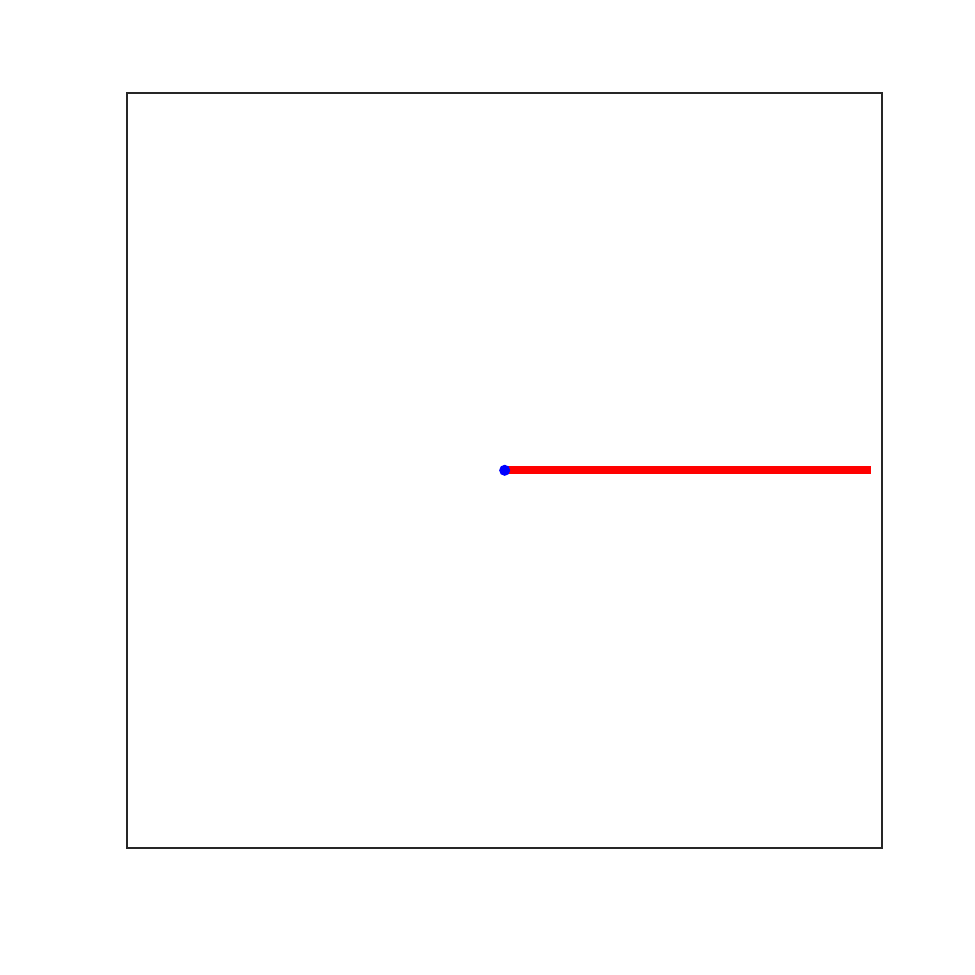}
    \caption{$t=0$}
  \end{subfigure}
  \begin{subfigure}[b]{0.16\textwidth}
    \includegraphics[width=\textwidth]{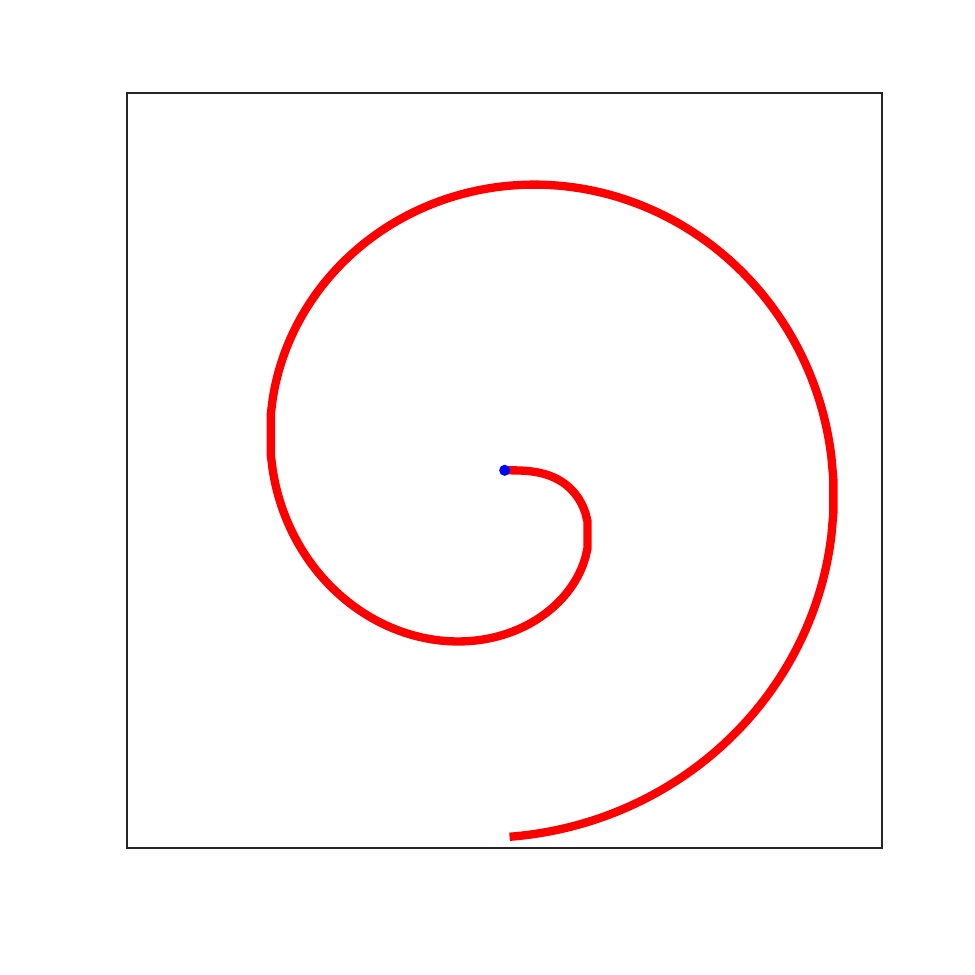}
    \caption{$t=2$}
  \end{subfigure}
  \begin{subfigure}[b]{0.16\textwidth}
    \includegraphics[width=\textwidth]{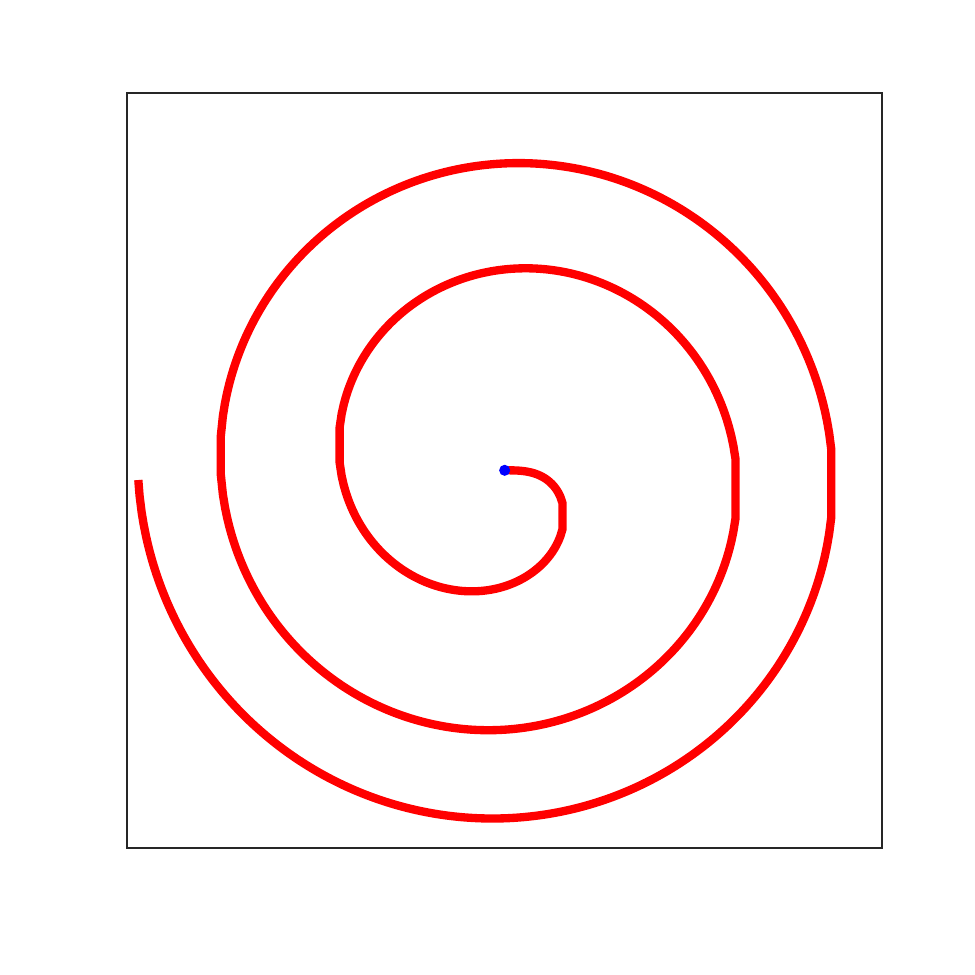}
    \caption{$t=4$}
  \end{subfigure}
  \begin{subfigure}[b]{0.16\textwidth}
    \includegraphics[width=\textwidth]{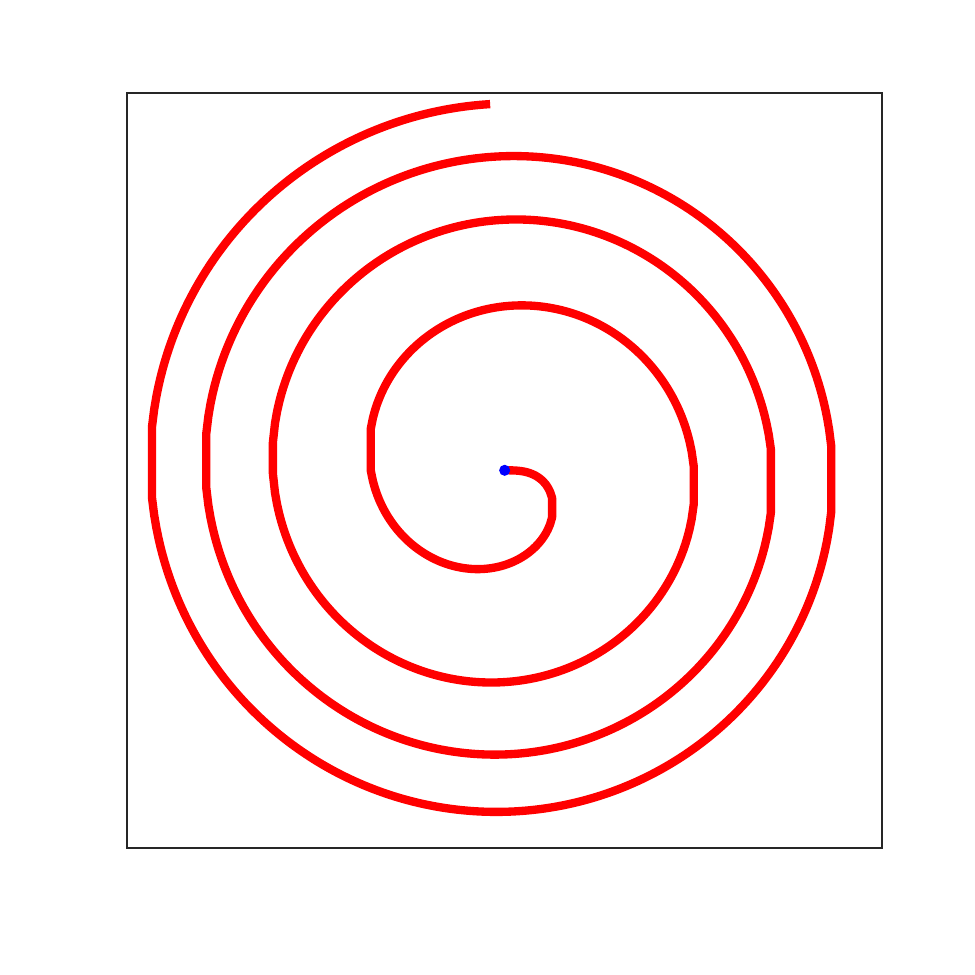}
    \caption{$t=6$}
  \end{subfigure}
  \begin{subfigure}[b]{0.16\textwidth}
    \includegraphics[width=\textwidth]{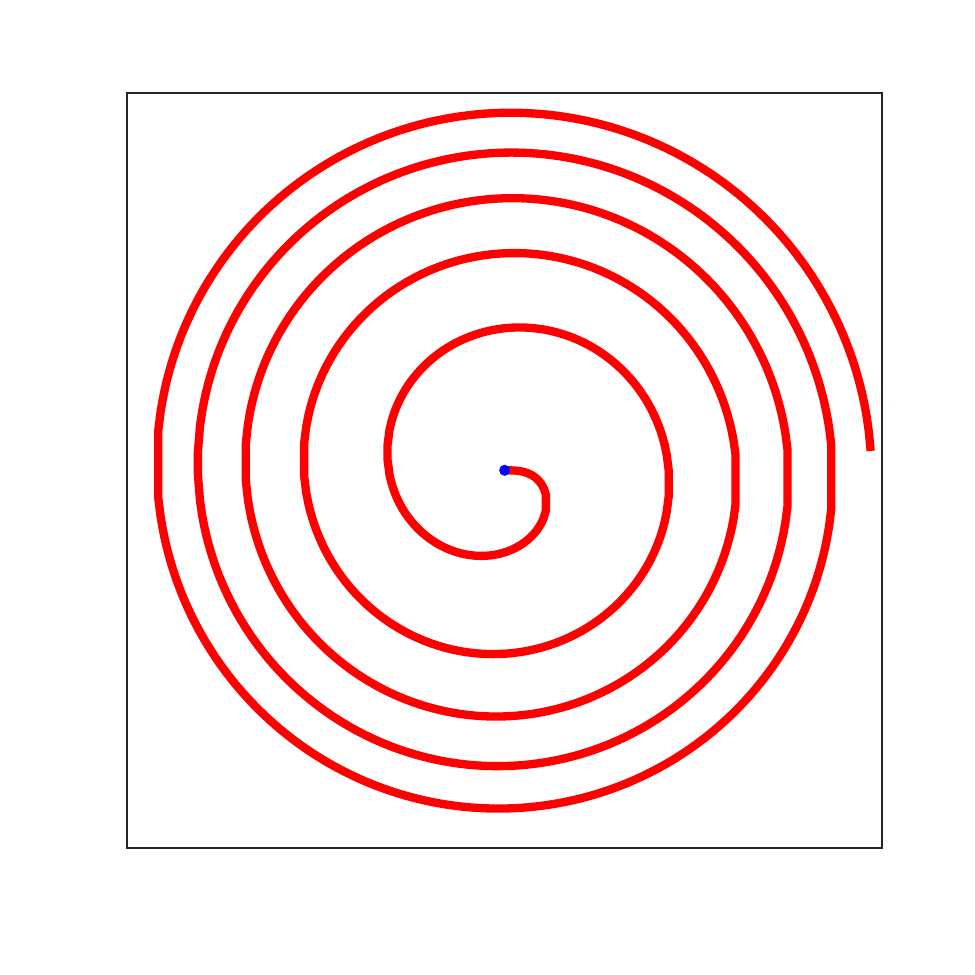}
    \caption{$t=8$}
  \end{subfigure}
  \begin{subfigure}[b]{0.16\textwidth}
    \includegraphics[width=\textwidth]{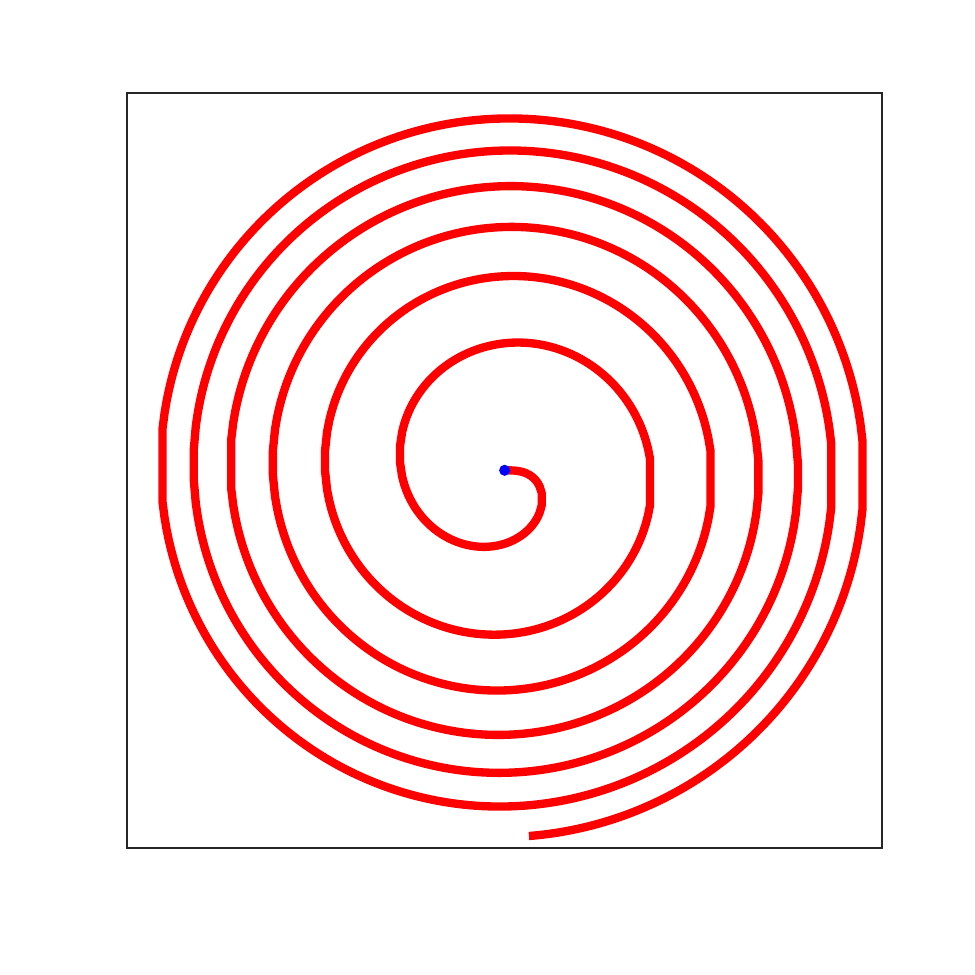}
    \caption{$t=10$}
  \end{subfigure}
  \caption{Evolving curves at different time instants}
  \label{fig:r2-length-evolution}
\end{figure}

\begin{figure}[H]
  \centering
  \begin{subfigure}[b]{0.49\textwidth}
    \includegraphics[width=\textwidth]{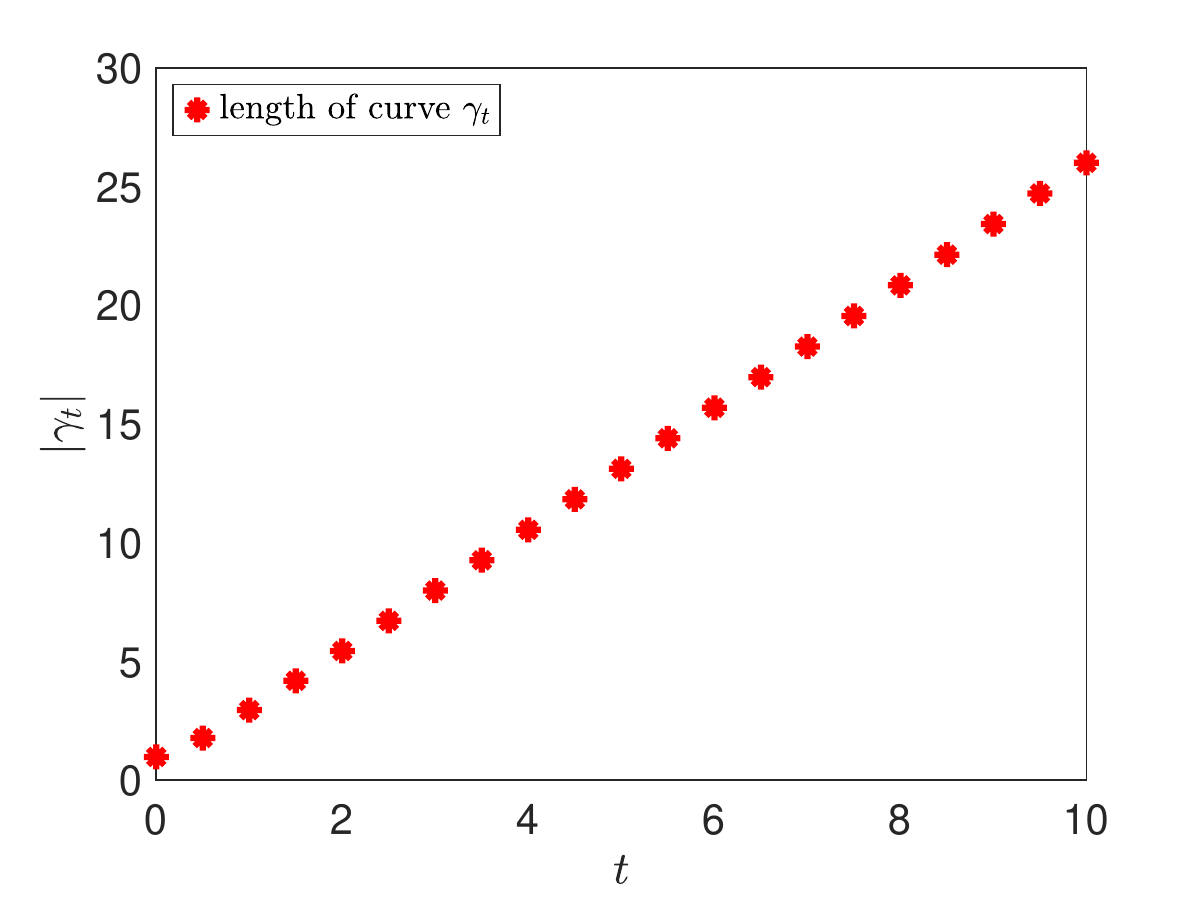}
    \caption{Graph of $t\mapsto|\gamma_t|$}
    \label{fig:r2-length-length}
  \end{subfigure}
  \begin{subfigure}[b]{0.49\textwidth}
    \includegraphics[width=\textwidth]{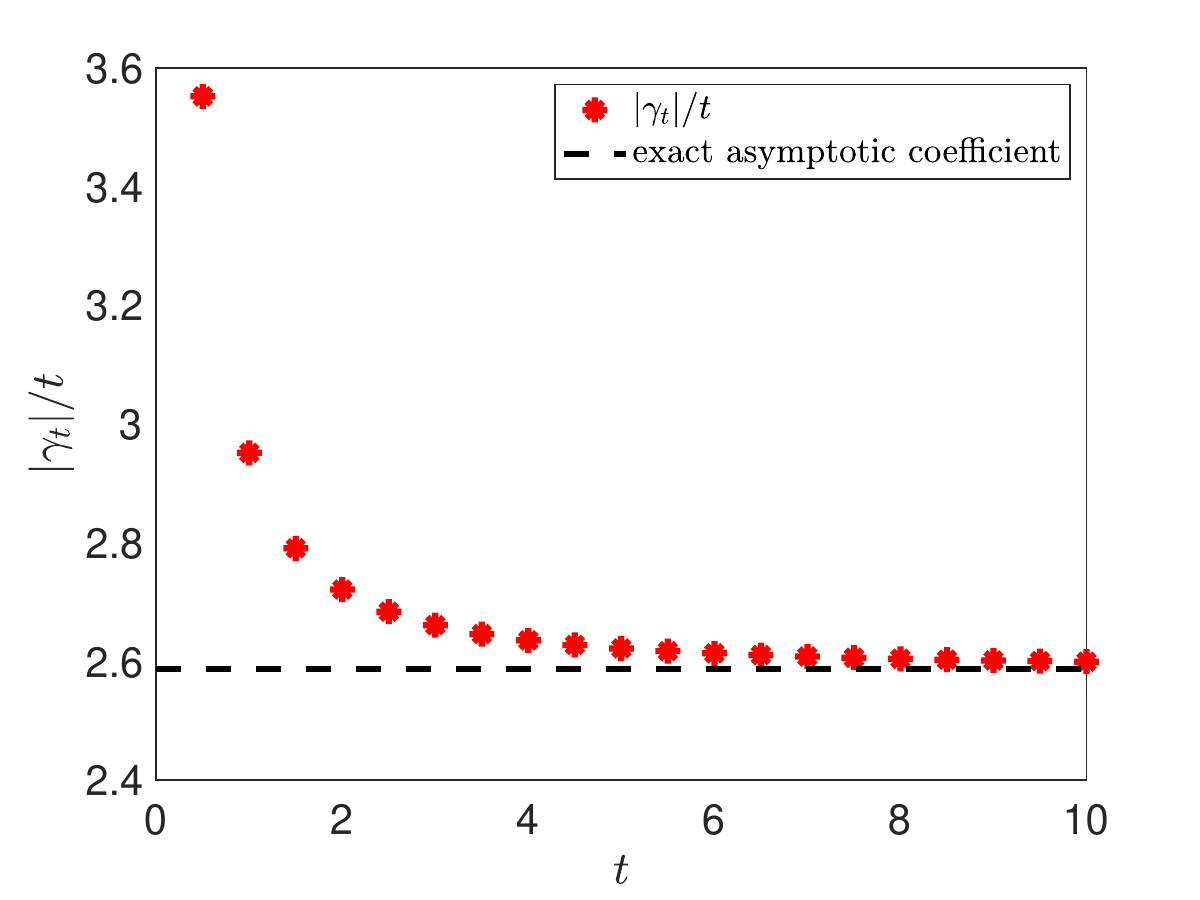}
    \caption{Graph of $t\mapsto|\gamma_t|/t$}
    \label{fig:r2-length-slope}
  \end{subfigure}
  \caption{Linear growth of $|\gamma_t|$ for large times $t$. The black
  dashed line in panel (b) is the exact asymptotic coefficient in
  \eqref{eq:r2-exact-curve-length}.}
  \label{fig:r2-length}
\end{figure}

\section{Conclusions and perspectives}
\label{section-conclusion}

\subsection{Conclusions}

\noindent\textit{The common mechanism.}
On every annulus contained in the regular twist region
\[
\mathcal R_{tw} :=\big\{x\in\Omega  ~:~  0<T(x)<+\infty,\ \nabla T(x)\ne0 \big\},
\]
the flow under the coordinates in
Lemma~\ref{rev:uniform-action-angle} reads as:
\[
\varPhi(t)(\psi^{-1}(h,\vartheta))
=\psi^{-1}(h,\vartheta+t\omega(h))
~\text{ with }
\omega(h)=\frac{2\pi}{T(\psi^{-1}(h,\vartheta))}.
\]
The differential of the cylinder map is
\[
\begin{pmatrix}
    1&0\\
    t\omega'(h)&1
\end{pmatrix}.
\]
Thus nonvanishing transverse variation of the orbit period produces
shear of order $t$.  In physical coordinates,
Theorem~\ref{20241023-yb-proposition-AsymptoticForJacobianOfVeolocityField}
gives a uniform first-order expansion of the flow Jacobian (away from the infinite-period orbits):
\[
J_{\varPhi(t)}(x)
=-t \, V\big( \varPhi(t)(x) \big) \otimes  \nabla\ln T(x) 
+  O(1)
~\text{ as }  t \rightarrow \infty.
\]
This Jacobian expansion, rather than the orbit-period function alone,
provides the structural link between the results of the paper.

Applied to the tangent vectors of an initial curve, the expansion yields the
first-order length formula in
Theorem~\ref{20241021-yb-theorem-OptimalGrowthForCurves}.  For every compact
set $E\Subset\mathcal R_{tw}$, uniformly for $x\in E$ as $t\to+\infty$, the
larger and smaller singular values of $J_{\varPhi(t)}(x)$ are comparable to
$1+t$ and $(1+t)^{-1}$, respectively.  The same shear underlies the scale
$t^{-1}$ in the two mixing theorems.  For
Theorem~\ref{20250103-yb-theorem-MixingScale}, a transverse displacement of
that order produces a phase change of order one, controlling the deformation of the orbit-relative complement (see Proposition \ref{rev:prop52}). 
For Theorem~\ref{third:optimal-negative-norm}, the nonzero angular
Fourier modes acquire the oscillatory factor
$e^{-ikt\omega(h)}$, and integration by parts in $h$ yields cancellation of
order $(1+t)^{-1}$ (see \eqref{20260902-EqualityOfFrequency}).  The two inverse-time mixing laws and the leading
curve-length term are therefore manifestations of the same shear.

\medskip
\noindent\textit{Distinctions between the results.}
The common mechanism does not make the three main theorems equivalent.
The negative-Sobolev result in Theorem \ref{third:optimal-negative-norm} is a linear Eulerian estimate for a scalar after
subtracting its invariant orbit average. Its
upper bound relies on $H^1$ regularity of the initial datum, $C^2$ regularity of
the velocity field, and oscillatory cancellation.  The mixing scale result in Theorem \ref{20250103-yb-theorem-MixingScale} is a
nonlinear geometric statement for a Lipschitz subdomain. It relies on
noninvariance, phase matching, and the interior geometry of the domain,
rather than on scalar cancellation. Both theorems impose
\[
d\big(A,\{T=0, +\infty\}\big)>0
~\text{ and }~
\inf_{ x \in A}  | \nabla T(x) | > 0
\]
on the subdomain $A$ appearing in their statements.

The curve growth result in Theorem \ref{20241021-yb-theorem-OptimalGrowthForCurves} concerns a one-dimensional Lagrangian object and applies
in a broader dynamical setting. It allows stable equilibria and
points where $\nabla T=0$, provided that the initial curve remains separated
from the infinite-period orbits.  At every parameter $\alpha$ for which
$0<T(\gamma_0(\alpha))<+\infty$, the leading shear is nonzero precisely
when
\[
\gamma_0'(\alpha)\cdot\nabla \ln T(\gamma_0(\alpha))\ne0.
\]
At a point where $\nabla T\ne0$, the vanishing of this directional
derivative is equivalent to $\gamma_0'$ being parallel to $V$.  If
$\nabla T=0$, the leading shear vanishes in every tangent direction. Hence the length
remains uniformly bounded precisely when, for almost every $\alpha$, either
$T(\gamma_0(\alpha))=0$, or
$0<T(\gamma_0(\alpha))<+\infty$ and
\[
\gamma_0'(\alpha)\cdot\nabla \ln  T(\gamma_0(\alpha))=0.
\]
The theorem does not in general
assert convergence of $|\gamma_t|/t$, because its leading coefficient
contains the time-dependent factor $|V(\gamma_t(\alpha))|$.
Proposition~\ref{rev:prop52} gives a complementary deformation theorem under
a uniform-orbit cone hypothesis and illustrates the same Jacobian mechanism.

\subsection{Perspectives}

The following problems require estimates beyond the analysis developed here. 

\begin{enumerate}[label=(\roman*)]
    \item \textit{Equilibria.}
    Near a stable equilibrium $p$ of order $m$,
    Lemma~\ref{20250329-yb-lemma-OrbitPeriodAroundEquilibra} and
    \eqref{ham:velocity-bounds} give
    \[
    |V(x)|\asymp|x-p|^m,
    ~\text{ and }~
    T(x)\asymp|x-p|^{1-m}.
    \]
    These estimates are used to establish the uniform Jacobian bounds in
    the curve-length expansion, but do not provide a uniformly
    bi-Lipschitz coordinate chart (in Lemma \ref{rev:uniform-action-angle}) up to $p$.  Indeed, the periodic fibres
    collapse at $p$, and when $m>1$,
    \[
    T(x)\longrightarrow+\infty
    \quad\text{as }x\to p, 
    \]
    although $T(p)=0$ by definition.  Extending the functional and geometric
    theorems to orbits whose closures meet a stable equilibrium
    requires substitutes for the uniform bi-Lipschitz coordinate bounds
    near $p$. A weighted coordinate theory is one possible approach. One must
    also distinguish centers with genuine shear from isochronous regions, where the
    leading variation of the period vanishes. The standard cellular flow  has been treated in \cite{brue2024enhanced}.
    
    \item \textit{Infinite-period trajectories.}
    Lemma~\ref{20250322-yb-propsotion-PropertiesOfOrbitPeriod} shows that every
    nonconstant trajectory with $T=+\infty$ converges to equilibria in forward
    and backward time.  The present theorems exclude these trajectories either
    through compact inclusion in $\mathcal R_{tw}$ or through
    \eqref{20250402-yb-FintiePeriodAssumptionOnCurve}.  Analysis across these trajectories requires quantitative  estimates near its limiting
    equilibria, together with a replacement for the periodic coordinate chart. The
    arguments of this paper do not determine whether the resulting laws retain
    the same algebraic orders or acquire logarithmic or different algebraic
    corrections.
    
    \item \textit{Time-dependent velocity fields.}
    For $V=V(t,x)$, there is generally no fixed orbit period or invariant orbit
    average. The first task is therefore to
    identify the appropriate limiting profile and dynamically accessible set.
    Time-periodic fields provide a natural first
model. Prior work shows that time-periodic flows can enhance relaxation even
when each frozen-time flow is Hamiltonian \cite{Kiselev2008}. These results motivate studying
conditions on the Poincar\'e map under which accumulated shear yields
quantitative estimates for mix-norm decay, mixing scales, and curve growth.
    
    \item \textit{Three-dimensional stationary flows.}
    In three dimensions, stationarity does not imply an integrable foliation
    by streamlines. The steady ABC Euler flows provide classical examples with
    chaotic particle trajectories \cite{dombre1986chaotic}.  A
    three-dimensional theory must accommodate both regular invariant regions, which
    may be organized by invariant tori, and chaotic components, which may be
    studied through local Poincar\'e maps and Lyapunov exponents. Positive stretching alone does
    not imply decay of a prescribed scalar mix-norm: recurrence, folding, and
    cancellation must also be controlled.  A central problem is to determine
    when estimates for the flow Jacobian yield two-sided functional and
    geometric mixing laws on the relevant dynamical components.
    
    \item \textit{Enhanced dissipation in two dimensions.}
    Consider
    \[
    \begin{cases}
        \partial_t\theta^\nu+V\cdot\nabla\theta^\nu
        =\nu\Delta\theta^\nu,
         &(t,x)\in(0,+\infty) \times\Omega,\\
        \partial_{\vec n}\theta^\nu=0,
        &(t,x)\in(0,+\infty) \times\partial\Omega,\\
        \theta^\nu|_{t=0}=\theta_0,
        & x\in\Omega.
    \end{cases}
    \]
    Enhanced
    dissipation for zero streamline-average data in two-dimensional Hamiltonian
    flows has already been studied; see
    \cite[Definition~1.1 and Theorems~1--3]{brue2024enhanced}. A question in the present
    setting is to obtain matching upper and lower
    viscosity-dependent decay estimates for the finite-order class
    (A1)--(A2) and to determine precisely how these
    estimates depend on the inviscid $(1+t)^{-1}$ mixing law.  
\end{enumerate}

\section{Appendix}
\label{section-appendix}

\subsection{Auxiliary lemmas}
\label{appendix-Others}

\begin{lemma}\label{lemma-20260122-CharacterizeInvariantSet}
Let $E \subset \Omega$. Then $\overline{E}$ is $V$-invariant if and only if
$
Orbit(E) \subset \overline{E}.
$
\end{lemma}

\begin{proof}
First we show the necessity. Suppose that $\overline{E}$ is $V$-invariant. Then
$
Orbit(\overline{E}) = \overline{E}.
$
Since $E \subset \overline{E}$, it follows that
$
Orbit(E) \subset \overline{E}.
$

Next we prove the sufficiency. Assume that $Orbit(E) \subset \overline{E}$. Then for each $t \in \mathbb{R}$,
$
\varPhi(t)(E) \subset \overline{E}.
$
This implies 
$
\varPhi(t)(\overline{E}) \subset \overline{E}$
for all  $t\in\mathbb{R}$.
Therefore,
\[
\overline{E}
= \varPhi(0)(\overline{E})
\subset Orbit(\overline{E})
= \bigcup_{t\in\mathbb{R}} \varPhi(t)(\overline{E})
\subset \overline{E},
\]
and hence $\overline{E}$ is $V$-invariant. The proof is completed.
\end{proof}

\begin{lemma}\label{20250322-yb-lemma-ZeroMixingScaleForOpenSets}
Let $\mathcal O \subset \mathbb R^n$ (with $n\in\mathbb N^+$) be a bounded nonempty open set. Then, 
\[
\inf_{x\in\overline{\mathcal O}} |B_r(x)\cap \mathcal O| >0
~\text{ for each } r >0.
\]
\end{lemma}

\begin{proof}
Fix $r>0$. Since every point of $\overline{\mathcal O}$ can be approximated by points in $\mathcal O$, the collection
$
\{B_{r/2}(x)\}_{x \in \mathcal O}
$
forms an open covering of the compact set $\overline{\mathcal O}$. Hence there exists a finite subcover
$
\{B_{r/2}(x_k)\}_{k=1}^N
$ 
 of $\overline{\Omega}$. 
Let $x\in\overline{\mathcal O}$ be arbitrary. Then $x\in B_{r/2}(x_{k_0})$ for some $k_0\in\{1,\dots,N\}$. Consequently,
$
B_r(x) \supset B_{r/2}(x_{k_0}).
$
Thus
\[
|B_r(x)\cap\mathcal O|
\ge
|B_{r/2}(x_{k_0})\cap\mathcal O|.
\]
Since $x$ was arbitrary, we obtain
\[
\inf_{x\in\overline{\mathcal O}} |B_r(x)\cap\mathcal O|
\ge
\min_{1\le k\le N} |B_{r/2}(x_k)\cap\mathcal O|
>0,
\]
which completes the proof.
%
\end{proof}

\begin{lemma}\label{lemma-20260122-FineMixingForInvariantSet}
Let $A\subset\Omega$ be a nonempty open set. For any $t\ge0$, the set $\overline{A}$ is $V$-invariant if and only if
\begin{align}\label{20260122-FineMixingForInvariantSet}
MixingScale(t,A)=0 .
\end{align}
\end{lemma}

\begin{proof}
Let $ t\geq 0$. 
First, we prove the necessity. Suppose that $\overline{A}$ is $V$-invariant. Then
\begin{align}\label{20260122-yb-SubsetRelation}
	\varPhi(t)(A) \subset Orbit(A) \subset Orbit(\overline{A}) = \overline{A}.
\end{align}
Fix an arbitrary $r>0$. By Lemma \ref{20250322-yb-lemma-ZeroMixingScaleForOpenSets}, applied with $\mathcal O = \varPhi(t)(A)$, there exists a constant $c>0$ such that
\begin{align*}
	| B_r(x) \cap \varPhi(t)(A) | > c 
	~ \text{for every } x \in \overline{\varPhi(t)(A)} .
\end{align*}
By (iii) of Definition \ref{20250103-yb-DefinitionOfMixingScale}, the above inequality together with \eqref{20260122-yb-SubsetRelation} implies that
\[
MixingScale(t,A) \le r.
\]
Since $r>0$ was chosen arbitrarily, the above implies \eqref{20260122-FineMixingForInvariantSet}. 

Next, we prove the sufficiency. Suppose that \eqref{20260122-FineMixingForInvariantSet} holds. 
Then, by (iii) of Definition \ref{20250103-yb-DefinitionOfMixingScale}, the set $Orbit(A)$ is approximated by $\varPhi(t)(A)$. This implies
\begin{align*}
	Orbit\big(\varPhi(t)(A)\big)
	\subset Orbit(A)
	\subset \overline{\varPhi(t)(A)}
	= \varPhi(t)(\overline{A}).
\end{align*}
Then, by Lemma \ref{lemma-20260122-CharacterizeInvariantSet}, 
 $\varPhi(t)(\overline{A})$ is $V$-invariant. Hence,  $\overline{A}$ is also $V$-invariant. The proof is completed.
\end{proof}


\begin{lemma}\label{20260903-EquivalentDistanceHypothesis}
For each nonempty set $E \subset \Omega$, Property-$E$ is defined as follows:
\begin{align*}
    d\big(E,\{T=0, +\infty\}\big)>0
    ~\text{ and }~
    \inf_{ x \in E}  | \nabla T(x) | > 0.
\end{align*}
Let $A \subset \Omega$ be a nonempty set.  
 Then, Property-$A$ is equivalent to Property-$Orbit(A)$. 
\end{lemma}

\begin{proof}
First of all, we claim
\begin{align}\label{BothAwayFromExtremeOrbitPeriods-20260926}
    d\big(A,\{T=0, +\infty\}\big)>0   ~\Longleftrightarrow~
    d\big(Orbit(A),\{T=0, +\infty\}\big)>0.
\end{align}
Clearly, the latter implies the former. To show the converse, assume that the former is true. Then,  since $T(\cdot)$ is of class $C^1$ outside $\{T=0, +\infty\}$, there are $c_1,c_2>0$ such that  
\begin{align*}
    c_1  \leq   T(x)  \leq  c_2, ~~  x \in \overline{A}. 
\end{align*}
It is clear that
\begin{align*}
    c_1  \leq   T(x)  \leq  c_2, ~~  x \in Orbit(A). 
\end{align*}
We want to prove that $Orbit(A)$ remains at a positive distance from $\{T=0, +\infty\}$. Suppose, for a contradiction, that this is false. Fix an arbitrary $\varepsilon >0$. Then, there is an $x_{\varepsilon} \in Orbit(A)$ such that
\begin{align*}
    d( x_{\varepsilon}, \{T=0, +\infty\}) < \varepsilon.
\end{align*}
At the same time, note that the orbit period $T(x_{\varepsilon})$ has an upper bound $c_2$. Thus, by ODE theory, we find that for some $C=C(\Omega, V, c_2)>0$ (independent of $\varepsilon$),
\begin{align*}
    \sup_{y \in Orbit(x_{\varepsilon}) }  d\big( y, \{T=0, +\infty\} \big) < C \varepsilon.
\end{align*}
In particular, since $A$ intersects $Orbit(x_{\varepsilon})$, we obtain that the distance between $A$ and $\{T=0, +\infty\}$  is bounded by $C \varepsilon$. Now, since $\varepsilon>0$ was arbitrary, we are led to a contradiction. Therefore, the claim \eqref{BothAwayFromExtremeOrbitPeriods-20260926} is true.

Next, to show the equivalence between Property-$Orbit(A)$ and Property-$A$, it is clear that 
Property-$Orbit(A)$ implies Property-$A$, since $A \subset Orbit(A)$. To prove the converse, assume that Property-$A$ holds. By \eqref{BothAwayFromExtremeOrbitPeriods-20260926} and the first inequality in Property-$A$, we have the first inequality in Property-$Orbit(A)$.

To show the second inequality in Property-$Orbit(A)$, fix an  $x_0 \in Orbit(A)$. Then, 
\begin{align*}
    x_1 := x(t_0; x_0) \in A
    ~\text{ for some }~
    t_0 \in \mathbb R. 
\end{align*}
Then, by \eqref{20250329-yb-GlobalUpperBoundsOfOrbitPeriodAndRatioOfVelocity} (with $C>0$) and the first inequality in Property-$A$, we obtain
\begin{align*}
    |\nabla T(x_0)| = |\nabla T(x_1)|   \frac{ |V(x_0)| }{ |V(x_1)| }
    > C^{-1}   \inf_{ x \in A}  | \nabla T(x) |.   
\end{align*}
This, together with the second inequality in Property-$A$, gives the second inequality in Property-$Orbit(A)$. 
In conclusion, Property-$Orbit(A)$ is true. This completes the proof. 
\end{proof}

\subsection{Coordinate system}
\label{appendix-Coordinate}

The coordinate system in Lemma \ref{rev:uniform-action-angle} is classical for the Hamiltonian system \eqref{20240925-yubiao-Flow} and has been used in the literature, for instance in \cite[Subsection 2.2]{brue2024enhanced} for 2D compact manifolds. For the sake of completeness of this paper, we present it and related important properties as follows. 

\begin{lemma}
    \csname @show@reffalse\endcsname 
    \label{rev:uniform-action-angle}
    Set $\mathbb T:=\mathbb R/(2\pi\mathbb Z)$. Let $E$ be a subdomain (open and connected) of the following open set:
    \begin{align}\label{20260901-NonTrivalPeriodRegion}
        \mathcal R
        :=\bigl\{
        x\in\Omega ~:~ 0<T(x)<+\infty
        \bigr\}.
    \end{align}
    Then, the set $H(E)$ is an
    open interval $(a,b)$, and there is a $C^1$ diffeomorphism
    $\psi:Orbit(E)\to(a,b)\times\mathbb T$ such that for all $t\in\mathbb R$ and $(h,\theta)\in(a,b)\times\mathbb T$, 
    \begin{align}\label{lip:aa}
        \varPhi(t)\bigl(\psi^{-1}(h,\theta)\bigr)
        &=\psi^{-1}\bigl(h,\theta+t\omega(h)\bigr),
        \\
        H\bigl(\psi^{-1}(h,\theta)\bigr)&=h,
        \nonumber
    \end{align}
    where $H$ is the Hamiltonian function given by \eqref{ham:representation}, and  $\omega(h):=2\pi/T(\psi^{-1}(h,\theta)) \in C^1((a,b))$ is the angular velocity  independent of
    $\theta$. Moreover, with 
    $x:=\psi^{-1}(h,\theta)$, 
    \begin{align}\label{lip:aab}
        \begin{cases}
            J_{\psi}(x) W
            & = \begin{pmatrix}
                1  & 0 \\
                q_{\psi}(x)  &  \omega(h)|V(x)|^{-2}
            \end{pmatrix}        
            \begin{pmatrix}
                W \cdot V^{\perp}(x) \\
                W \cdot V(x)
            \end{pmatrix},
            ~W \in \mathbb R^2,
            \vspace{0.5em}\\
            \det J_{\psi}(x) &= -\omega(h)
        \end{cases}
    \end{align}
    for some continuous function $q_{\psi}$. 
    Furthermore, if further assume that 
    \begin{align*}
         d:=d\big(E, \{T=0,+\infty\} \big)>0 
         ~\text{ and }~
         \kappa:= \inf_{x \in E} |\nabla T(x)| > 0,
    \end{align*}
then there is a constant $C=C(\Omega, V, d, \kappa)$ such that when $x = \psi^{-1}(h, \theta) \in Orbit(E)$, 
    \begin{align}\label{20260902-Upper-Lower-Bound-OfCoordinateMap}
        \| J_{\psi} (x) \| + \| J_{ \psi^{-1}}(h, \theta) \|
        + \| (\omega, 1/\omega) \|_{ C^1( (a,b); \mathbb R^2)}
        + \| 1 / \omega' \|_{ C( (a,b)) }
        \leq C.
    \end{align}
\end{lemma}

\begin{proof}
    First of all, $\nabla H \in C^1(\Omega)$ is not zero over the set $\mathcal R$. Indeed, it follows from \eqref{ham:representation} that $|\nabla H| = |V|$, and from (i) of Lemma \ref{20250322-yb-propsotion-PropertiesOfOrbitPeriod} that  both $V$ and $T$ are nonzero on $\mathcal R$. Thus, $\nabla H$ is also not zero over  $\mathcal R$.

    The function $H$ cannot attain its local maxima or minima in $E$. Otherwise, it would have at least one critical point where $\nabla H$ is zero. Then, because $E$ is connected and open as a domain, there are numbers $a,b \in \mathbb R$ such that
    \begin{align*}
        H( E ) = (a,b).
    \end{align*}
    Furthermore, the orbit $Orbit(E)$ is also open and connected, and for each $h \in (a,b)$, $H^{-1}(h) \cap Orbit(E)$ is the orbit $Orbit(x_h)$ with some $x_h \in E$. Indeed, there are no two different orbits $Orbit(x_{h_1}) \subset E$ and $Orbit(x_{h_2}) \subset E$ included in  the same level set $H^{-1}(h)$. Otherwise, when moving transversely from one orbit to the other in the same component $\mathcal R$, the value of $H$ must return to the same value $h$. Then, during this transverse motion, there must be an extremum of $H$, which means that $H$ has a critical point in $Orbit(E)$. This leads to a contradiction. 
    
    Fix an $x_0 \in Orbit(E)$. Consider the following gradient flow
    \begin{align*}
        x'(t) = \frac{ \nabla H( x(t)) }{  | \nabla H( x(t)) |^2  }, ~ t \in (a,b);~~ x|_{ t= H(x_0)} =x_0.
    \end{align*}
    Solve this equation to obtain a $C^1$ curve $\gamma(\cdot)$, which will be identified as the baseline $\theta =0$, pulled back by the coordinate map. One can directly check that $H(\gamma(t))$ is strictly increasing as $t$ increases and that
    \begin{align}\label{20260902-TimeIsEnergy}
        H( \gamma(t) ) = t  ~\text{ when }~  t \in (a,b). 
    \end{align}
    
    We now define the inverse coordinate map: when $(h, \theta) \in (a,b) \times \mathbb T$, 
    \begin{align}\label{20260902-EnergyAngleCoordinate}
        \psi^{-1}(h, \theta) := \varPhi\big( \theta / \omega(h) \big) ( \gamma(h) )
        ~\text{ with }~
        \omega(h) := 2 \pi /T( \gamma(h) ). 
    \end{align}
    The map $\psi^{-1}$ is well defined due to the periodicity of each orbit, i.e., for each $k \in \mathbb Z$, 
    \begin{align*}
        \psi^{-1}(h, \theta + 2k \pi)
        =& \varPhi\big( \theta / \omega(h) + k T( \gamma(h) )\big) ( \gamma(h) )
        \nonumber\\
        =& \varPhi\big( \theta / \omega(h) \big) ( \gamma(h) )
        = \psi^{-1}(h, \theta).
    \end{align*}
    Furthermore, for each $h$, the variable $\theta$ labels the orbit $Orbit( \gamma(h) )$; the functions $\varPhi(\cdot)$ (given in \eqref{20240926-yubiao-DefinitionOfFlow}), $T(\cdot)$ and $\gamma(\cdot)$ are of class $C^1$. It is clear that $\psi^{-1} $ is of class $C^1$. One can also directly check that it is one-to-one from $(a,b) \times \mathbb T$ to the orbit $Orbit(E)$. 
    
    Now, we verify \eqref{lip:aa}. For all $t\in\mathbb R$ and $(h,\theta)\in(a,b)\times\mathbb T$, we see from \eqref{20260902-EnergyAngleCoordinate} that 
    \begin{align*}
        \varPhi(t)\bigl(\psi^{-1}(h,\theta)\bigr)
        =\varPhi\big( t + \theta / \omega(h) \big) ( \gamma(h) )
        = \psi^{-1}(h,\theta + t \omega(h) ), 
        \nonumber
    \end{align*}
    and from \eqref{20260902-EnergyAngleCoordinate} and \eqref{20260902-TimeIsEnergy} that
    \begin{align*}
        H\bigl(\psi^{-1}(h,\theta)\bigr) = H ( \gamma(h) ) = h.
    \end{align*}
    These lead to \eqref{lip:aa}.

    Now, we  prove \eqref{lip:aab}. The second equality in \eqref{lip:aab} follows from the first one. To show the first one, we fix $x = \psi^{-1}(h, \theta)$ and differentiate with respect to the time variable at $t=0$ in \eqref{lip:aa} to obtain 
    \begin{align}\label{20260902-FormulaOfTheta}
        \begin{pmatrix}
            0 \\
            \omega(h)
        \end{pmatrix}
        = 
        \frac{d}{dt}  \begin{pmatrix}
            h \\
            \theta + t \omega(h)
        \end{pmatrix}
        = \frac{d}{dt}  \psi \circ \varPhi(t) \circ \psi^{-1} ( h, \theta)
        =  J_{\psi}(x) V(x) . 
    \end{align}
    At the same time, we have that for each $W\in \mathbb R^2$,
    \begin{align*}
        H(x+ \varepsilon W) = H(x) +  \varepsilon \nabla H \cdot W + 
        o(\varepsilon)
        ~\text{ as }~  
        \varepsilon \rightarrow 0. 
    \end{align*}
    With the projection $\pi_1$ onto the variable $h$, the above, together with \eqref{ham:representation}, implies 
    \begin{align*}
        \pi_1 J_{\psi}(x)W = \frac{d}{d\varepsilon} H(x+\varepsilon W) 
        \big|_{\varepsilon=0}
        = \nabla H \cdot W
        = V^{\perp} \cdot W,
        ~~ W\in \mathbb R^2.
    \end{align*}
    By direct computation, the above, along with \eqref{20260902-FormulaOfTheta}, gives the first equality in \eqref{lip:aab}.

    Finally, we verify \eqref{20260902-Upper-Lower-Bound-OfCoordinateMap}. Consider the following extended set
    \begin{align*}
        \widetilde{\mathcal R} := \big\{ x\in \overline{\Omega} ~:~
            0 < T(x) < \infty
        \big\}.
    \end{align*}
    Since $E$ is connected, the set $E$ is included in a connected component $\mathcal R_1$ of the set $\widetilde{\mathcal R}$. The orbit $Orbit(\mathcal R_1)$ is also connected and included in $\widetilde{\mathcal R}$, so we have $Orbit(\mathcal R_1) = \mathcal R_1$. Then, by similar arguments, the aforementioned coordinate system $(\psi, Orbit(E))$ can be extended to $(\psi, \mathcal R_1)$. Here, we mention that $\mathcal R_1$ may reach the boundary $\partial\Omega$. 
    
    Because $E$ is assumed to be bounded with parameters $d, \kappa$, the orbit $Orbit(E)$ is also bounded with new parameters $\hat d, \hat \kappa$ (by Lemma \ref{20260903-EquivalentDistanceHypothesis} in the appendix):
    \begin{align*}
        \hat d := d\big(Orbit(E), \{T=0,+\infty\} \big)>0
        ~\text{ and }~
        \hat\kappa := \inf_{x \in Orbit(E)} |\nabla T(x)| > 0.
    \end{align*}
    Thus, $Orbit(E)$ is compactly contained in $\mathcal R_1$. 
    Note that the maps $\psi, \psi^{-1}$ are of class $C^1$ over the set $\mathcal R_1$, and that the functions $\omega, \omega^{-1}$ are of class $C^1$ over $\mathcal R_1$. Then, \eqref{20260902-Upper-Lower-Bound-OfCoordinateMap} follows at once. 
    This completes the proof. 
\end{proof}

\subsection{Proof of Lemma \ref{20241001-yb-lemma-Hamiltonian}}
\label{appendix-HamiltonianFunction}

%

First of all, one can directly check that $P_k$ is $C^2$, and the identities \eqref{20250329-yb-AsymptoticsOfHamiltonianAndVelocity} now follow directly from
\eqref{ham:local-model}--\eqref{ham:remainder}. In particular, 
$V=\mathcal J\nabla H$ and $J_V=\mathcal J D^2H$.

Next, we show \eqref{ham:velocity-bounds}. Compute $\nabla P_k$ at $y \neq 0$: 
\[
\nabla P_k(y)=|y|^{m_k-1} S_k y
+\frac{m_k-1}{2} |y|^{m_k-3} (y^\top S_ky)y.
\]
By the nonsingularity of $S_k$ and the fact that $m_k \geq 1$,   $\nabla P_k$ is not zero on the unit circle. Otherwise, for some $e \in \mathbb S^1$, 
\begin{align*}
    0 = \nabla P_k(e) = \bigg( I_2 + \frac{m_k-1}{2} e \otimes e  \bigg) S_k e ,   
\end{align*}
which implies $S_k e=0$, a contradiction. 
Now, by the homogeneity of $P_k$, we obtain
\[
c|y|^{m_k}\le|\nabla P_k(y)|\le C|y|^{m_k}
~\text{ and }~
\|D^2 P_k(y)\|\le C|y|^{m_k-1},
~ y \in \mathbb R^2.
\]
This, together with the second and third equalities in  \eqref{20250329-yb-AsymptoticsOfHamiltonianAndVelocity},  implies \eqref{ham:velocity-bounds}. 

Finally, we verify that every stable equilibrium lies in $\Omega$. Assume that $x_k^*$ is a stable equilibrium in $\overline{\Omega}$. In this case,  $S_k$ is definite and so  $x_k^*$ is a strict local extremum of $H$ (by \eqref{ham:local-model} and \eqref{ham:remainder}).
If $x_k^*$ were on $\partial\Omega$, the local Lipschitz boundary would contain points arbitrarily close to $x_k^*$ with the same Hamiltonian value (see \eqref{ham:representation}). This contradicts the strict extremum at $x_k^*$. Hence, every stable equilibrium lies in $\Omega$.

\subsection{Proof of Lemma \ref{20250329-yb-lemma-EquilibriaInFinitePeriodRegion}}
\label{appendix-EquilibriaInFinitePeriodRegion}

First, by the definition of $\Omega_{\geq \varepsilon}$ in \eqref{20250323-yb-LargePeirodLayer}, the set $\Omega_{\geq \varepsilon}$ is closed.
Let $x_k^*$ be an unstable equilibrium. We claim that
\begin{align}\label{20250330-yb-InfinitePeriodAroundUnstableEquilibrium}
    d\big(x_k^*, \{T=+\infty\}\big) = 0.
\end{align}
Once this is established, it follows from \eqref{20250323-yb-LargePeirodLayer} that the set $\Omega_{\geq \varepsilon}$ cannot contain any unstable equilibrium.

It remains to prove \eqref{20250330-yb-InfinitePeriodAroundUnstableEquilibrium}.  
Since $x_k^*$ is unstable, by Definition \ref{20250329-yb-ClassificationOfEquilibria}, we have that the symmetric matrix $S_k $ is indefinite. By assumption (A2), there exists an orthogonal coordinate system $\{\vec v_1,\vec v_2\}$ centered at $x_k^*$, two positive constants $a,b$, and a remainder term $R \in C^2(\mathbb R^2)$, satisfying
\begin{align}\label{20250330-yb-GoodRemainder}
    \partial_{x_1}^{\alpha}\partial_{x_2}^{\beta}R(x_1,x_2)
    = o\!\left(|(x_1,x_2)|^{2-\alpha-\beta}\right),
    ~\text{ as }  (x_1,x_2) \rightarrow 0, ~
    \alpha+\beta \le 2,
\end{align}
such that
\begin{align*}
    H(x_k^* + x_1\vec v_1 + x_2\vec v_2)  - H(x_k^*)
    =
    \pm  |(x_1,x_2)|^{m_k-1}
    \big(
    a^2 x_1^2 - b^2 x_2^2 + R(x_1,x_2)
    \big).
\end{align*}
Without loss of generality, we assume that
\begin{align*}
    H(x_k^* + x_1\vec v_1 + x_2\vec v_2)  - H(x_k^*)
    &=
    H_p(x_1,x_2)
    +
    |(x_1,x_2)|^{m_k-1} R(x_1,x_2)
    \nonumber\\
    &:=
    |(x_1,x_2)|^{m_k-1}
    \big(
    a^2 x_1^2 - b^2 x_2^2 + R(x_1,x_2)
    \big).
\end{align*}
By direct computation, there is a small $r_0>0$ such that for each $r \in (0, r_0) $, the level set $\{ H_p(x) =0 \}$ on $\partial B_r(0)$ has four points that can be parameterized by four $C^1$ functions $\{ y_k(r)\}_{k=1}^4$ (with $r$ as the independent variable). Furthermore,
\begin{align*}
    |\nabla H_p(r)| \asymp  r^m  \text{ as }  r \rightarrow 0^+. 
\end{align*}
By \eqref{20250330-yb-GoodRemainder}, we find that $H$ also satisfies similar properties: there is a small $r_1 \in (0, r_0)$ such that for each $r \in (0, r_1) $, the level set $\{ H(x) =H(x_k^*) \}$ on $\partial B_r(x_k^*)$ has four points that can be parameterized by four $C^1$ functions 
\begin{align}\label{FourBranchesNearUnstablePoint-20261009}
    \{ \hat y_k(r)\}_{k=1}^4  \text{(with $r$ as the independent variable). }
\end{align}
 Furthermore,
\begin{align*}
    |\nabla H(x_k^* + r)| \asymp  r^m  \text{ as }  r \rightarrow 0^+. 
\end{align*}
Therefore, for each curve $\hat y_k( \cdot )$, the orbit period is computed as follows:
\begin{align*}
    T( y_k(\cdot) ) \geq  \int_0^{r_1}  \frac{ |dy_k(r)| }{ |V( y_k(r) )| }
    = \int_0^{r_1}  \frac{ |dy_k(r)| }{ |\nabla H( y_k(r) )| }
    \gtrsim \int_0^{r_1}    |y_k(r) |^{-m}    |dy_k(r)|
    = +\infty. 
\end{align*}
This means that each curve $\hat y_k( \cdot )$ has the infinite period. When $x_k^* \in \Omega$, all these curves are in $\Omega$ after reducing $r_1$. When $x_k^* \in \partial\Omega$, the boundary $\partial\Omega$ near $x_k^*$ is the level set of $H$ (see \eqref{ham:representation}), which implies that at least one curve of $\{ \hat y_k(r)\}_{k=1}^4$  lies in $\overline{ \Omega }$. 
Therefore,  the conclusion  \eqref{20250330-yb-InfinitePeriodAroundUnstableEquilibrium} is true. This completes the proof.

\subsection{Proof of Lemma \ref{20250322-yb-propsotion-PropertiesOfOrbitPeriod}}
\label{appendix-PropertiesOfOrbitPeriod}

The statements (i)–(iii) are proved below one by one.

\vskip 5pt
\noindent\textit{(i)} 
The sufficiency is obvious, since every equilibrium is a constant solution of equation \eqref{20240925-yubiao-Flow}. 
To show the necessity, assume that $T(x_0)=0$. We argue by contradiction and suppose that $V(x_0)\neq 0$. Then, by \eqref{20240925-yubiao-Flow}, there exists $\delta_0>0$ such that
\begin{align*}
    x(t; x_0) \neq x_0 
    \quad \text{for all } t\in (0,\delta_0).
\end{align*}
From this and the definition of $T(x_0)$ in \eqref{20241021-yb-PeriodOfOrbits}, we deduce that $T(x_0)\ge \delta_0$, which contradicts the assumption $T(x_0)=0$. Hence $V(x_0)=0$, and the necessity follows.

\vskip 5pt
\noindent\textit{(ii)} 
We first prove the sufficiency. Assume that the solution $x(\cdot;x_0)$ is not constant and that \eqref{20250329-yb-SourcesOfSolutionsWithInfinitePeriod} holds. To show that $T(x_0)=+\infty$, suppose by contradiction that $T(x_0)<+\infty$. Then
\begin{align*}
    x(kT(x_0);x_0)=x_0 
    \quad \text{for every integer } k .
\end{align*}
Together with \eqref{20250329-yb-SourcesOfSolutionsWithInfinitePeriod}, this implies
\begin{align*}
    x_0=x_k^*=x_l^*
\end{align*}
for some equilibria $x_k^*$ and $x_l^*$ given in \eqref{20250329-yb-SourcesOfSolutionsWithInfinitePeriod}. Consequently, $x(\cdot;x_0)\equiv x_0$, which contradicts the assumption that the solution is not constant. Hence $T(x_0)=+\infty$, proving the sufficiency.

Next we prove the necessity. Assume that $T(x_0)=+\infty$. Then the solution $x(\cdot;x_0)$ cannot be constant; otherwise equation \eqref{20240925-yubiao-Flow} would imply that the velocity is identically zero along the trajectory, and hence $x_0$ would be an equilibrium, contradicting statement (i).

To prove the necessity, it remains to prove \eqref{20250329-yb-SourcesOfSolutionsWithInfinitePeriod}. We only prove the second equality, as the first can be shown similarly. The goal is to show that for any sequence of times tending to $+\infty$, there exists a subsequence whose corresponding trajectory points converge to an unstable equilibrium. Once this is established, the desired equality follows from the discreteness of equilibria (under assumption (A2)) and the connectedness of the limit set $ \cap_{n\geq 1} \overline{ \{ x(t;x_0) ~:~ t \geq n\} }$.

Let $\{t_k\}_{k\ge1}$ be a sequence with $t_k\to+\infty$. Since the trajectory $x(\cdot;x_0)$ remains in the bounded set $\overline{\Omega}$, there exists a subsequence (still denoted by $\{t_k\}$) and a point $y_\infty\in\overline{\Omega}$ such that
\begin{align}\label{20250403-yb-ConvergenceOfTimeSequence}
    \lim_{k\to+\infty} x(t_k;x_0)=y_\infty .
\end{align}
We claim that $y_\infty$ is an equilibrium. By contradiction, suppose that $V(y_\infty)\neq0$. Then there exist $\delta_1>0$ and a small neighborhood $U_1$ of $y_\infty$ such that $U_1$ can be written as the disjoint union of solution segments of equation \eqref{20240925-yubiao-Flow}:
\begin{align*}
    U_1=\bigcup_{j\in J}\{x(t;z_j)\}_{t\in(-\delta_1,\delta_1)},
\end{align*}
where $J$ is an index set and the solution curves are mutually distinct.

Let $H$ be the Hamiltonian given by assumption (A1). By \eqref{ham:representation},
\begin{align*}
    |\nabla H(y_\infty)|=|V(y_\infty)|\neq0 .
\end{align*}
Since $H$ is constant along each solution of equation \eqref{20240925-yubiao-Flow}, we can assume that after possibly shrinking $U_1$,
\begin{itemize}
    \item[] ($\mathcal P$) the function $H$ takes distinct values on different solution curves in $U_1$.
\end{itemize}
Meanwhile, from \eqref{20250403-yb-ConvergenceOfTimeSequence}, there exists $N_1$ such that for all $k\ge N_1$,
\begin{align*}
    x(t_k;x_0)\in U_1, 
    ~ 
    y_\infty\in U_1,
    ~ \text{ and }~
    H(x(t_k;x_0))=H(y_\infty).
\end{align*}
This, together with Property ($\mathcal P$), implies that all points $x(t_k;x_0)$ with $k\ge N_1$ and $y_\infty$ lie on the same solution curve. Hence there exists $j_0\in J$ such that
\begin{align*}
    \{x(t_k;x_0)\}_{k\ge N_1}
    \subset
    \{x(t;z_{j_0})\}_{t\in(-\delta_1,\delta_1)},
    ~ \text{ and }~
    y_\infty\in
    \{x(t;z_{j_0})\}_{t\in(-\delta_1,\delta_1)} .
\end{align*}
Therefore, for each $k\ge N_1$ there exists $\hat t_k\in(-2\delta_1,2\delta_1)$ such that
\begin{align*}
    x(\hat t_k; x(t_k;x_0))=y_\infty .
\end{align*}
In particular, there exist $k_1,k_2\ge N_1$ such that
\begin{align*}
    t_{k_2}+\hat t_{k_2}>t_{k_1}+\hat t_{k_1}
    \quad \text{and}\quad
    x(t_{k_j}+\hat t_{k_j};x_0)=y_\infty,\quad j=1,2 .
\end{align*}
Thus $y_\infty$ is a periodic point of the trajectory $x(\cdot;x_0)$, and hence $x_0$ is also periodic. Therefore there exists $\hat t>0$ such that $x(\hat t;x_0)=x_0$, which implies $T(x_0)<+\infty$ by \eqref{20241021-yb-PeriodOfOrbits}. This contradicts the assumption $T(x_0)=+\infty$. Hence $y_\infty$ must be an equilibrium. 

Now by \eqref{20250403-yb-ConvergenceOfTimeSequence}, a part of the non-constant trajectory $x(\cdot;x_0)$ converges to the equilibirum $y_{\infty}$. It is clear that  in any neighborhood of $y_{\infty}$, this trajectory and $y_{\infty}$ are on a same level set $\{ H = H(y_{\infty})\}$. This can not happen for any stable equilibrium $x_k^*$, since $H(\cdot) - H(x_k^*)$ keeps a strict sign around $x_k^*$ by  \eqref{ham:local-model} and Definition \ref{20250329-yb-ClassificationOfEquilibria}. Thus, this equilibirum $y_{\infty}$ is unstable. So \eqref{20250329-yb-SourcesOfSolutionsWithInfinitePeriod} holds. In conclusion, the necessity is proved.

Finally, note that there are at most four branches issuing from an unstable equilibrium (see \eqref{FourBranchesNearUnstablePoint-20261009}). Thus, there are only finitely many trajectories to equation \eqref{20240925-yubiao-Flow} satisfying \eqref{20250329-yb-SourcesOfSolutionsWithInfinitePeriod}. 
This completes the proof of statement (ii).

\vskip 5pt
\noindent\textit{(iii)} 
Suppose that $0<T(x_0)<+\infty$. By statement (i) and \eqref{ham:representation}, we have
\begin{align}\label{20250329-yb-NonEquilibrium}
    V(x_0)\neq0 
    \text{ and }
    \nabla H(x_0) \neq 0.
\end{align}
Write $\nu(p) := \frac{ \nabla H(p) }{ | \nabla H(p) | }$ for $p$ near $x_0$. Denote by $\Sigma$ the surface $\{ H(x) = H(x_0)\}$ around $x_0$. Without loss of generality, by \eqref{20250329-yb-NonEquilibrium}, assume that $\nabla H \neq 0$ over $\Sigma$. 
 There is a small $s_0 >0$ such that for each $p\in \Sigma$,  the  function $H$ is strictly increasing over the following $C^1$ curve:
\begin{align*}
    q(s; p) := p + s \nu(p),   ~  s \in (-s_0, s_0), 
\end{align*}
and the map $(s,p) \mapsto q(s; p)$ parameterizes a neighborhood of $x_0$. 

Arbitrarily fix a $p\in \Sigma$. We study the returning time of points on $q(\cdot;p)$. For this purpose, define a $C^1$ function $f:\mathbb R\times (-s_0, s_0)$ by
\begin{align*}
    f(t, s):= \big(  x(t; q(s;p)) - p  \big)   \cdot  V(p),
    ~ t\in \mathbb R, ~ s\in   (-s_0, s_0).
\end{align*}
Since $x(\cdot;p)$ solves equation \eqref{20240925-yubiao-Flow}, 
it follows from \eqref{20250329-yb-NonEquilibrium} and \eqref{20241021-yb-PeriodOfOrbits} that
\begin{align*}
    \partial_t f(T(p), 0)
    =
    x'(T(p); p)  \cdot  V(p)
    = 
    |V(p)|^2  \neq  0 .
\end{align*}
Therefore, the implicit function theorem applied to $f$ yields that there is a function $\tau(\cdot;p)$ (with $s\in (-s_1, s_1)$ for some $s_1 \in (0,s_0)$), of class $C^1$ in both $s$ and $p$, such that
\begin{align*}
     f( \tau(s;p) , s ) = \Big(  x\big( \tau(s;p); q(s;p) \big) - p  \Big)   \cdot  V(p)
     =0,  ~  s \in (-s_1, s_1) .
\end{align*}
Since $H\big(  x(\tau(s;p); q(s;p) ) \big) = H(q(s;p))$ for all $s$ and $H$ is strictly increasing over $(-s_1, s_1)$, the above yields 
\begin{align}\label{ReturnTime-20260916}
    x\big( \tau(s;p); q(s;p) \big)  = q(s;p),  ~  s \in (-s_1, s_1). 
\end{align}

Next, we see from \eqref{ReturnTime-20260916} that  for each $s \in (-s_1, s_1)$, $\tau(s;p)$ is a multiple of $T( q(s;p))$.  
We claim that for small $s$,
\begin{align}\label{ReturnTime-Period-20260916}
    \tau(s;p) = T( q(s;p)). 
\end{align}
Otherwise, there is a sequence $\{ s_k \}_{k \geq 1}$ converging to $0$ such that $\tau(s_k;p) \geq 2 T( q(s_k;p))$ for each $k$. Then,  $T(q(s_k;p)) \leq 
\frac{2}{3} \tau(0;p) = \frac{2}{3} T(p)$ for large $k$. At the same time, the lower bound of $T(\cdot)$ around $p$ should be positive. 
Hence, there is a subsequence of $\{ s_k \}_{k \geq 1}$, still denoted by the same manner, and a  $t^* \in (0, \frac{2}{3} T(p) ]$ such that
\begin{align*}
    \lim_{ k\rightarrow +\infty }   T(q(s_k;p))
    = t^*
    ~\text{ and }~
    x\big( T(q(s_k;p));  q(s_k;p) \big) = q(s_k;p),  ~k \in \mathbb N^+.
\end{align*}
It is clear that $x(t^*; p ) = p$, which implies that 
the solution $x(t^*; p)$ returns to $p$ within time $\leq \frac{2}{3} T(p)$, which leads to a contradiction. Therefore, the claim \eqref{ReturnTime-Period-20260916} is true. 

Finally, because the map $(s,p) \mapsto q(s; p)$ parameterizes a neighborhood of $x_0$, 
the desired conclusion follows from \eqref{ReturnTime-Period-20260916}. This completes the proof.

\section*{Acknowledgments}
The authors thank Prof. Enrique Zuazua for his helpful comments.
The first author is partially supported by NSF DMS-2111486, DMS-2205117 and AFOSR FA9550-23-1-0675. This work was initiated during the author's visit to the Chair for Dynamics, Control, Machine Learning and Numerics, Friedrich-Alexander-Universit\"at Erlangen-N\"urnberg, Germany, with support from the Humboldt Research Fellowship for Experienced Researchers program of the Alexander von Humboldt Foundation.
The third author is supported by the National Natural Science Foundation of China under grant 12671535 and the Humboldt Research Fellowship for Experienced Researchers program from the Alexander von Humboldt Foundation.

\bigskip\bigskip
\noindent\textbf{Data Availability Statement} Data sharing is not applicable to this paper because no
dataset was analysed or generated  during the study.

\bigskip
\noindent\textbf{Statements and Declarations}

\noindent\textbf{Competing Interests} The authors declare that  they have no financial interests nor any conflict of interest.

\bibliographystyle{abbrv}
\bibliography{references.bib}

\begin{thebibliography}{10}

\bibitem{alberti2016exponential}
G.~Alberti, G.~Crippa, and A.~L. Mazzucato.
\newblock Exponential self-similar mixing by incompressible flows.
\newblock {\em J. Amer. Math. Soc.}, 32(2):445--490, 2019.

\bibitem{manu2016how}
M.~Aminian, F.~Bernardi, R.~Camassa, D.~Harris, and R.~McLaughlin.
\newblock How boundaries shape chemical delivery in microfluidics.
\newblock {\em Science}, 354(6317):1252--1256, 2016.

\bibitem{beebe2002biology}
D.~J. Beebe, G.~A. Mensing, and G.~Walker.
\newblock Physics and applications of microfluidics in biology.
\newblock {\em Annu. Rev. Biomed. Eng.}, 4:261--286, 2002.

\bibitem{bonicatto2021regularity}
P.~Bonicatto and E.~Marconi.
\newblock Regularity estimates for the flow of {BV} autonomous divergence-free vector fields in {$\mathbb{R}^2$}.
\newblock {\em Comm. Partial Differential Equations}, 46(12):2235--2267, 2021.

\bibitem{Bressan-2006}
A.~Bressan.
\newblock A lemma and a conjecture on the cost of rearrangements.
\newblock {\em Rend. Sem. Mat. Univ. Padova}, 110:97--102, 2003.

\bibitem{brue2024enhanced}
E.~Bru\`e, M.~Coti~Zelati, and E.~Marconi.
\newblock Enhanced dissipation for two-dimensional {H}amiltonian flows.
\newblock {\em Arch. Ration. Mech. Anal.}, 248(5):Paper No. 84, 37, 2024.

\bibitem{chakravarthy1996mixing}
V.~S. Chakravarthy and J.~M. Ottino.
\newblock Mixing of two viscous fluids in a rectangular cavity.
\newblock {\em Chem. Eng. Sci.}, 51(14):3613--3622, 1996.

\bibitem{zelati2024mixing}
M.~Coti~Zelati, G.~Crippa, G.~Iyer, and A.~L. Mazzucato.
\newblock Mixing in incompressible flows: transport, dissipation, and their interplay.
\newblock {\em Notices Amer. Math. Soc.}, 71(5):593--604, 2024.

\bibitem{crippa2008estimates}
G.~Crippa and C.~De~Lellis.
\newblock Estimates and regularity results for the {D}i{P}erna-{L}ions flow.
\newblock {\em J. Reine Angew. Math.}, 616:15--46, 2008.

\bibitem{crippa2019polynomial}
G.~Crippa, R.~Luc\`a, and C.~Schulze.
\newblock Polynomial mixing under a certain stationary {E}uler flow.
\newblock {\em Phys. D}, 394:44--55, 2019.

\bibitem{d1999control}
D.~D'Alessandro, M.~Dahleh, and I.~Mezi\'c.
\newblock Control of mixing in fluid flow: a maximum entropy approach.
\newblock {\em IEEE Trans. Automat. Control}, 44(10):1852--1863, 1999.

\bibitem{dombre1986chaotic}
T.~Dombre, U.~Frisch, J.~M. Greene, M.~H\'{e}non, A.~Mehr, and A.~M. Soward.
\newblock Chaotic streamlines in the {ABC} flows.
\newblock {\em J. Fluid Mech.}, 167:353--391, 1986.

\bibitem{doswell1984kinematic}
C.~A. Doswell~III.
\newblock A kinematic analysis of frontogenesis associated with a nondivergent vortex.
\newblock {\em J. Atmos. Sci.}, 41(7):1242--1248, 1984.

\bibitem{elgindi2019universal}
T.~M. Elgindi and A.~Zlato{\v s}.
\newblock Universal mixers in all dimensions.
\newblock {\em Adv. Math.}, 356:106807, 33, 2019.

\bibitem{foures2014optimal}
D.~P.~G. Foures, C.~P. Caulfield, and P.~J. Schmid.
\newblock Optimal mixing in two-dimensional plane {P}oiseuille flow at finite {P}\'eclet number.
\newblock {\em J. Fluid Mech.}, 748:241--277, 2014.

\bibitem{hu2023feedback}
W.~Hu, C.~N. Rautenberg, and X.~Zheng.
\newblock Feedback control for fluid mixing via advection.
\newblock {\em J. Differential Equations}, 374:126--153, 2023.

\bibitem{hu2018boundaryNS}
W.~Hu and J.~Wu.
\newblock Boundary control for optimal mixing via {N}avier-{S}tokes flows.
\newblock {\em SIAM J. Control Optim.}, 56(4):2768--2801, 2018.

\bibitem{iyer2014lower}
G.~Iyer, A.~Kiselev, and X.~Xu.
\newblock Lower bounds on the mix norm of passive scalars advected by incompressible enstrophy-constrained flows.
\newblock {\em Nonlinearity}, 27(5):973--985, 2014.

\bibitem{Kiselev2008}
A.~Kiselev, R.~Shterenberg, and A.~Zlato{\v{s}}.
\newblock Relaxation enhancement by time-periodic flows.
\newblock {\em Indiana Univ. Math. J.}, 57(5):2137--2152, 2008.

\bibitem{li2026hamiltonian}
Z.~Li and E.~Zuazua.
\newblock Hamiltonian interface dynamics for reduced-order optimization of incompressible mixing.
\newblock {\em arXiv:2605.04688}, 2026.

\bibitem{lin2011optimal}
Z.~Lin, J.-L. Thiffeault, and C.~R. Doering.
\newblock Optimal stirring strategies for passive scalar mixing.
\newblock {\em J. Fluid Mech.}, 675:465--476, 2011.

\bibitem{liu2008mixing}
W.~Liu.
\newblock Mixing enhancement by optimal flow advection.
\newblock {\em SIAM J. Control Optim.}, 47(2):624--638, 2008.

\bibitem{mathew2007optimal}
G.~Mathew, I.~Mezi{\' c}, S.~Grivopoulos, U.~Vaidya, and L.~Petzold.
\newblock Optimal control of mixing in {S}tokes fluid flows.
\newblock {\em J. Fluid Mech.}, 580:261--281, 2007.

\bibitem{mathew2005multiscale}
G.~Mathew, I.~Mezi\'c, and L.~Petzold.
\newblock A multiscale measure for mixing.
\newblock {\em Phys. D}, 211(1-2):23--46, 2005.

\bibitem{seis2013maximal}
C.~Seis.
\newblock Maximal mixing by incompressible fluid flows.
\newblock {\em Nonlinearity}, 26(12):3279--3289, 2013.

\bibitem{abraham2002chaotic}
A.~D. Stroock, S.~K.~W. Dertinger, A.~Ajdari, I.~Mezić, H.~A. Stone, and G.~M. Whitesides.
\newblock Chaotic mixer for microchannels.
\newblock {\em Science}, 295(5555):647--651, 2002.

\bibitem{thiffeault2012multiscale}
J.-L. Thiffeault.
\newblock Using multiscale norms to quantify mixing and transport.
\newblock {\em Nonlinearity}, 25(2):R1--R44, 2012.

\bibitem{vikhansky2002enhancement}
A.~Vikhansky.
\newblock Enhancement of laminar mixing by optimal control methods.
\newblock {\em Chem. Eng. Sci.}, 57(14):2719--2725, 2002.

\bibitem{yao2017mixing}
Y.~Yao and A.~Zlato{\v s}.
\newblock Mixing and un-mixing by incompressible flows.
\newblock {\em J. Eur. Math. Soc. (JEMS)}, 19(7):1911--1948, 2017.

\bibitem{zheng2023numerical}
X.~Zheng, W.~Hu, and J.~Wu.
\newblock Numerical algorithms and simulations of boundary dynamic control for optimal mixing in unsteady {S}tokes flows.
\newblock {\em Comput. Methods Appl. Mech. Engrg.}, 417:Paper No. 116455, 24, 2023.

\bibitem{zillinger2019scales}
C.~Zillinger.
\newblock On geometric and analytic mixing scales: comparability and convergence rates for transport problems.
\newblock {\em Pure Appl. Anal.}, 1(4):543--570, 2019.

\end{thebibliography}

\end{document}